\documentclass[preprint,12pt]{elsarticle}

\usepackage{graphics}

\usepackage{amssymb}
\usepackage{amsthm}
\usepackage{amsmath}
\usepackage{mathabx}
\usepackage{color}
\usepackage[latin1]{inputenc} 
\usepackage{amsopn}
\usepackage{amsmath}
\usepackage{mathabx}
\usepackage{pstricks}
\usepackage{color}
\usepackage{pstcol}
\usepackage{pst-plot}
\usepackage{pst-tree}
\usepackage{pst-eps}
\usepackage{multido}
\usepackage{pst-node}
\usepackage{pst-3d}
\usepackage{colordvi}
\usepackage{multirow}
\usepackage{pst-grad}
\usepackage{colortbl}
\usepackage{graphics}

\DeclareMathOperator{\interior}{Int}

\newtheorem{theorem}{Theorem}[section]

\newtheorem{corollary}[theorem]{Corollary}
\newtheorem{definition}[theorem]{Definition}

\newtheorem{lemma}[theorem]{Lemma}
\newtheorem{proposition}[theorem]{Proposition}
\newtheorem{remark}[theorem]{Remark}

\journal{ }

\begin{document}
\nocite{*}

\renewcommand{\listtablename}{\'indice de tablas}

\begin{frontmatter}



\title{A fully data-driven method for estimating density level sets}

\author[rvt]{A. Rodr\'iguez-Casal}
\author[rvt]{P. Saavedra-Nieves \corref{cor1}}
\address[rvt]{Department of Statistics and Operations Research, University of Santiago de Compostela, Spain}
\cortext[cor1]{Corresponding author: paula.saavedra@usc.es (P. Saavedra-Nieves)}

\begin{abstract}Density level sets can be estimated using plug-in methods, excess mass algorithms or a hybrid of the two previous methodologies. The plug-in algorithms are based on replacing the unknown density by some nonparametric estimator, usually the kernel. Thus, the bandwidth selection is a fundamental problem
from an applied perspective. However, if some a priori information about the geometry of the level set is available, then excess mass algorithms
could be useful. Hybrid methods such that granulometric smoothing algorithm assume a mild geometric restriction on the level set and it requires a
pilot nonparametric estimator of the density. In this work, a new hybrid algorithm is proposed under the assumption that the level set is $r-$convex. The main problem in practice is that $r$ is an unknown geometric characteristic of the set. A stochastic algorithm is proposed
for selecting its optimal value. The resulting data-driven reconstruction of the level set is able to
achieve the same convergence rates as the granulometric smoothing method. However, they do no depend on any penalty term because, although the value of the shape index $r$ is a priori unknown, it is estimated in a data-driven way from the sample points. The practical performance of the
estimator proposed is illustrated through a real data example.
\end{abstract}

\begin{keyword}Density level set  \sep $r-$convexity \sep hybrid methodology \sep leukaemia clustering
\end{keyword}

\end{frontmatter}

\section{Introduction}\label{introduction}

Level set estimation theory deals with the problem of reconstructing an unknown set $G(t)=\{f \geq t\}$ from a random sample of points $\mathcal{X}_n=\{X_1,...,X_n\}$ of a random variable $X$, where $f$ is the density of $X$, and $t$ denotes a positive threshold. Since Hartigan (1975) introduced the notion of
population clusters as the connected components of density level sets, many interesting works have been published (see, for instance, Mammen and Polonik, 2014 or Steinwart, 2014) and many applications have appeared. The concept of clustering is related to the notion of the mode
and, in fact, some clustering algorithms are based on the estimation of modes (see Cuevas,
Febrero and Fraiman, 2000). An interesting application of this clustering approach to
astronomical sky surveys was proposed by Jang (2006) and Klemel{\"a} (2004, 2006) using
a similar point of view to develop methods for visualizing multivariate density estimates.
Goldenshluger and Zeevi (2004) used level set estimation in the context of the Hough
transform, which is a well-known computer vision algorithm. Some problems in flow
cytometry involve the statistical problem of reconstructing a level set for the difference between two probability densities (see Roederer and Hardy, 2001). In addition, interesting
applications include the detection of mine fields based on aerial observations, the analysis
of seismic data, as well as certain issues in image segmentation (see Huo and Lu, 2004).
The detection of outliers is another important application of level set estimation (see
Gardner et al., 2006 or Markou and Singh, 2003 for a review). An outlier can be thought
of as an observation that does not
belong to the effective support determined by the level set. This approach follows
that of Devroye and Wise (1980) to determine whether a manufacturing process is out of
control. For quality control schemes, see also Ba\'illo, Cuevas and Justel (2000) or Ba\'illo
and Cuevas (2006).

There are three methodologies in literature for estimating level sets: Plug-in, excess mass and hybrid. The choice of an algorithm depends on the geometric assumptions made on the shape of the level set:\vspace{.25cm}\\
The \emph{plug-in estimation} is the most natural choice to estimate $G(t)$ when no geometric information about the level set is available. The estimator proposed is $\hat{G}(t)=\{f_n\geq t\}$ where $f_n$ is a nonparametric estimator for the density function. Usually, $f_n$ denotes the kernel estimator. Given $\mathcal{X}_n$, the kernel density estimator at point $x$ is defined as\vspace{.1cm}
\begin{equation}
\label{estimacionnucleo2}f_n(x)=\frac{1}{n}\sum_{i=1}^n K_H\left(x-X_i\right),\vspace{1.8mm}
\end{equation}where $K_H(z)=|H|^{-1/2}K(H^{-1/2}z)$, $ |\mbox{ }|$ represents the determinant, $K:\mathbb{R}^d\rightarrow \mathbb{R}$ denotes a kernel function (in what follows the Gaussian density) and $H$, a $(d\times d)-$dimensional symmetric positive definite matrix. The estimator defined in (\ref{estimacionnucleo2}) is heavily dependent on the matrix $H$, see Wand and Jones (1995). Therefore, the practical problem of the plug-in methodology is the choice of this matrix. Unlike density estimation, the level set estimation has been considered in literature from many points of view but, in general, without deepening in methods for selecting $H$. In fact, this problem was first considered by Ba\'illo and Cuevas (2006) in the context of nonparametric statistical quality control. Singh et al. (2009) presented a plug-in procedure that is based on an empirical density estimator, the regular histogram. Later, Samworth and Wand (2010) derive an automatic bandwidth selection rule
to estimate density level sets but only in the one-dimensional case.\vspace{.25cm}\\
The \emph{excess mass estimation} assumes that the researcher has information a priori about the shape of the level set $G(t)$. This methodology was first proposed by Hartigan (1987) and M\"{u}ller and Sawitzki
(1987). Then, Polonik (1995) extended and investigated it in a very general framework. These algorithms are based on a quite simple idea: The set $G(t)$ maximizes the functional
$$H_t(B)=\mathbb{P}(B)-t\mu(B),\vspace{.5mm}\label{Ht}$$on the Borel sets $B$ where $\mathbb{P}$ denotes the probability measure induced by $f$ and $\mu\label{lebesgue2}$, the Lebesgue measure. In addition, $H_t$ can be estimated empirically. So, if $G(t)$ is assumed to belong to a family of sets then it could be reconstructed by maximizing the empirical version of the previous functional on the family considered. Consequently, unlike the plug-in approximation, excess mass methods do not need to smooth the sample $\mathcal{X}_n$ and, in addition, they impose geometric restrictions on the estimators.\vspace{.25cm}\\
The last and third methodology is a \emph{hybrid} of the two previous ones. Just as the excess mass methods, the hybrid methodology assumes some shape restrictions on the class of sets considered and, like the plug-in methods, it needs to smooth the data set. Walther (1997) proposed the granulometric smoothing method to reconstruct a level set assuming that it and its complement are both $r-$convex. Saavedra-Nieves et al. (2014) presented two new hybrid methods for estimating convex and $r-$convex sets. A closed set $A\subset\mathbb{R}^d$ is said to be $r-$convex, for some $r>0$, if $A=C_{r}(A)$, where
$$C_{r}(A)=\bigcap_{\{B_r(x):B_r(x)\cap
A=\emptyset\}}\left(B_r(x)\right)^c$$
denotes the $r-$convex hull of $A$ and $B_r(x)$, the open ball with
center $x$ and radius $r$. The $r-$convex hull is closely related to the closing of $A$ by $B_r(0)$ from the mathematical morphology, see Serra (1982). It can be shown that
$$
C_{r}(A)=(A\oplus r B_1(0))\ominus r B_1(0),
$$
where $\lambda C=\{\lambda c: c\in C\}$, $C\oplus D=\{c+d:\ c\in C, d\in D\}$ and $C\ominus D=\{x\in\mathbb{R}^d:\ \{x\}\oplus D\subset C\}$, for $\lambda \in \mathbb{R}$ and
sets $C$ and $D$.

According to the simulation results presented in Saavedra-Nieves et al. (2014), the performance of the $r-$convex hull algorithm is quite promising. This paper is focused on proposing a data-driven method for reconstructing density level set under the flexible assumption of $r-$convexity basing on this initial proposal. The main disadvantage of the $r-$convex hull method proposed in Saavedra-Nieves et al. (2014) is the that the parameter $r$ is unknown. An automatic selection criterion will be proposed
in this work. Once the parameter $r$ is estimated, it is natural to propose a resulting density level set estimator based on the estimator of $r$. Two metrics between sets are usually considered in order to assess the performance of a set estimator. Let $A$ and $C$ be two closed, bounded, nonempty subsets of $\mathbb{R}^{d}$. The
Hausdorff distance between $A$ and $C$ is defined by\vspace{-0.13cm}
$$
d_{H}(A,C)=\max\left\{\sup_{a\in A}d(a,C),\sup_{c\in C}d(c,A)\right\},\vspace{-0.13cm}
$$where $d(a,C)=\inf\{\|a-c\|:c\in C\}$ and $\|\mbox{ }\|$ denotes the Euclidean norm. On the other hand, if $A$ and $C$
are two bounded and Borel sets then the distance in measure between $A$ and $C$ is defined by $d_{\mu}(A,C)=\mu(A\triangle C)$, where $\mu$ denotes the Lebesgue measure and $\triangle$, the symmetric difference, that is, $A\triangle C=(A \setminus C)\cup(C \setminus A). $
Hausdorff distance quantifies the physical proximity between two sets whereas the distance in measure
is useful to quantify their similarity in content. However, neither of these distances are completely useful for measuring
the similarity between the shape of two sets. The Hausdorff distance between boundaries, $d_H(\partial A,\partial C)$, can be also used to evaluate the
performance of the estimators, see Ba\'illo and Cuevas (2001), Cuevas and Rodr\'iguez-Casal (2004) or Rodr\'iguez-Casal (2007).

This paper is organized as follows. The optimal parameter is defined and an estimator for it is established in Section \ref{fg}. The consistency of this new estimator is proved in Section \ref{mainresults}. In addition, the resulting density level set estimator is presented and its consistency and convergence rates will be showed too. A real data example is presented in Section \ref{gorr}. The performance of the new algorithm will be illustrated by comparing the distribution of controls and cases in a leukaemia data set. Proofs are deferred to Section \ref{prprpr}. Finally, a serie of useful theoretical results contained in Walther (1997) are showed in Section \ref{apendix}.

\section{Selection of the optimal parameter}\label{fg}

According to the previous comments, the first step is to determinate the optimal value of the smoothing parameter to be estimated. If a set $A$ is $r-$convex then $A$ is $r^*-$convex for all $r^*$ verifying $0<r^*\leq r$. Therefore, we are interested in estimating the greatest value of $r>0$ such that $G(t)$ is $r-$convex.

\begin{definition}\label{r_0_t}Let $G(t)$ be a compact, nonempty, nonconvex and $r-$convex level set for some $r>0\label{r0t}$. It is defined
\begin{equation}\label{estimadorr0levelsetu}
    r_0(t)=\sup\{\gamma>0:C_\gamma(G(t))=G(t)\}.
\end{equation}
\end{definition}
For simplicity in the exposition, we have assumed that $G(t)$ is not convex in order to guarantee that the set $\{\gamma>0:C_\gamma(G(t))=G(t)\}$ is upper bounded. Of course, if $G(t)$ is convex $r_0(t) $ would be infinity. Notice that, in this case, the parameter depends on the level $t>0$ considered, see Figure \ref{oosssodfssfs}.

The following geometric property has been assumed on the level set $G(t)$:\vspace{4mm}\\
($R_{\lambda}^r$) A closed ball of radius $\lambda>0$ rolls freely in $G(t)$ and a closed ball of \textcolor[rgb]{1.00,1.00,1.00}{($R_{\lambda}^r$)} radius $r>0$ rolls freely in $\overline{G(t)^c}$.\vspace{.3mm}\\

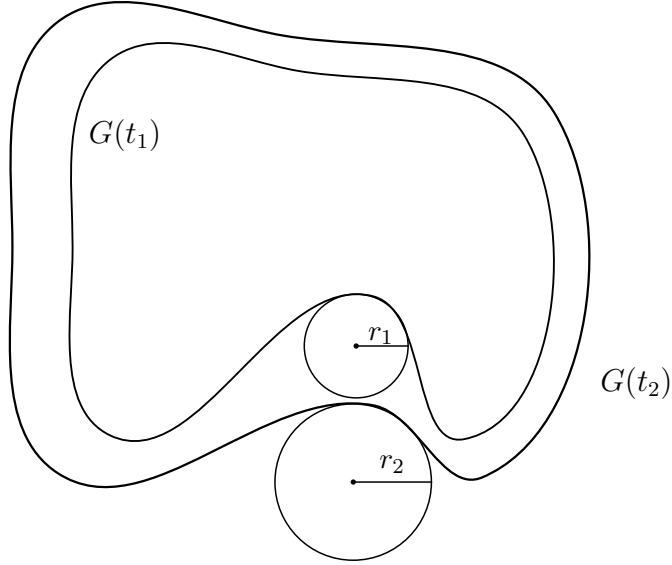
\begin{figure*}\vspace{-1.cm}
\hspace{10cm}\begin{pspicture}(7.3,-2.3)(10,6.4)

\rput(-1.,-0.5){\scalebox{1.2}[1.2]{\psccurve[showpoints=false,fillstyle=solid,fillcolor=white,linecolor=black,linewidth=.25mm,linearc=3](1.5,0)(1,2.5)(1.5,5)(4,4.85)(7,4)(6.2,-0.05)(5,.75)
 }}
 \psccurve[showpoints=false,fillstyle=solid,fillcolor=white,linecolor=black,linewidth=.25mm,linearc=3](1.5,0)(1,2.5)(1.5,5)(4,4.85)(7,4)(6.2,-0.05)(5,1.85)
  \rput(1.7,4){$G(t_1)$}
\rput(8.5,.7){$G(t_2)$}
  \pscircle[linearc=0.25,linecolor=black,linewidth=0.2mm,linestyle=solid,dash=3pt 2pt](4.77,1.19){.7}
  \psdots[dotsize=2pt](4.77,1.19)
  \psline[linearc=0.25,linecolor=black,linewidth=0.2mm,linestyle=solid,dash=3pt 2pt](4.77,1.19)(5.47,1.19)
\rput(5.1,1.33){\small{$r_1$}}
\pscircle[linearc=0.25,linecolor=black,linewidth=0.2mm,linestyle=solid,dash=3pt 2pt](4.73,-.62){1.05}
\psdots[dotsize=2pt](4.73,-.62)
  \psline[linearc=0.25,linecolor=black,linewidth=0.2mm,linestyle=solid,dash=3pt 2pt](4.73,-.62)(5.78,-.62)
  \rput(5.25,-.38){\small{$r_2$}}
  $ $
\end{pspicture}\caption{Level sets $G(t_i),$ with $i=1,2$ and $t_1$ greater than $t_2$. In addition, $C_{r_i}(G(t_i))=G(t_i)$ with $r_i$ denoting $r_0(t_i)$ for $i=1,2$. In this case, $r_2>r_1$.}\label{oosssodfssfs}\vspace{-0.25cm}
\end{figure*}

It is said that $A$ satisfies the $r-$rolling condition if each boundary point $a\in\partial A$ there exists a closed ball with radius $r$, $B_r[x]$, such that $a\in B_r[x]$ and $B_r[x]\subset A$. The intuitive concept of rolling freely can be seen as a sort of geometric smoothness statement. There exist interesting relationships between this property and $r-$convexity. In particular, Cuevas et al. (2012) proved that if $A$ is compact and $r-$convex then $\overline{A}^c$ fulfills the $r-$rolling condition. However, the reciprocal is always not true. For a in depth analysis of these two shape restrictions see Walther (1997, 1999). Satisfying the shape condition ($R_{\lambda}^r$) is a quite natural general property for level sets of densities. In fact, in Theorem 2 by Walther (1997) was proved that, under some assumptions on the density $f$, its level sets satisfy the conditions in Theorem 1 by Walther (1997) for $r=\lambda=m/k$, see below. Then, according to Theorem 2 in Walther (1997), the following assumptions are considered on $f$:\vspace{.25cm}

\begin{description}
  \item[A. ]\begin{enumerate}\item The threshold $t$ of $G(t)$ belongs to $(l,u)$ with $-\infty<l\leq u<\sup (f)$.
                                \item $f\in\mathcal{C}^p(U)$, $p\geq 1$ where $U$ is a bounded open set containing $\overline{G(l-\zeta)}\setminus \interior (G(u+\zeta))$ for some $\zeta>0$  where $G(u+\zeta)$ is bounded, see Figure \ref{oookffghhffffffffffffffffffffffhg1o2lpoi}.
                    \item The gradient of $f$, $\nabla f$, satisfies $|\nabla f|\geq  m  >0$ as well as Lipschitz condition on $U$:
      $$|\nabla f(x)-\nabla f(y)|\leq k|x-y|\mbox{ for }x,y\in U.$$
                  \end{enumerate}
\end{description}

\begin{figure}[h!]\centering\vspace{-.2cm}
\begin{pspicture}(-2,-1.95)(10,7)

\rput(-1.58,-1.){\scalebox{1.3}[1.4]{\psccurve[showpoints=false,fillstyle=solid,fillcolor=white,linecolor=black,linewidth=0.15mm,linearc=3,linestyle=dashed,dash=2pt 2pt](1.5,0)(1.9,2.5)(1.5,5)(4,4.7)(7,4)(6.2,-0.01)(5,0.6)}}


\psccurve[showpoints=false,fillstyle=solid,fillcolor=white,linecolor=gray,linewidth=0.25mm,linearc=3](1.5,0)(1.9,2.5)(1.5,5)(4,4.35)(7,4)(6.2,.3)(5,1.5)


\psccurve[showpoints=false,fillstyle=solid,fillcolor=white,linecolor=gray,linewidth=0.25mm,linearc=3](5.8,1.3)(6.6,3.5)(6.15,0.9)

\rput(1.93,0.61){\scalebox{0.7}[0.75]{\psccurve[showpoints=false,fillstyle=solid,fillcolor=white,linecolor=black,linewidth=0.25mm,linearc=3,linestyle=dashed,dash=3pt 2pt](5.8,1.3)(6.6,3.5)(6.15,0.9) }}
\psccurve[showpoints=false,fillstyle=solid,fillcolor=white,linecolor=gray,linewidth=0.25mm,linearc=3](2,0.8)(2.9,2)(2.55,0.8)

\rput(0.9,0.5){\scalebox{0.65}[0.65]{\psccurve[showpoints=false,fillstyle=solid,fillcolor=white,linecolor=black,linewidth=0.25mm,linearc=3,linestyle=dashed,dash=3pt 2pt](2,0.8)(2.9,2)(2.55,0.8)  }}

\rput(5.1,3){$G(u+\zeta)$}
\rput(1.6,5.5){$G(l-\zeta)$}
\rput(7.8,5.00){$U$}
\end{pspicture}\vspace{0 cm}\caption{$G(u+\zeta)$ and $G(l-\zeta)$ in gray. The open set $U$ in dashed line.}\label{oookffghhffffffffffffffffffffffhg1o2lpoi}
$ $\\
\end{figure}
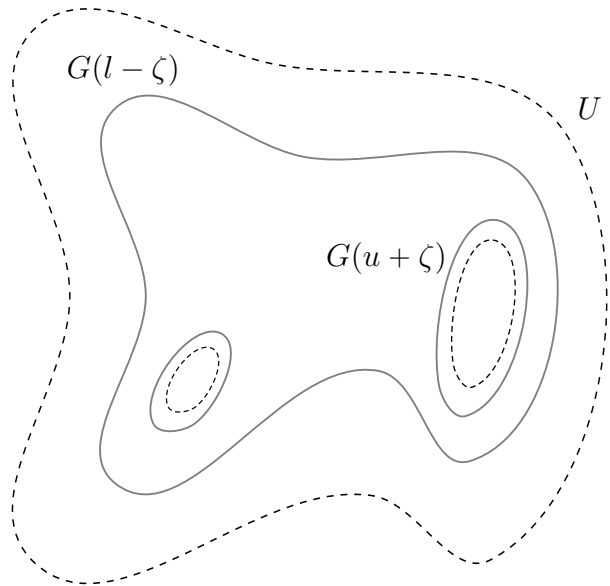

Under (A), it is verified that $r_0(t)\geq m/k$. In addition, the consideration of the shape condition ($R_{\lambda}^r$) has allowed us to guarantee the $r-$convexity of the level set $G(t)$, see Proposition \ref{rconvexo2}. The proof of this result is not showed here since it is a direct consequence of Proposition 2.3 in Rodr\'iguez-Casal and Saavedra-Nieves (2014). 

\begin{proposition}\label{rconvexo2}Let $G(t)\subset \mathbb{R}^{d}$ be a compact and nonempty level set verifying ($R_{\lambda}^r$). Then, $G(t)$ is $r-$convex.\end{proposition}

\subsection{Defining the estimator for the optimal parameter}\label{ty}
The method of the $r-$convex hull proposed in Saavedra-Nieves et al. (2014) divides the original sample $\mathcal{X}_n$ into two subsamples, $\{X_i\in\mathcal{X}_n:f_n(X_i)\geq t\}$ and $\{X_i\in\mathcal{X}_n:f_n(X_i)< t\}$. The estimator for the density level set is constructed as the $r-$convex hull of the sample points where the density estimator is greater than or equal to the threshold. Therefore, it takes into account the information contained only in one of the two subsamples. Then, the information about the complement of the level set $G(t)$ is not taken advantage. Our proposal here will solve this problem by modifying slightly the original algorithm. First, an estimator for the parameter defined in (\ref{estimadorr0levelsetu}) will be proposed. Its definition depends on a sequence $D_n$ satisfying the assumption:
\begin{description}
   \item[D. ]$D_n$ is equal to $M(\log{n}/n)^{p/(d+2p)}$ for a big enough value of the constant $M>0$.\label{Dn}
\end{description}

\begin{definition}\label{jejejeje}Let $G(t)$ be a compact, nonempty and nonconvex level set. Under assumptions (A) and (D), let $\mathcal{X}_n$ be a random sample generated from a distribution with density function $f$. An estimator for the parameter defined in (\ref{estimadorr0levelsetu}) can be defined as
\begin{equation}\label{estimadorr0hat}
    \hat{r}_0(t)=\sup\{\gamma>0:C_\gamma(\mathcal{X}_n^+(t))\cap \mathcal{X}_n^-(t)=\emptyset\},
\end{equation}where
$$
   \mathcal{X}_n^+(t)=\{X\in\mathcal{X}_n: f_n(X)\geq t+D_n\}\mbox{ and }\mathcal{X}_n^-(t)=\{X\in\mathcal{X}_n: f_n(X)< t-D_n\}.
\label{mas2}$$
\end{definition}

Here, the original sample $\mathcal{X}_n$ is divided into three subsamples, $\mathcal{X}_n^+(t)$, $\mathcal{X}_n^-(t)$ and $\mathcal{X}_n\setminus\left(\mathcal{X}_n^+(t)\cup\mathcal{X}_n^-(t)\right)$. From an intuitive point of view, $ \mathcal{X}_n^+(t)$ and $\mathcal{X}_n^-(t)$ should be contained in $G(t)$ and its complementary, respectively. This property is proved in Proposition \ref{mostramaisenGlambda}, even for convex sets. According to Definition \ref{jejejeje}, we have assumed that $G(t)$ is not convex only for simplicity in the exposition. If $G(t)$ is convex then $\hat{r}_0(t)=\infty$ and, therefore, the convex hull of sample points, $\mbox{conv}(\mathcal{X}_n^+(t))$, would reconstruct the level set $G(t)$. In addition, Proposition \ref{alberto5} ensures that $\mathcal{X}_n^+(t)\neq\emptyset$. If $G(t)$ is nonconvex then it can be seen that, with probability one and for $n$ large enough, the set $\{\gamma>0: C_\gamma(\mathcal{X}_n^+(t))\cap \mathcal{X}_n^-(t)=\emptyset\}$ is nonempty and upper bounded. So, the estimator proposed in (\ref{estimadorr0hat}) is well-defined. In order to guarantee that the estimator satisfies these interesting and natural properties, two conditions on the kernel estimator $f_n$ of $f$ must be considered, see again Walther (1997) for more details:
\begin{description}
       \item[K. ]\begin{enumerate}
                    \item The kernel function $K$ is a continuous kernel of order at least $p$ with bounded support and finite variation.
                    \item The bandwidth parameter is of the order $(\log{n}/n)^{1/(d+2p)}$.
                  \end{enumerate}
\end{description}

\newpage

\begin{figure}[h!]\vspace{-2cm}
\hspace{1cm}\begin{picture}(-30,220)
\centering
\put(170,6.5){\includegraphics[height=7.9cm, width=0.43\textwidth]{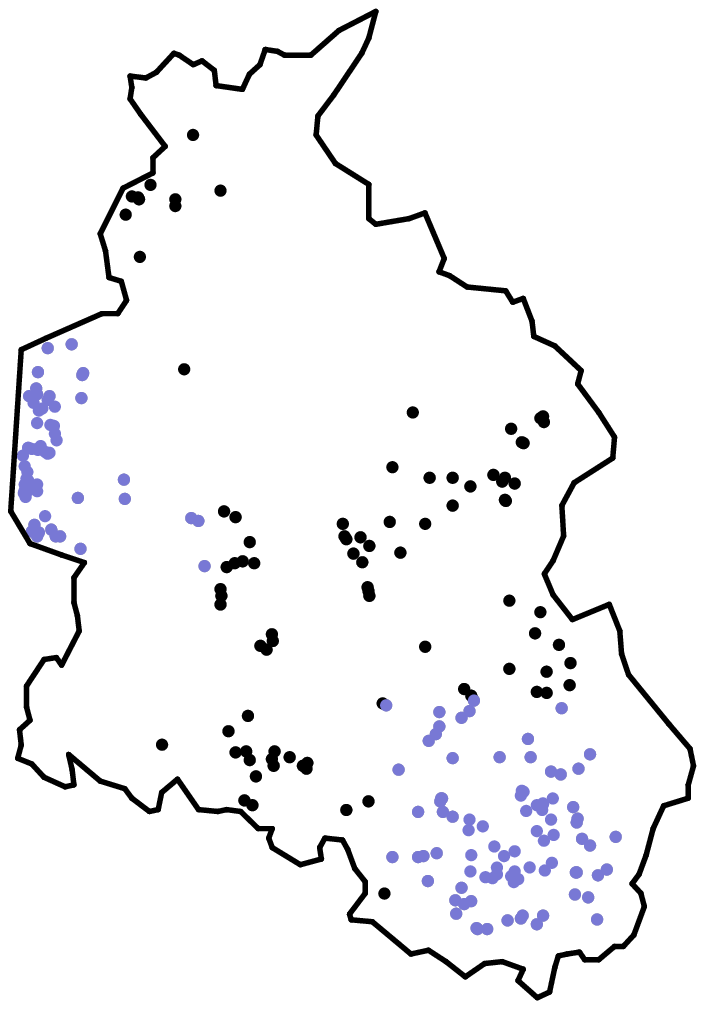}}
\put(1,6.5){\includegraphics[height=7.9cm, width=0.43\textwidth]{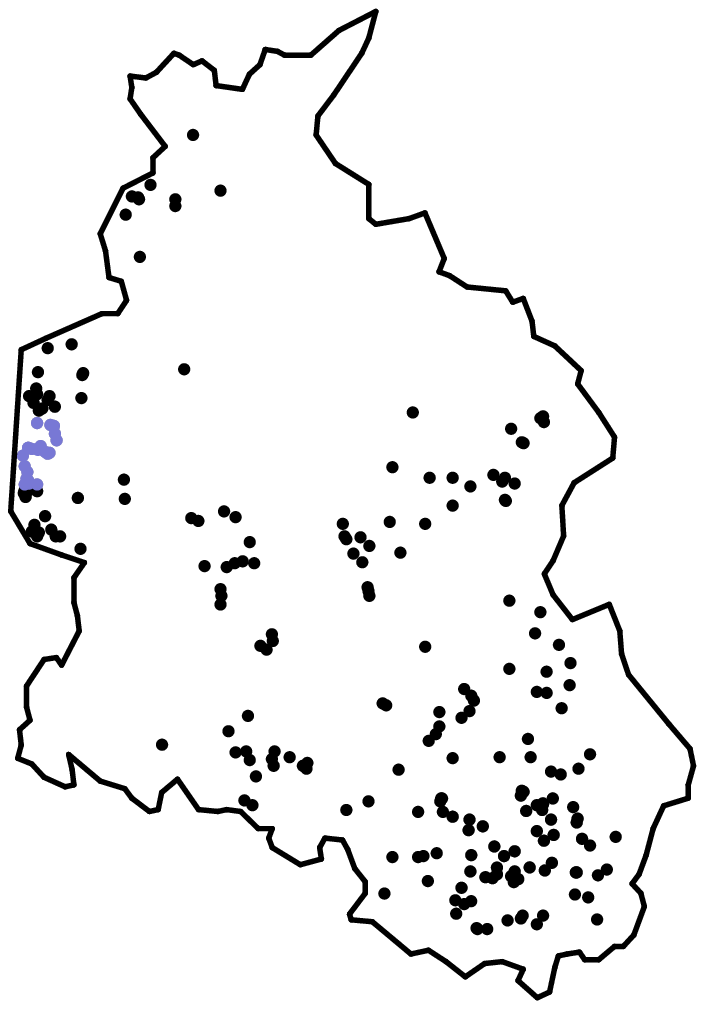}}
  \end{picture}\vspace{-1.7cm}\caption{The set $\mathcal{X}_{n}^+(t_i)$ is represented in blue for $t=t_i,\mbox{ }i=1,2$ with $t_1$ (left) greater than $t_2$ (right). }\label{rty}
\end{figure}

$ $

$\vspace{0.1cm} $

\begin{figure}[h!]
\hspace{1cm}\begin{picture}(-30,330)
\centering
\pscircle[fillcolor=white,linecolor=black,linewidth=0.2mm ](-0.2,4.745){.565}
\pscircle[fillcolor=white,linecolor=black,linewidth=0.2mm ](2.45,9.5){1.7}
\psline[fillcolor=white,linecolor=black,linewidth=0.2mm ](0.05,5.23)(1.55,8.08)
\psline[fillcolor=white,linecolor=black,linewidth=0.2mm ](1.5,10.88)(1.41,8.87)
\psline[fillcolor=white,linecolor=black,linewidth=0.2mm ](1.41,8.87)(1.7,8.3)
\psline[fillcolor=white,linecolor=black,linewidth=0.2mm ](1.7,8.3)(2.4,7.8)
\put(12,198){\includegraphics[height=5 cm, width=0.285\textwidth]{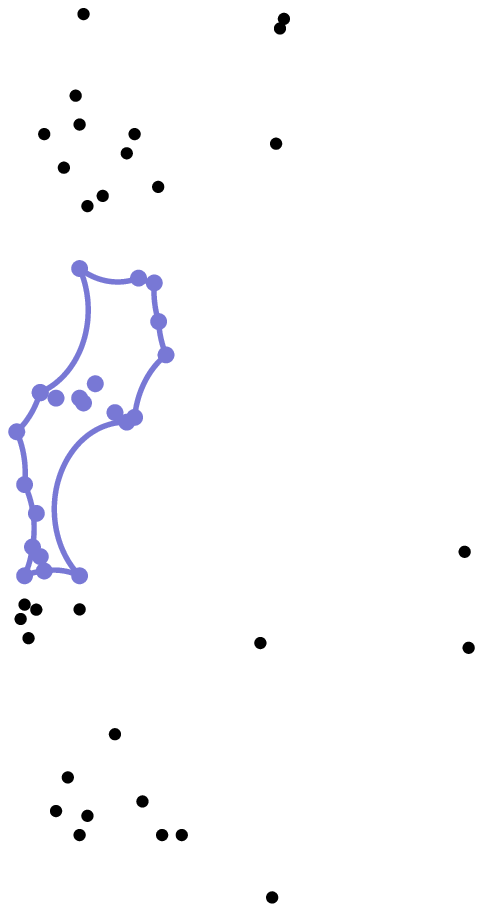}}
\put(-60,6.5){\includegraphics[height=7.9cm, width=0.43\textwidth]{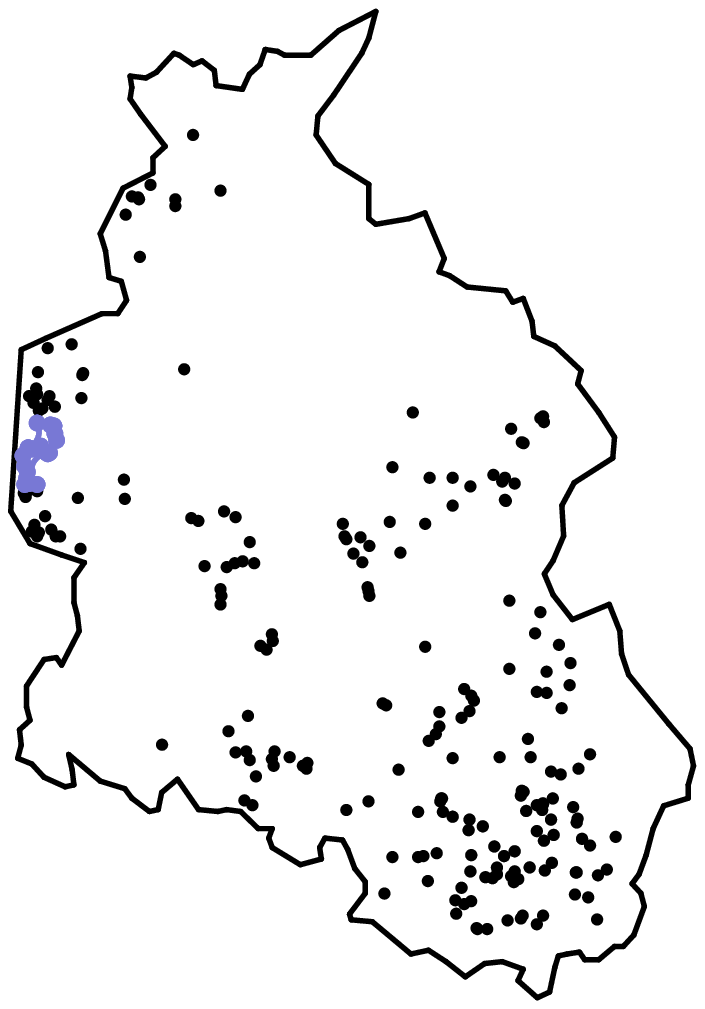}}
\pscircle[fillcolor=white,linecolor=black,linewidth=0.2mm ](4.3,4.745){.565}
\pscircle[fillcolor=white,linecolor=black,linewidth=0.2mm ](6.95,9.5){1.7}
\psline[fillcolor=white,linecolor=black,linewidth=0.2mm ](4.55,5.23)(6.15,8.0)
\psline[fillcolor=white,linecolor=black,linewidth=0.2mm ](6,10.88)(5.91,8.87)
\psline[fillcolor=white,linecolor=black,linewidth=0.2mm ](5.91,8.87)(6.2,8.3)
\psline[fillcolor=white,linecolor=black,linewidth=0.2mm ](6.2,8.3)(6.9,7.8)
\put(142,198){\includegraphics[height=5 cm, width=0.285\textwidth]{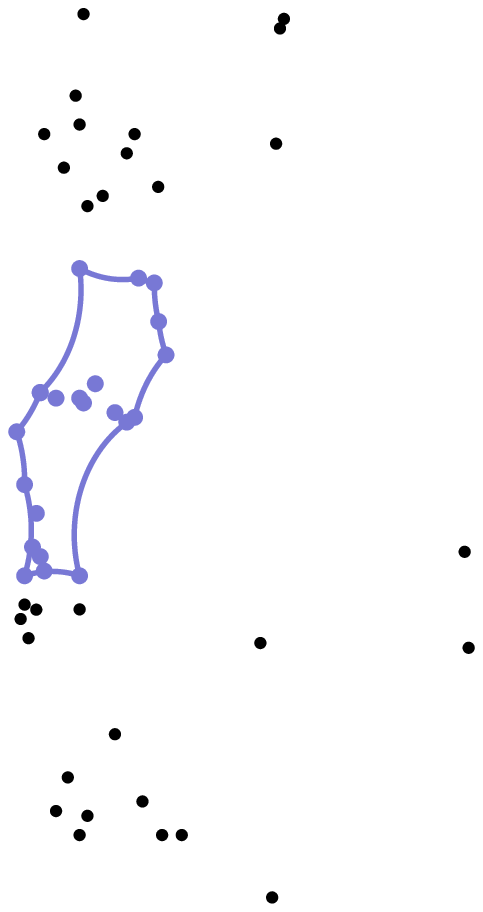}}
\put(70,6.5){\includegraphics[height=7.9cm, width=0.43\textwidth]{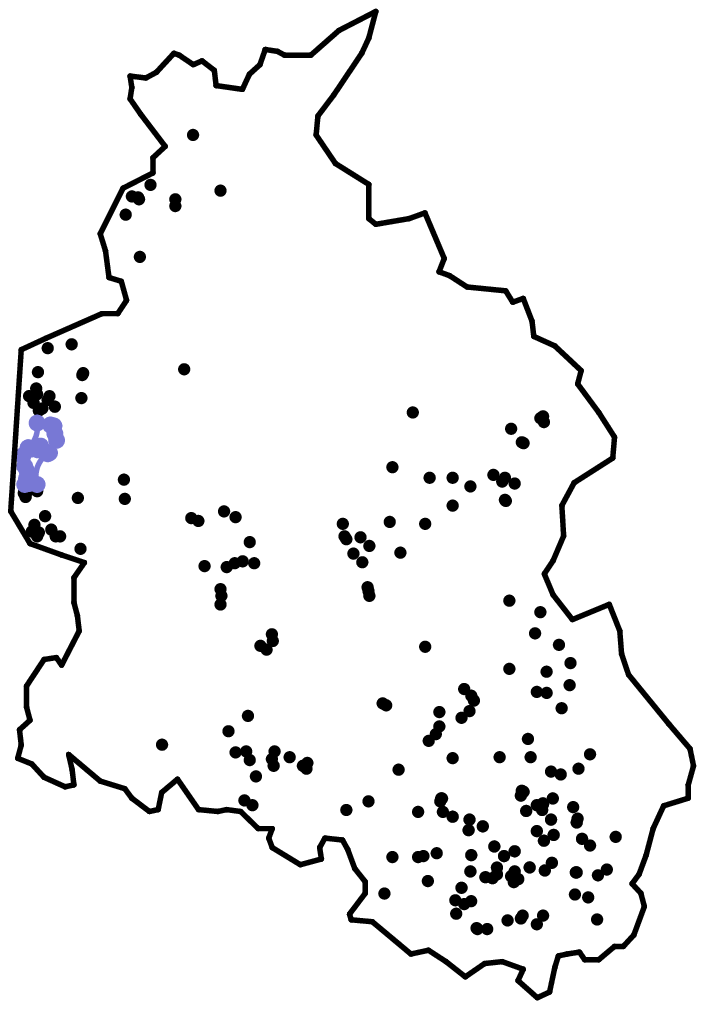}}
\pscircle[fillcolor=white,linecolor=black,linewidth=0.2mm ](8.9,4.745){.565}
\pscircle[fillcolor=white,linecolor=black,linewidth=0.2mm ](11.55,9.5){1.7}
\psline[fillcolor=white,linecolor=black,linewidth=0.2mm ](9.15,5.23)(10.75,8.0)
\psline[fillcolor=white,linecolor=black,linewidth=0.2mm ](10.6,10.88)(10.51,8.87)
\psline[fillcolor=white,linecolor=black,linewidth=0.2mm ](10.51,8.87)(10.8,8.3)
\psline[fillcolor=white,linecolor=black,linewidth=0.2mm ](10.8,8.3)(11.5,7.8)
\put(272,198){\includegraphics[height=5 cm, width=0.285\textwidth]{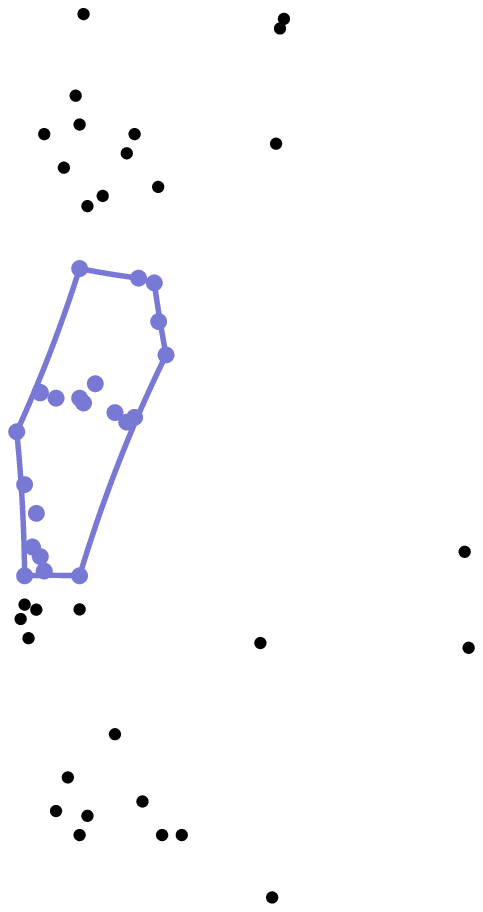}}
\put(200,6.5){\includegraphics[height=7.9cm, width=0.43\textwidth]{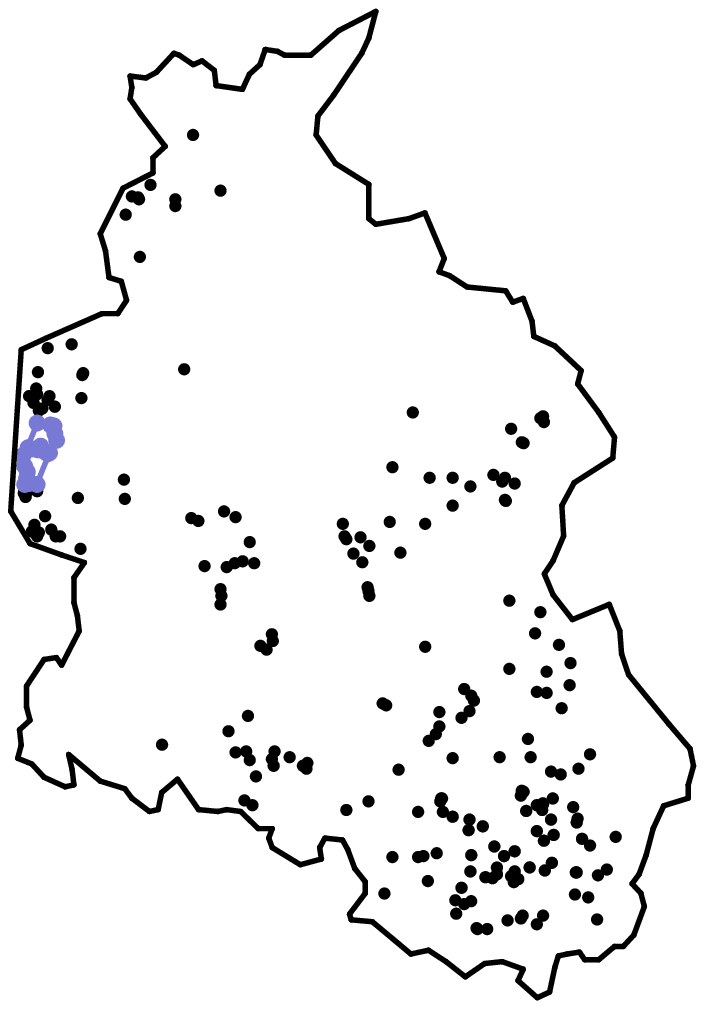}}
  \end{picture}\vspace{-1.3cm}\caption{The set $\mathcal{X}_{n}^+(t_1)$ is represented in blue. $C_{0.02}(\mathcal{X}_{n}^+(t_1))$ (left), $C_{0.03}(\mathcal{X}_{n}^+(t_1))$ (center) and $C_{0.3}(\mathcal{X}_{n}^+(t_1))$ (right).}\label{tttt1}
\end{figure}

\newpage

\begin{figure}[h!]
\hspace{1cm}\begin{picture}(-30,330)
\centering
\pscircle[fillcolor=white,linecolor=black,linewidth=0.2mm ](1.3,5.05){.8}
\put(4,211){\includegraphics[height=3.9cm, width=0.24\textwidth]{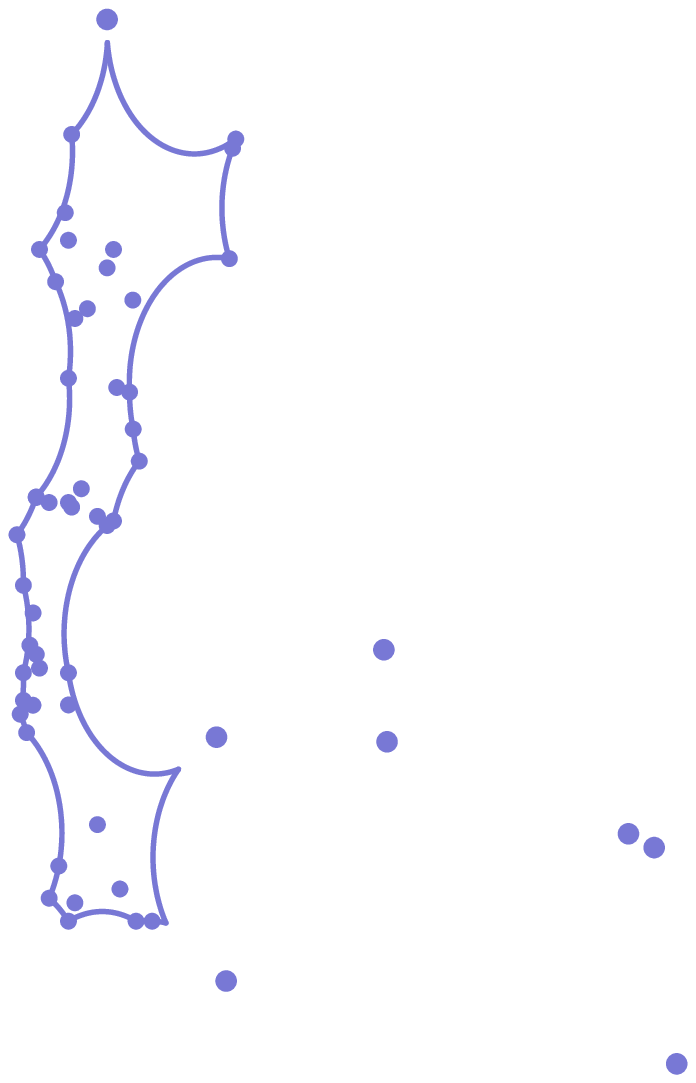}}
\put(-28,7.9){\includegraphics[height=8.3cm, width=0.43\textwidth]{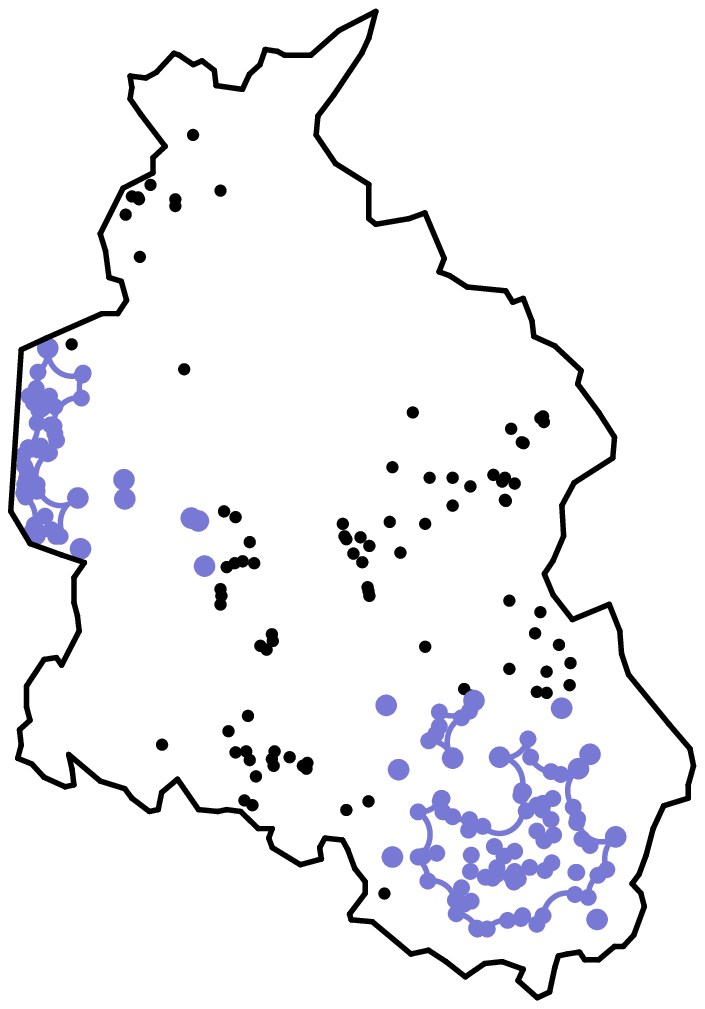}}
\pscircle[fillcolor=white,linecolor=black,linewidth=0.2mm ](3.07,3.13){.85}
\pscircle[fillcolor=white,linecolor=black,linewidth=0.2mm ](2.4,9.5){1.7}
\pscircle[fillcolor=white,linecolor=black,linewidth=0.2mm ](2,-.6){1.7}
\put(-19,-90){\includegraphics[height=4.75cm, width=0.3\textwidth]{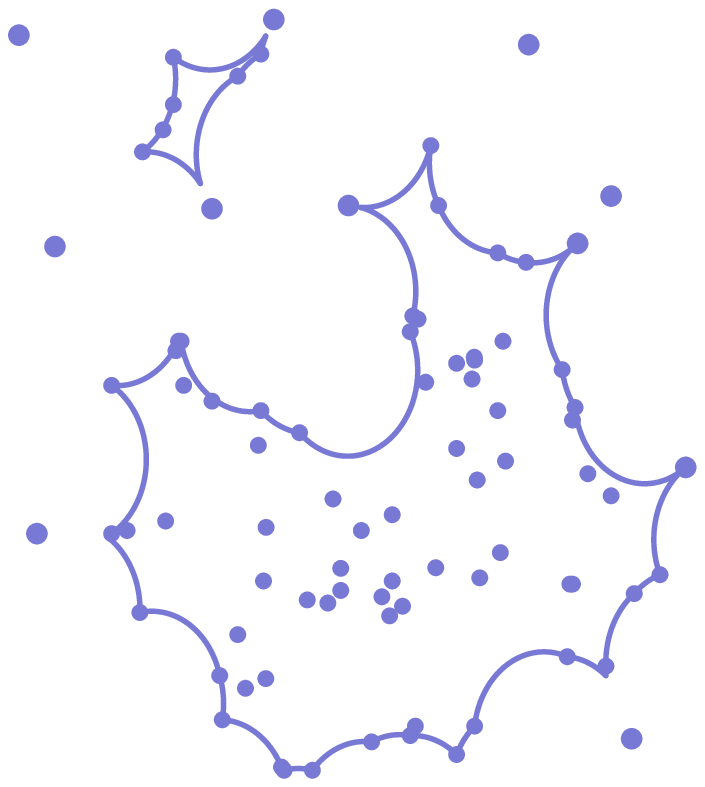}}
\psline[fillcolor=white,linecolor=black,linewidth=0.2mm ](1.65,8.)(1.85,8.1)
\psline[fillcolor=white,linecolor=black,linewidth=0.2mm ](1.85,8.1)(1.35,8.2)
\psline[fillcolor=white,linecolor=black,linewidth=0.2mm ](1.35,8.2)(1.08,8.5)
\psline[fillcolor=white,linecolor=black,linewidth=0.2mm ](1.08,8.5)(1.15,10.65)
\psline[fillcolor=white,linecolor=black,linewidth=0.2mm ](1.28,10.75)(2.,11.)
\psline[fillcolor=white,linecolor=black,linewidth=0.2mm ](2.,11.)(2.2,11)
\psline[fillcolor=white,linecolor=black,linewidth=0.2mm ](2.2,11)(2.3,11.2)
\psline[fillcolor=white,linecolor=black,linewidth=0.2mm ](6.25,8.)(6.45,8.1)
\psline[fillcolor=white,linecolor=black,linewidth=0.2mm ](6.45,8.1)(5.95,8.2)
\psline[fillcolor=white,linecolor=black,linewidth=0.2mm ](5.95,8.2)(5.68,8.5)
\psline[fillcolor=white,linecolor=black,linewidth=0.2mm ](5.68,8.5)(5.75,10.65)
\psline[fillcolor=white,linecolor=black,linewidth=0.2mm ](5.88,10.75)(6.6,11.)
\psline[fillcolor=white,linecolor=black,linewidth=0.2mm ](6.6,11.)(6.8,11)
\psline[fillcolor=white,linecolor=black,linewidth=0.2mm ](6.8,11)(6.9,11.2)
\psdots*[dotsize=1.5pt,linecolor=black](1.7,10.85)
\psdots*[dotsize=1.5pt,linecolor=black](2.8,10.35)
\psdots*[dotsize=1.5pt,linecolor=black](3.3,8.85)
\psdots*[dotsize=1.5pt,linecolor=black](3.39,8.78)
\psdots*[dotsize=1.5pt,linecolor=black](3.425,8.4)
\psdots*[dotsize=1.5pt,linecolor=black](6.3,10.85)
\psdots*[dotsize=1.5pt,linecolor=black](7.4,10.35)
\psdots*[dotsize=1.5pt,linecolor=black](7.9,8.85)
\psdots*[dotsize=1.5pt,linecolor=black](7.99,8.78)
\psdots*[dotsize=1.5pt,linecolor=black](8.025,8.4)
\psdots*[dotsize=1.5pt,linecolor=black](.6,-.82)
\psdots*[dotsize=1.5pt,linecolor=black](.45,-.835)
\psdots*[dotsize=1.5pt,linecolor=black](.73,-1.55)
\psdots*[dotsize=1.5pt,linecolor=black](1.453,0.6)
\psdots*[dotsize=1.5pt,linecolor=black](2.39,0.48)
\psdots*[dotsize=1.5pt,linecolor=black](2.25,0.48)
\psdots*[dotsize=1.5pt,linecolor=black](2.69,0.83)
\psdots*[dotsize=1.5pt,linecolor=black](2.68,0.57)
\psdots*[dotsize=1.5pt,linecolor=black](2.38,0.73)
\psdots*[dotsize=1.5pt,linecolor=black](2.,0.75)
\psdots*[dotsize=1.5pt,linecolor=black](5.2,-.82)
\psdots*[dotsize=1.5pt,linecolor=black](5.05,-.835)
\psdots*[dotsize=1.5pt,linecolor=black](5.33,-1.45)
\psdots*[dotsize=1.5pt,linecolor=black](6.2,0.6)
\psdots*[dotsize=1.5pt,linecolor=black](6.99,0.48)
\psdots*[dotsize=1.5pt,linecolor=black](6.85,0.48)
\psdots*[dotsize=1.5pt,linecolor=black](7.29,0.83)
\psdots*[dotsize=1.5pt,linecolor=black](7.28,0.57)
\psdots*[dotsize=1.5pt,linecolor=black](6.98,0.73)
\psdots*[dotsize=1.5pt,linecolor=black](6.6,0.75)
\psline[fillcolor=white,linecolor=black,linewidth=0.2mm ](1.45,5.83)(2.,7.87)
\psline[fillcolor=white,linecolor=black,linewidth=0.2mm ](6.05,5.8)(6.6,7.88)
\psline[fillcolor=white,linecolor=black,linewidth=0.2mm ](2.55,1.)(2.95,2.3)
\psline[fillcolor=white,linecolor=black,linewidth=0.2mm ](7.15,1.)(7.55,2.3)
\pscircle[fillcolor=white,linecolor=black,linewidth=0.2mm ](7.6,3.148){.9}
\pscircle[fillcolor=white,linecolor=black,linewidth=0.2mm ](7.0,9.5){1.7}
\pscircle[fillcolor=white,linecolor=black,linewidth=0.2mm ](6.6,-.6){1.7}
\pscircle[fillcolor=white,linecolor=black,linewidth=0.2mm ](5.87,5.08){.8}
\put(136,214){\includegraphics[height=4cm, width=0.25\textwidth]{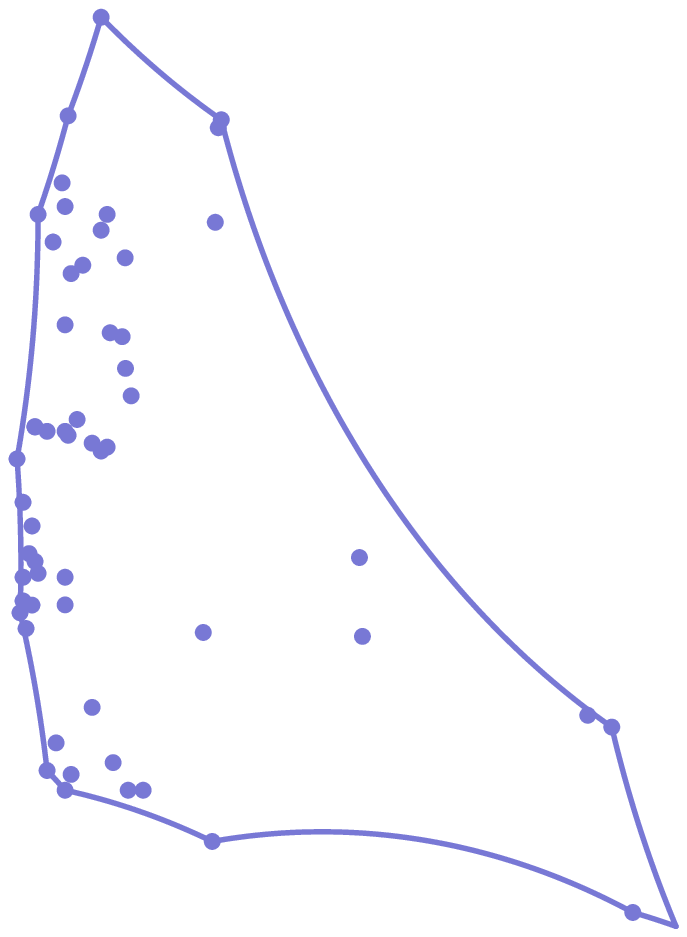}}
\put(103,7.9){\includegraphics[height=8.3cm, width=0.43\textwidth]{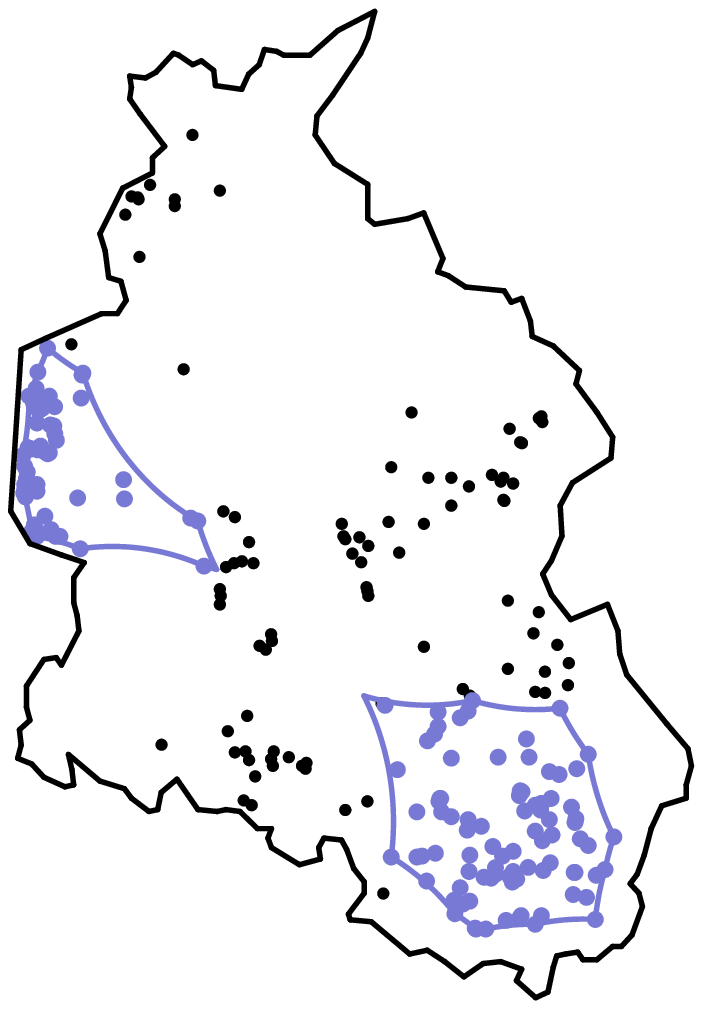}}
\put(123,-87){\includegraphics[height=4.75cm, width=0.279\textwidth]{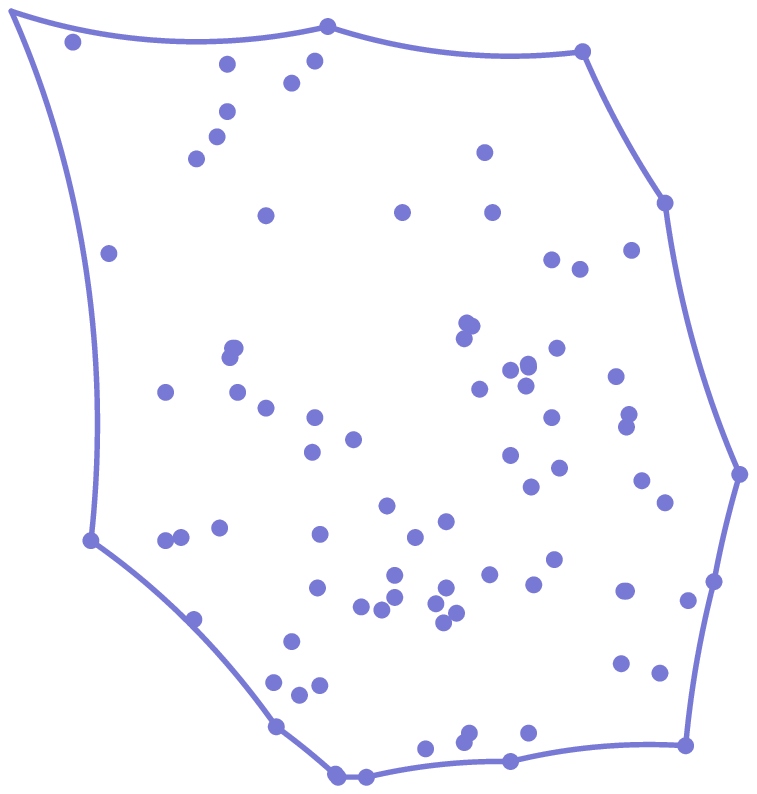}}
\psline[fillcolor=white,linecolor=black,linewidth=0.2mm ](3.4,-1.3)(3.68,-.5)
\psline[fillcolor=white,linecolor=black,linewidth=0.2mm ](3.4,-1.3)(3.49,-1.4)
\psline[fillcolor=white,linecolor=black,linewidth=0.2mm ](8.,-1.3)(8.28,-.5)
\psline[fillcolor=white,linecolor=black,linewidth=0.2mm ](8.,-1.3)(8.09,-1.4)
\psline[fillcolor=white,linecolor=black,linewidth=0.2mm ](4.98,-1)(5.1,-1.05)
\psline[fillcolor=white,linecolor=black,linewidth=0.2mm ](5.1,-1.05)(5.25,-1.4)
\psline[fillcolor=white,linecolor=black,linewidth=0.2mm ](5.25,-1.4)(5.18,-1.55)
\put(230,7.9){\includegraphics[height=8.3cm, width=0.43\textwidth]{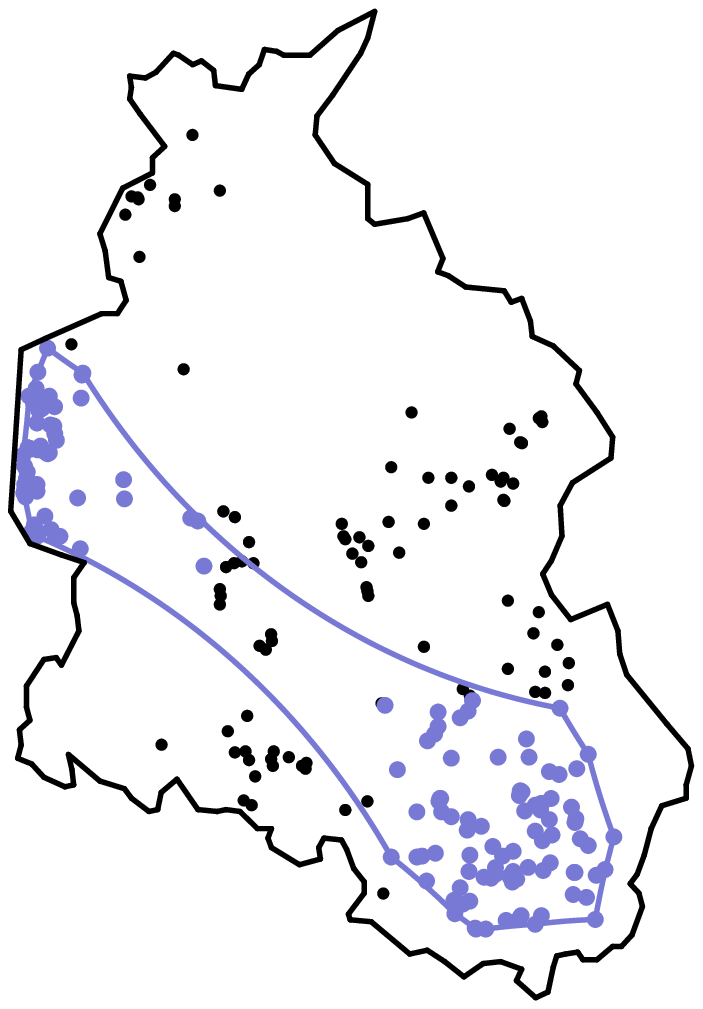}}
 \end{picture}\vspace{2.6cm}\caption{The set $\mathcal{X}_{n}^+(t_2)$ is represented in blue. $C_{0.03}(\mathcal{X}_{n}^+(t_2))$ (left), $C_{0.3}(\mathcal{X}_{n}^+(t_2))$ (center) and $C_{0.9}(\mathcal{X}_{n}^+(t_2))$ (right).}\label{tttt2}
\end{figure}

In Figure \ref{rty}, we show $\mathcal{X}_n^+(t)$ in blue for the data corresponding to 322 residential locations of cases of diagnosed of leukaemia on Lancashire and Greater Manchester, in the North West of England by considering two different values of the parameter $t>0$. This data set will be presented in detail in Section \ref{gorr} where a procedure for calculating, in practise, the sequence $D_n$ and hence $\mathcal{X}_n^+(t)$ and $\mathcal{X}_n^-(t)$ will be proposed. In these two cases, the $r-$convex hulls of $\mathcal{X}_n^+(t)$ are represented for different values of the parameter $r$ in Figures \ref{tttt1} and \ref{tttt2}, respectively. It is clear, see Figure \ref{tttt2}, that the influence of the parameter $r$ is important for estimating $G(t)$. Reconstructing it in a data-driven way is necessary.

\begin{proposition}\label{mostramaisenGlambda}Let $G(t)$ be a compact and nonempty level set. Under assumptions (A), (D) and (K), let $\mathcal{X}_n$ be a random sample generated from a distribution with density function $f$ and $\mathcal{X}_n^+(t)$ and $\mathcal{X}_n^-(t)$ established in Definition \ref{jejejeje}. Then,
\begin{equation*}\mathbb{P}\left(\mathcal{X}_n^+(t)\subset G(t),\mbox{ }\mathcal{X}_n^-(t)\subset G(t)^c , \mbox{ eventually}\right)=1.\end{equation*}
 \end{proposition}

Proposition \ref{alberto5} bounds the Euclidian distance between $G(t)$ and $\mathcal{X}_n^+(t)$ guaranteeing, in particular, that the set $\mathcal{X}_n^+(t)$ is nonempty eventually, see Figure \ref{jjjjjj}.

\begin{proposition}\label{alberto5}Let $G(t)$ be a compact and nonempty level set. Under assumptions (A) and (D), let $\mathcal{X}_n$ be a random sample generated from a distribution with density function $f$ and $\mathcal{X}_n^+(t)$ established in Definition \ref{jejejeje}. Then, for all $\epsilon>0$ it is verified that
$$\mathbb{P}\left( \sup_{x\in G(t)}d(x,\mathcal{X}_n^+(t))\leq \epsilon, \mbox{eventually}\right)=1.$$
\begin{figure}[h!]\centering
\begin{pspicture}(1,0)(14,5.35)
\psccurve[showpoints=false,fillstyle=solid,fillcolor=white,linecolor=black,linewidth=0.2mm,linearc=3](5,1)(5,4)(8,3.5)(10,4.5)(10,1)
\pscircle[ linecolor=black,linewidth=0.2mm,linearc=3,linestyle=dashed,dash=3pt 2pt](9.7,4.3){.7}
\psdots*[dotsize=2.5pt](9.7,4.3)
\psdots*[dotsize=2.5pt](9.7,3.8)
\rput(9.8,5.3){$\tiny{B_\epsilon(x_1)}$}
\rput(9.9,4.07){$\tiny{x_{n_1}^+}$}
\pscircle[ linecolor=black,linewidth=0.2mm,linearc=3,linestyle=dashed,dash=3pt 2pt](5.0,4.){.7}
\psdots*[dotsize=2.5pt](5,4.)
\rput(4.75,5){$\tiny{B_\epsilon(x_2)}$}
\psdots*[dotsize=2.5pt](5,3.5)
\rput(5.1,3.7){$\tiny{x_{n_2}^+}$}
\rput(4.2,1){$G(t)$}
\end{pspicture}
\caption{$x_1,\mbox{ }x_2\in G(t)$ and $x_{n_{i}}^+\in B_\epsilon(x_i),\mbox{ }i\in \{1,2\}$.}\label{jjjjjj}
\label{Figuraaa2}
\end{figure}
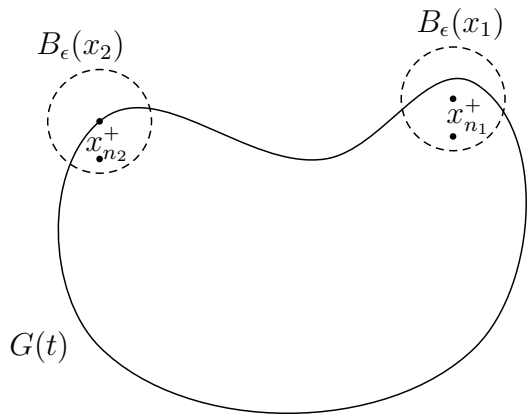

 \end{proposition}

Corollary \ref{xeneralizacionsoporte} shows, in particular, that $\mathcal{X}_n^+(t)$ is a consistent estimator for $G(t)$ in Hausdorff distance.

\begin{corollary}\label{xeneralizacionsoporte}Let $G(t)$ be a compact and nonempty level set. Under assumptions (A), (D) and (K), let $\mathcal{X}_n$ be a random sample generated from a distribution with density function $f$ and $\mathcal{X}_n^+(t)$ established in Definition \ref{jejejeje}. Then, for all $\epsilon>0$ it is verified that
$$\mathbb{P}\left( d_H( G(t),\mathcal{X}_n^+(t) )\leq \epsilon,  \mbox{ eventually}\right)=1.$$
\end{corollary}

\section{Main results}\label{mainresults}

Theorem \ref{rmaiorr02} will show that $\hat{r}_0(t)$ consistently estimates $r_0(t)$.

\begin{theorem}\label{rmaiorr02}Let $G(t)$ be a compact, nonempty and nonconvex level set. Under assumptions (A), (D) and (K), let $\mathcal{X}_n$ be a random sample generated from a distribution with density function $f$, $r_0(t)$ and $\hat{r}_0(t)$ established in Definitions \ref{r_0_t} and \ref{jejejeje}, respectively. Then,
$ \hat{r}_0(t)$  converges to $r_0$, almost surely.

\end{theorem}

Once the consistency for the estimator of the smoothing parameter $\hat{r}_0(t)$ defined in (\ref{estimadorr0hat}) has been proved, it would be
natural to study the behavior of the random set $C_{\hat{r}_0(t)}(\mathcal{X}_n^+(t))$ as an estimator for the level set $G(t)$.  However, the
consistency can not be guaranteed in this case. This problem can be solved by considering the estimator $C_{r_n(t)}(\mathcal{X}_n^+(t))$ as the estimator of the level set $G(t)$ where $r_n(t)=\nu \hat{r}_0(t)$ for a fixed value $\nu\in (0,1)$. This estimator provides a consistent reconstruction of the theoretical level set and the convergence rates are provided in Theorem \ref{principal}. Before exposing this key result, it is necessary to present an interesting auxiliary proof. Proposition \ref{llllll} establishes that the estimator $C_{r_n(t)}(\mathcal{X}_n^+(t))$ is contained in the theoretical level set with probability one and for $n$ large enough.

\begin{proposition}\label{llllll}Let $G(t)$ be a compact, nonempty and nonconvex level set. Under assumptions (A), (D) and (K), let $\mathcal{X}_n$ be a random sample generated from a distribution with density function $f$, $r_0(t)$ and $\hat{r}_0(t)$ established in Definitions \ref{r_0_t} and \ref{jejejeje}. Let $\nu\in(0,1)$ be a fixed number and $r_n(t)=\nu\hat{r}_0(t)$. Then,
$$\mathbb{P}( C_{r_n(t)}(\mathcal{X}_n^+(t))\subset G(t),\mbox{ eventually})=1.$$
\end{proposition}

\begin{theorem}\label{principal}Let $G(t)$ be a compact, nonempty and nonconvex level set. Under assumptions (A), (D) and (K), let $\mathcal{X}_n$ be a random sample generated from a distribution with density function $f$, let $\mathcal{X}_n^+(t)$ be established in Definition \ref{jejejeje} and let $r_n(t)=\nu \hat{r}_0(t)$ where $\nu\in (0,1)$ is a fixed number and $\hat{r}_0(t)$ defined in (\ref{estimadorr0hat}). Then, $$ d_H(C_{r_n(t)}(\mathcal{X}_n^+(t)),G(t))=O\left(\max\left\{ \left(  \frac{\log{n}}{n}\right)^{p/(d+2p)},\left(\frac{\log{n}}{n}\right)^{\frac{2}{d+1}}\right\}\right),$$ $\mbox{almost surely}$. The same convergence order holds for $d_\mu(C_{r_n(t)}(\mathcal{X}_n^+(t)),G(t))$.
\end{theorem}

\begin{remark}If the smoothing parameter is unknown for the granulometric smoothing method then it provides the same convergence rates than the algorithm proposed but it incurs a penalty, see Theorem 3 in Walther (1997). The rates obtained in Theorem \ref{principal} do no depend on any penalty term because, although $r_0(t)$ is a priori unknown, it is estimated in a data-driven way from $\mathcal{X}_n$.\end{remark}

\section{A real example}\label{gorr}

The question of whether the geographical incidence of disease shows any tendency towards
clustering in geographical space has a long and rich history. For instance, do cases of disease tend to occur in proximity to other cases? The problem has become
more urgent in recent years in the light of concerns raised about possible links between disease incidence and potential sources of environmental contamination,
such as nuclear installations.

A priori we may expect to observe a certain amount of clustering due to natural background variation in the population from which events arise. For example,
cases of cancer will always cluster because of the distribution of population at risk. In such instances, we are more interested in detecting evidence of clustering over and above this underlying environmental heterogeneity; in other words, in discovering whether the distribution of one type of event clusters relative to that of another. Then, it is really interesting to consider level sets that have a probability content greater than or equal to $1-\tau$ with $\tau\in(0,1)$ fixed a priori by the  practitioner. Therefore, the value of $t$ is unknown and an alternative level set definition can be presented:
$$
L(\tau)=\{x\in\mathbb{R}^d:f(x)\geq f_\tau\}\vspace{1.5mm}
$$where
$$
f_\tau=\sup
\left\{y\in(0,\infty):\int_{-\infty}^{\infty}f(t)\mathbb{I}_{\{f(t)\geq
y\}}\geq1-\tau\right\}. \vspace{2.8mm}
$$Given $\tau\in (0,1)$, the threshold $f_\tau$ must be estimated before selecting a method for reconstructing the level set. It is possible to consider numerical integration methods or calculate the $\tau-$quantile of the empirical distribution of $f_n(X_1),...,f_n(X_n)$, see Cadre (2006) and Cadre et al. (2009), respectively.

Then, from a practical point of view, reconstructing level sets $L(\tau)$ may be more interesting than estimating $G(t)$. In this work, the algorithm for calculating the estimator for the smoothing parameter defined in (\ref{estimadorr0hat}) is detailed next for this particular case. Of course, it could be easily adapted if level sets $G(t)$ must be reconstructed.

Once the value of $\tau\in (0,1)$ is given by the practitioner, it should be natural, as first step, to estimate the threshold $f_\tau$ and, then, determinate the sets $\mathcal{X}_n^+(\hat{f}_\tau)$ and $\mathcal{X}_n^-(\hat{f}_\tau)$. However, these two previous sets depend on the sequence $D_n$ that tends to zero when the sample size tends to infinity, see Definition \ref{jejejeje}. This sequence does not rely on $\mathcal{X}_n$. However, in practise and for a fixed value of $n$, we think that establishing some relationship between them could be really useful. For this, a bootstrap procedure will be proposed in order to estimate two values of two probability content verifying that $\hat{\tau}^-\leq \tau\leq \hat{\tau}^+$. In addition, it is assumed that $\hat{\tau}^+$ and $\hat{\tau}^-$ could not to be symmetric around the $\tau$. From these two values, two thresholds $\hat{f}_\tau^+$ and $\hat{f}_\tau^-$ can be determinated. Therefore, it would be possible to calculate the subsets $\mathcal{X}_{n,+}(\hat{f}_\tau^+)$ and $\mathcal{X}_{n,-}(\hat{f}_\tau^-)$, see Notation for remembering their definitions. Notice that, in most of cases, $\mathcal{X}_n\neq \mathcal{X}_{n,+}(\hat{f}_\tau^+)\cup \mathcal{X}_{n,-}(\hat{f}_\tau^-)$. Therefore, the information contained in $\mathcal{X}_n\setminus (\mathcal{X}_{n,+}(\hat{f}_\tau^+)\cup \mathcal{X}_{n,-}(\hat{f}_\tau^-))$ is not taken in advantage. To solve this, we propose to use  $k-$nearest neighbors considering $\mathcal{X}_{n,+}(\hat{f}_\tau^+)$ and $\mathcal{X}_{n,-}(\hat{f}_\tau^-)$ as training samples for classifying the full sample $\mathcal{X}_n$. In particular, the set $\mathcal{X}_n\setminus (\mathcal{X}_{n,+}(\hat{f}_\tau^+)\cup \mathcal{X}_{n,-}(\hat{f}_\tau^-))$ will be classified. Therefore, a value $\hat{k}\geq 1$ for the nearest neighbors must selected too. Below, the bootstrap procedure considered for calculating  $\hat{f}_\tau^+$, $\hat{f}_\tau^-$ and $\hat{k}$ will be explained in detail. Before, the algorithm for estimating $r_0(f_\tau)$ will be exposed. For simplicity in the exposition, the estimator will be denoted by $\hat{r}_0(\hat{f}_{\tau})$. Dichotomy algorithms will be considered. Therefore, a maximum number of iterations $J$ and two initial points $r_m$ and $r_M$ with $r_m<r_M$ must be selected. In practise, it is necessary to guarantee that $C_{r_m}(\mathcal{X}_{n,+}(\hat{f}_\tau^+))\cap \mathcal{X}_{n,-}(\hat{f}_\tau^-)= \emptyset$ and $C_{r_M}(\mathcal{X}_{n,+}(\hat{f}_\tau^+))\cap \mathcal{X}_{n,-}(\hat{f}_\tau^-)\neq \emptyset$, respectively. Then, a value close enough to zero must be chosen for $r_m$ and $r_M$ should be big enough for guaranteeing that $C_{r_M}(\mathcal{X}_{n,+}(\hat{f}_\tau^+))$ coincides or is almost equal to $\mbox{conv}(\mathcal{X}_{n,+}(\hat{f}_\tau^+))$. Of course, if $\mbox{conv}(\mathcal{X}_{n,+}(\hat{f}_\tau^+))\cap \mathcal{X}_{n,-}(\hat{f}_\tau^-)= \emptyset$ then $\hat{r}_0(\hat{f}_{\tau })=\infty$ and, therefore, $\hat{L}(\tau)=\mbox{conv}(\mathcal{X}_{n,+}(\hat{f}_\tau^+))$. Taking the previous comments under consideration, $\hat{r}_0(\hat{f}_{\tau })$ will be computed as follows:

\begin{enumerate}
\item Use $\hat{k}-$nearest neighbors algorithm considering $\mathcal{X}_{n,+}(\hat{f}_\tau^+)$ and $\mathcal{X}_{n,-}(\hat{f}_\tau^-)$ as a trainning sample for classifying the original full sample $\mathcal{X}_n$. The two resulting sets are denoted, for simplicity in the exposition, by the name of the original sets.
   \item In each iteration and while the number of them is smaller than $J$:\vspace{.08cm}
      \begin{enumerate}
      \item $r=(r_m+r_M)/2.$\vspace{.08cm}
      \item If $\mathcal{X}_{n,-}(\hat{f}_\tau^-)\cap C_{r}(\mathcal{X}_{n,+}(\hat{f}_\tau^+))\neq \emptyset$ then $r_M=r$.\vspace{.08cm}
       \item Otherwise, $r_m=r$.\vspace{.08cm}
       \end{enumerate}Then, $\hat{r}_0(\hat{f}_{\tau })=r_m$ and $\hat{L}(\tau)=C_{\hat{r}_0(\hat{f}_{\tau })}(\mathcal{X}_{n,+}(\hat{f}_\tau^+))$.\vspace{.08cm}
\end{enumerate}

It should be noted that the $r-$convex and convex hulls of a sample points can be easily computed (at least for the bidimensional case), see Pateiro-L\'opez and Rodr\'iguez-Casal (2010) and Renka (1996), respectively.

Once the algorithm for estimating the smoothing parameter was exposed, it only remains to detail the procedure in order to calculate the two thresholds, $\hat{f}_\tau^+$ and $\hat{f}_\tau^-$, and $\hat{k}$. A bootstrap method is proposed for selecting them by minimizing an error criteria between sets, the distance in measure. In an analogous way, other distances between sets could be considered.

To sum up, the next inputs should be given: the probability content $\tau\in(0,1)$, a big enough sample size $M$, a step $\Delta$ and a positive integer $I$ for defining the vectors $\tau^+=(\tau,\tau+\Delta,...,\tau+I\Delta)$ and $\tau^-=(\tau,\tau-\Delta,...,\tau-I\Delta)$ verifying $\tau+I\Delta\leq (n-1)/n$ and $\tau-I\Delta\geq 1/n$ in order to avoid empty sets, a number of bootstrap iterations $B$, a vector $k$ of length $K$ containing the number of nearest neighbors to be considered and, as before, a maximum number of iterations $J$ for the dichotomy algorithm. On the other hand, the selector for the bandwidth parameter of Bowman (1984) and Rudemo (1982) could be considered for density estimation in the univariate case since its generalization for the multivariate case is computed easily, see Duong (2007).

\begin{enumerate}
   \item Estimate by Monte Carlo approach the threshold $f_\tau^*$ in the bootstrap world:
    \begin{enumerate}
          \item Draw a bootstrap sample of size $M$ from $f_n$ where $f_n$ denotes the kernel estimator with bandwidth $H$ obtained from $\mathcal{X}_n$. It is denoted by $\mathcal{X}_M^*$.
          \item Obtain $f_\tau^*$ determinating the quantile $\tau$ of the empirical distribution of $f_n(\mathcal{X}_M^*)$. Therefore, $L^*(\tau)=\{f_n\geq f_\tau^*\}$ represents the theoretical level set in the bootstrap world.
  \end{enumerate}
        \item This step must be repeated $B$ times:
        \begin{enumerate}
            \item Draw a bootstrap sample of size $n$ from $f_n$. It will be denoted by $\mathcal{X}_n^*$.
            \item Calculate $f_n(\mathcal{X}_n^*)$ and $f_n^{*}(\mathcal{X}_n^*)$ where $f_n^{*}$ is the kernel estimator calculated from $\mathcal{X}_n^*$ with bandwidth $H^*$.
            \item In each iteration, while $j_1$ and $j_2$ are smaller or equal than $I+1$ and while $j_3$ is smaller than $K$:
                    \begin{enumerate}
                         \item Obtain $\hat{f}_{\tau,*}^+$ and $\hat{f}_{\tau,*}^-$ determinating the quantiles $\tau^+(j_1)$ and $\tau^-(j_2)$ of the empirical distribution of $f_n^{*}(\mathcal{X}_n^*)$, respectively.
                        \item Calculate $\mathcal{X}_{n,+}^*(\hat{f}_{\tau,*}^+)$, $\mathcal{X}_{n,-}^*(\hat{f}_{\tau,*}^-)$ and $\mathcal{X}_n^*\setminus (\mathcal{X}_{n,+}^*(\hat{f}_{\tau,*}^+)\cup\mathcal{X}_{n,-}^*(\hat{f}_{\tau,*}^-))$.
                        \item Use $k(j_3)-$nearest neighbors algorithm considering $\mathcal{X}_{n,+}^*(\hat{f}_{\tau,*}^+)$ and $\mathcal{X}_{n,-}^*(\hat{f}_{\tau,*}^-)$ as a trainning sample for classifying the full sample $\mathcal{X}_n^*$. The two resulting sets are denoted, for simplicity in the exposition again by $\mathcal{X}_{n,+}^*(\hat{f}_{\tau,*}^+)$ and $\mathcal{X}_{n,-}^*(\hat{f}_{\tau,*}^-)$.
                            \item Estimate the smoothing parameter from $\mathcal{X}_{n,+}^*(\hat{f}_{\tau,*}^+)$ and $\mathcal{X}_{n,-}^*(\hat{f}_{\tau,*}^-)$ using the previous dichotomy algorithm. It will be denoted by  $\hat{r}_0^*(\hat{f}_{\tau,*})$.
                                                       \item Estimate the error $d_\mu \left(L^*(\tau), C_{\hat{r}_0^*(\hat{f}_{\tau,*})}(\mathcal{X}_{n,+}^*(\hat{f}_{\tau,*}^+)) \right)$ induced by $f_n$ as follows:
                                 \begin{enumerate}
    \item Draw another bootstrap sample $\mathcal{Y}_M^*$ from $f_n$ of size $M$.
          \item Determinate which points in $\mathcal{Y}_M^*$ are and are not in $L^*(\tau)$:  $$\mathcal{Y}_{M,+}^{*}(f_\tau^*)=\{Y\in \mathcal{Y}_M^*:f_n(Y)\geq f_\tau^*\}$$and$$\mathcal{Y}_{M,-}^{*}(f_\tau^*)=\{Y\in \mathcal{Y}_M^*:f_n(Y)< f_\tau^*\}.$$
                                \item Calculate the cardinal of the union of the two sets \small{
                                            $$     \left\{\mathcal{Y}_{M,-}^{*}(f_\tau^*)\cap C_{\hat{r}_0^*(\hat{f}_{\tau,*})}(\mathcal{X}_{n,+}^*(\hat{f}_{\tau,*}^+))\right\}$$and$$ \left\{\mathcal{Y}_{M,+}^*(f_\tau^*)\cap C_{\hat{r}_0^*(\hat{f}_{\tau,*})}(\mathcal{X}_{n,+}^*(\hat{f}_{\tau,*}^+))^c\right\} .$$}Then, divide the result obtained by $M$.
                                                               \end{enumerate}
                    \end{enumerate}
                              \end{enumerate}
                              \item Select the values in $\tau^+$, $\tau^-$ and $k$ which provides the lowest empirical means of the $B$ errors calculated. They will be denoted by $\hat{\tau}^+$, $\hat{\tau}^-$ and $\hat{k}$, respectively.
                                   \item Obtain $\hat{f}_\tau^+$ and $\hat{f}_\tau^-$ determinating the quantiles $\hat{\tau}^+$ and $\hat{\tau}^-$ of the empirical distribution of $f_n(\mathcal{X}_n)$, respectively.
                    \end{enumerate}

In order to assess the applicability of this estimation method, we will consider a real data set. It derives from the study that provided the data in Henderson et al. (2002). It contains $1221$ pairs of points in Lancashire and Greater Manchester. Concretely, it contains the residential coordinates for the 233 cases of diagnosed chronic granulocytic leukemia registered between 1982 up to 1998 (inclusive), together with 988 controls. For the selection of controls, population counts in each of the 8131 census enumeration districts that make up the study-region, stratified by age and sex, were extracted from the 1991 census. The counts were then used to obtain a stratified random sample of two controls per case with coordinates given by their corresponding centroid coordinates (slightly jittered to avoid coincident points). In Figure \ref{oouigf}, the contour of Lancashire and Greater Manchester and the samples of cases and controls are showed. This data set is available on the  of Prof. Peter J. Diggle, Lancaster University.

The evidence for clustering of the cases of leukaemia in the North West of England will be studied. Analyzing whether the distribution of leukaemia mirrored that of the population as a whole or whether there was evidence, as implied by concerned local residents, of clustering. For this, it could help identify the peaks or the modes of the density estimation in the resulting surface allowing to visualize easily an excess of case intensity over that of population.

Then, we have estimated the level sets $L(\tau)$ from the samples of cases and controls for relatively high values of the probability content $\tau$. More specifically, the values of $\tau$ considered are $0.7$, $0.75$, $0.8$, $0.85$, $0.9$ and $0.95$. In addition, we have fixed $I=10$, $\Delta=\min\{(1-\tau-3/n)/I,(\tau-3/n)/I,0.01\} $ where $n$ denotes the sample size of cases or controls depending on the situation, $k=(1,3,5)$, $M=3000$, and $B=500$.

Following the previous algorithm, Table \ref{esti} shows the values obtained for $\hat{k}$, $\hat{\tau}^+$, $\hat{\tau}^-$ and $\hat{r}_0(\hat{f}_\tau)$ for the samples of cases and controls with the different values of $\tau$ considered. According to the results obtained for $\hat{r}_0(\hat{f}_\tau)$, $r-$convexity property plays an interesting role, mainly for the sample of cases. Only for $\tau$ equal to $0.95$ the level set estimator is convex.
The level set estimators for
$ $

$ $

\begin{figure}[h!]\vspace{-.6cm}
\begin{picture}(10,210)
\put(120,30){\includegraphics[height=7.2cm, width=.4\textwidth]{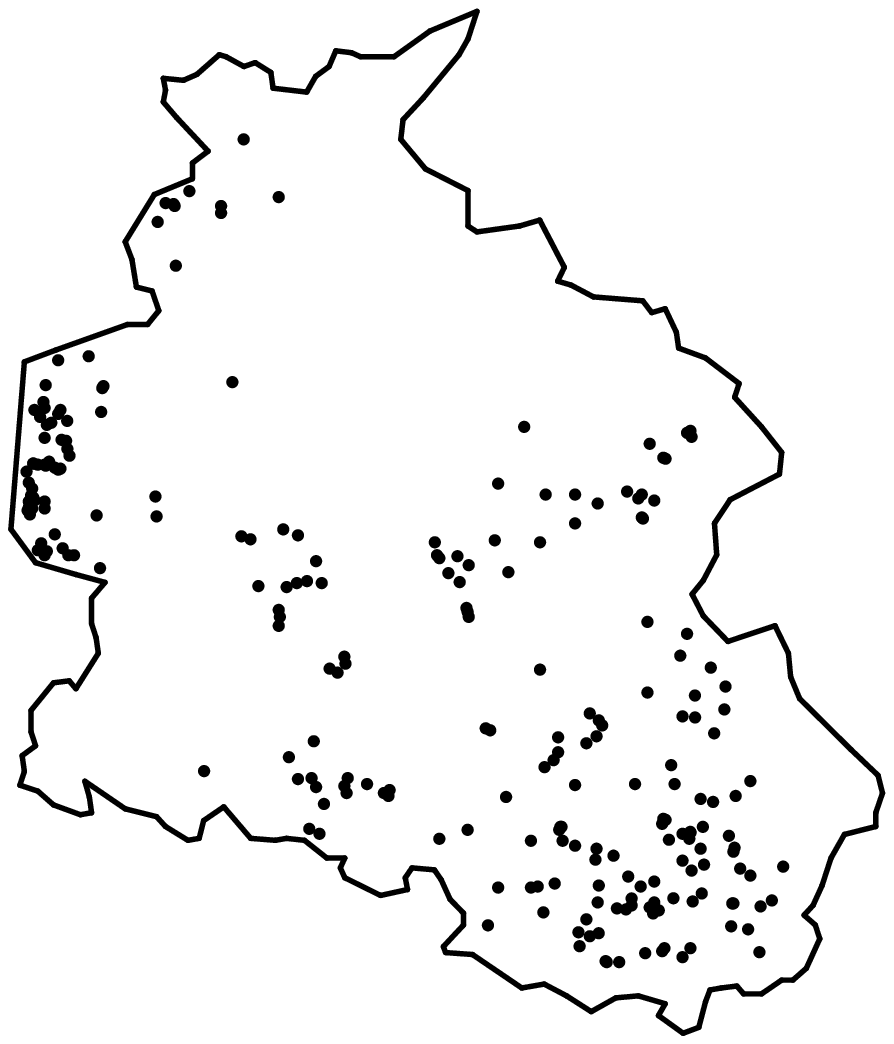}} 
\put(260,30){\includegraphics[height=7.2cm, width=.4\textwidth]{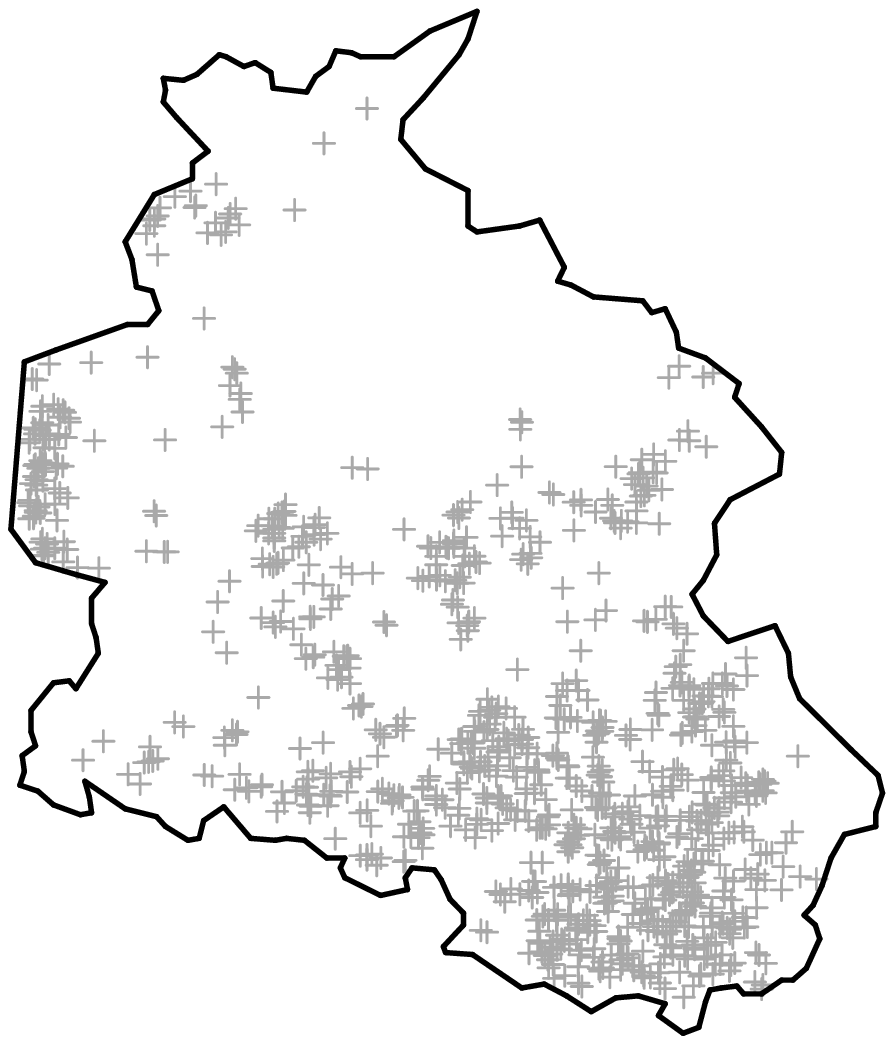}}
\put(-20,30){\includegraphics[height=7.2cm, width=.4\textwidth]{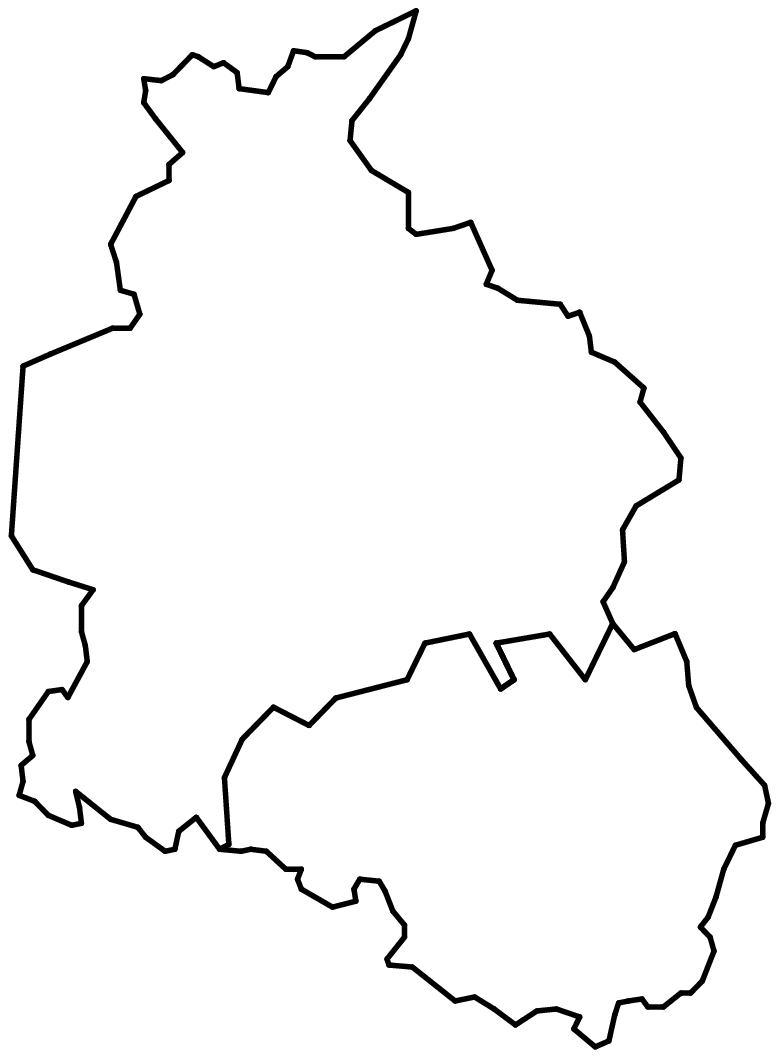}}
\put(33,150){Lancashire}
\put(64,110){Greater}
\put(49,100){Manchester}
\end{picture}\vspace{-2.2cm}
\caption{\normalsize Sub-regions of Lancashire and Greater Manchester on the North West of England (left), distribution of 233 cases of diagnosed leukaemia (center) and 988 controls on Lancashire and Greater Manchester (right) in the North West of England.}\label{oouigf}
\end{figure}

$ $

$ $

 \begin{table}[h!]
  \begin{center}
\begin{tabular}{rrrrrr}
    &  $\tau$& $\hat{\tau}^+$&$\hat{\tau}^-$&$\hat{k}$&$\hat{r}_0(\hat{f}_\tau)$\\
        \hline
     \multirow{7}{*}{Cases}& & & & \\
      & $0.7$&  0.7  &   0.66     &   5  &  0.529 \\
      & $0.75$&    0.75   &  0.69 &   5   & 0.571 \\
      & $0.8$&  0.84 &   0.70    &   1   &   0.382 \\
      & $0.85$&   0.85   &   0.8 &  3  &0.434     \\
       & $0.9$&   0.908  & 0.812   & 1   &   0.544   \\
        & $0.95$&   0.975   &  0.95   & 5  &  $\infty$   \\
     \multirow{7}{*}{Controls}& & & &  \\
       & $0.7$&   0.74 &     0.62     &  1  & 0.060  \\
      & $0.75$&    0.78   &  0.66 &  1   &0.265  \\
      & $0.8$&  0.8  &  0.8   &  1   &  0.178 \\
      & $0.85$&   0.93  &    0.85 & 1   & $\infty$    \\
       & $0.9$&   0.967  &  0.822  &  1 &  $\infty$    \\
        & $0.95$&    0.973   & 0.907   &    1& $\infty$  \\
    \hline
       \end{tabular}
 \end{center}\caption{Estimators of $k$, $\tau^+$, $\tau^-$ and $r_0(f_\tau)$ for the samples of cases and controls with different values of $\tau$.}\label{esti}
 \end{table}
\newpage

\begin{figure}[h!]\hspace{1.1cm}\vspace{-3cm}
\begin{picture}(-85,260)
\scalebox{0.8}[0.8]{
\put(-49,-56){\includegraphics[height=12 cm, width=.7\textwidth]{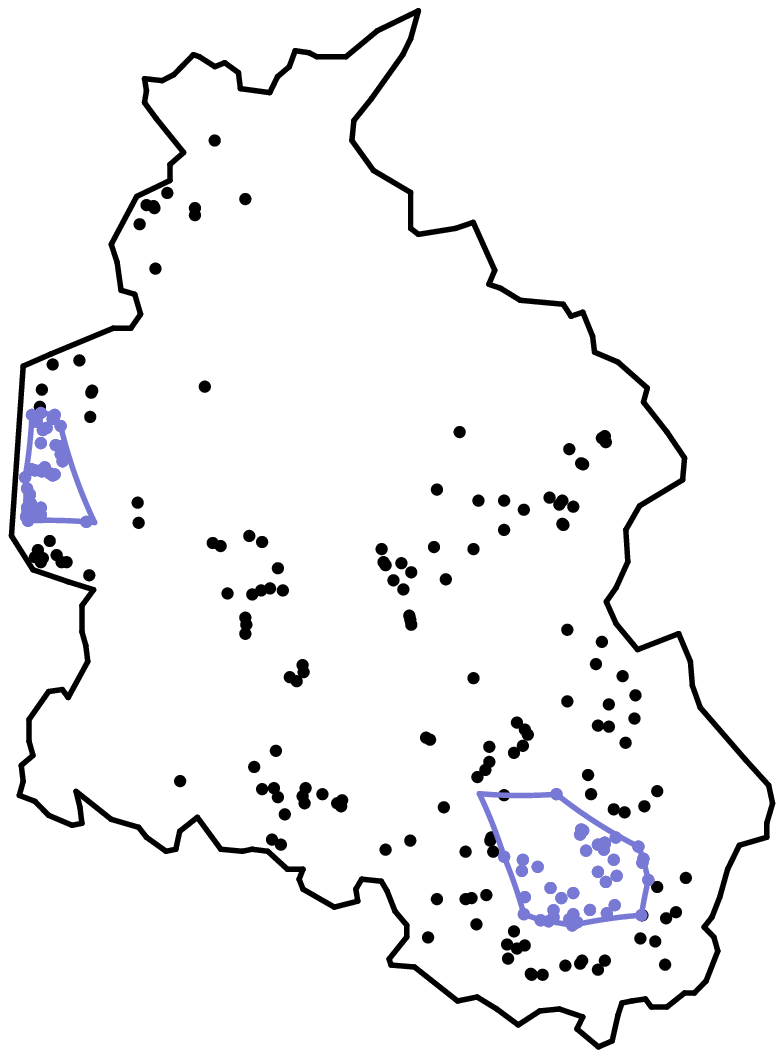}}
\put(157,-56){\includegraphics[height=12cm, width=.7\textwidth]{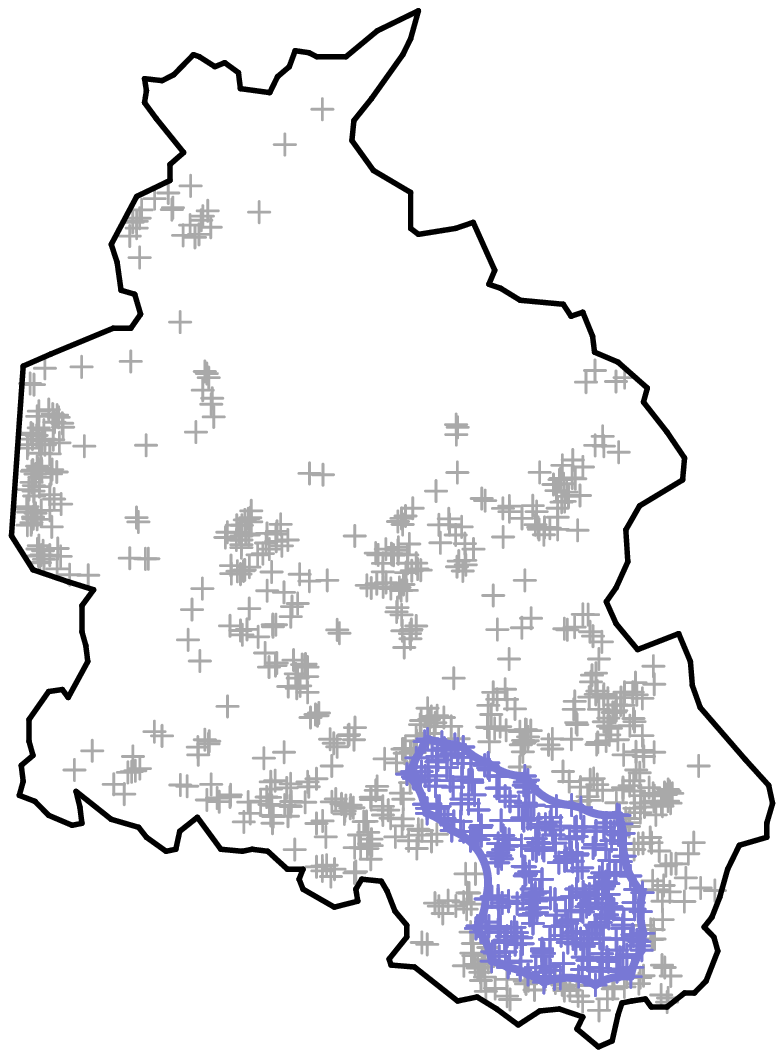}}
\put(-49,-311){\includegraphics[height=12 cm, width=.7\textwidth]{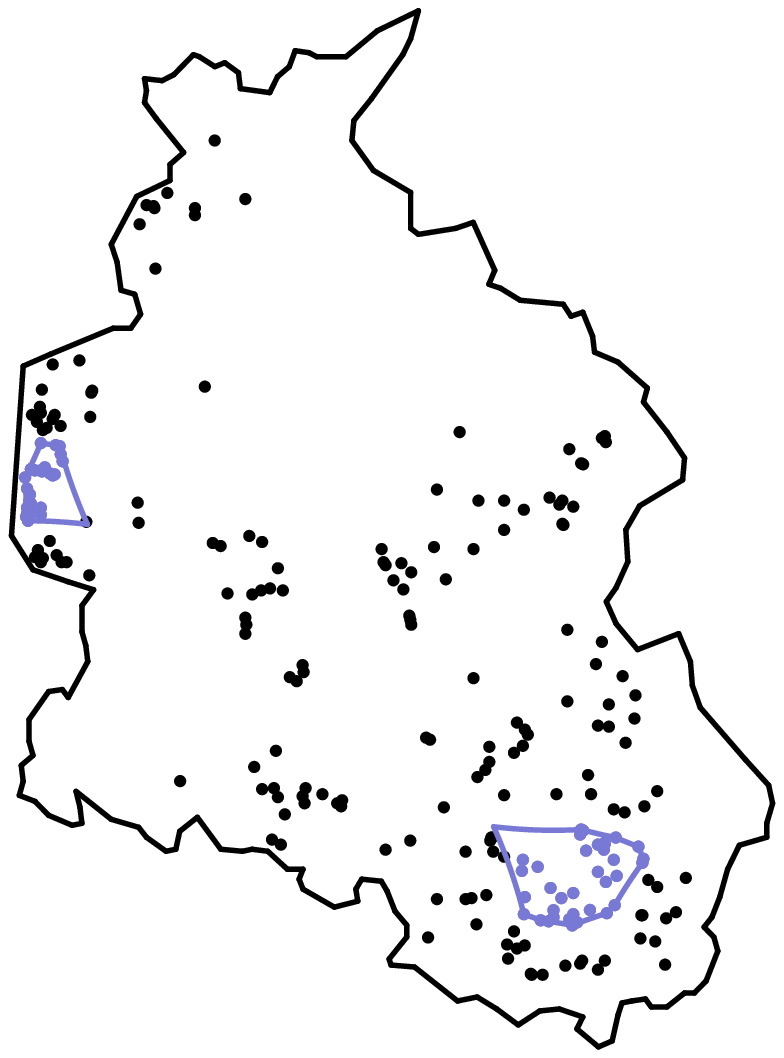}}
\put(157,-311){\includegraphics[height=12cm, width=.7\textwidth]{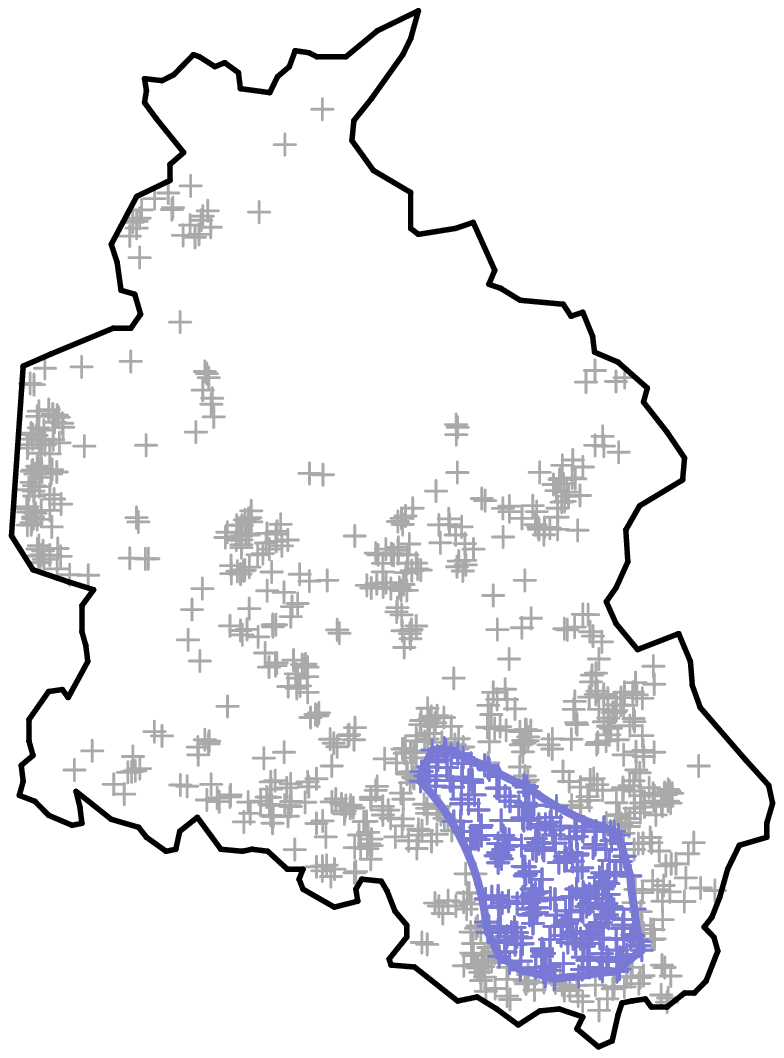}}
}\end{picture}
\vspace{10.cm}\caption{In the first column, estimated level sets for the sample of 322 cases diagnosed of leukaemia on the North West of England with $\tau=0.7$ (top) and $\tau=0.75$ (bottom). In the second column, estimated level sets for the sample of 988 controls of leukaemia on the North West of England with $\tau=0.7$ (top) and $\tau=0.75$ (bottom).}\label{hhhhhiuiop1}
\end{figure}

\newpage

\begin{figure}[h!]\hspace{1.1cm}\vspace{-3cm}
\begin{picture}(-85,260)
\scalebox{0.8}[0.8]{
\put(-49,-56){\includegraphics[height=12 cm, width=.7\textwidth]{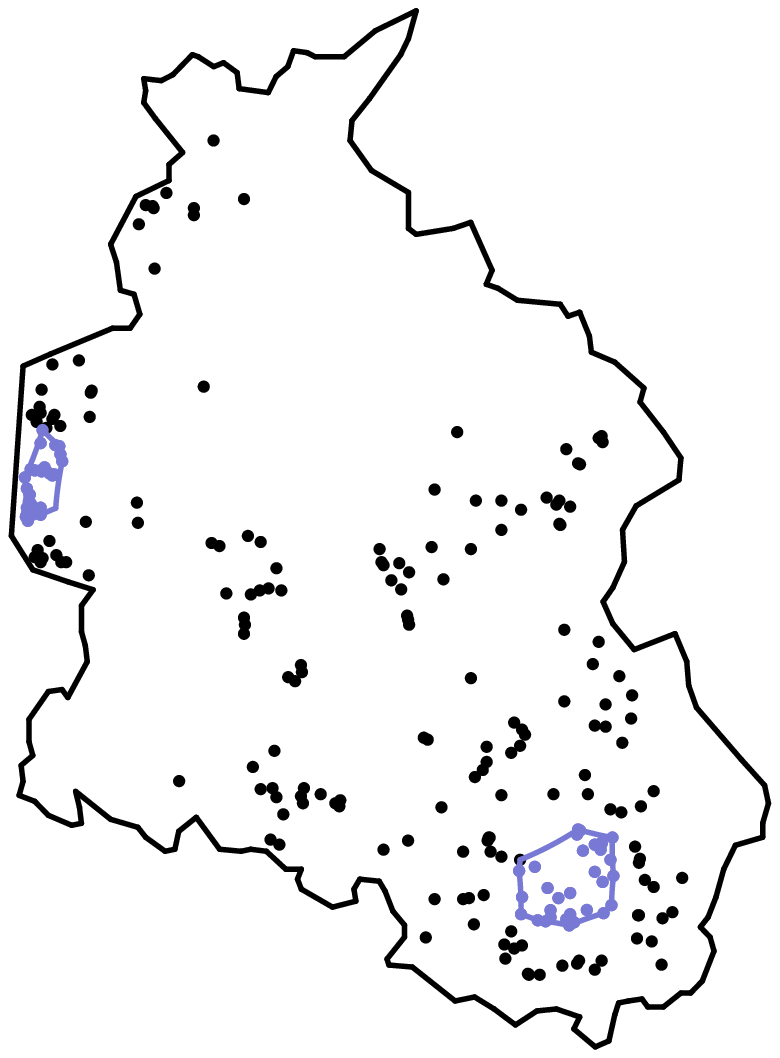}}
\put(157,-56){\includegraphics[height=12cm, width=.7\textwidth]{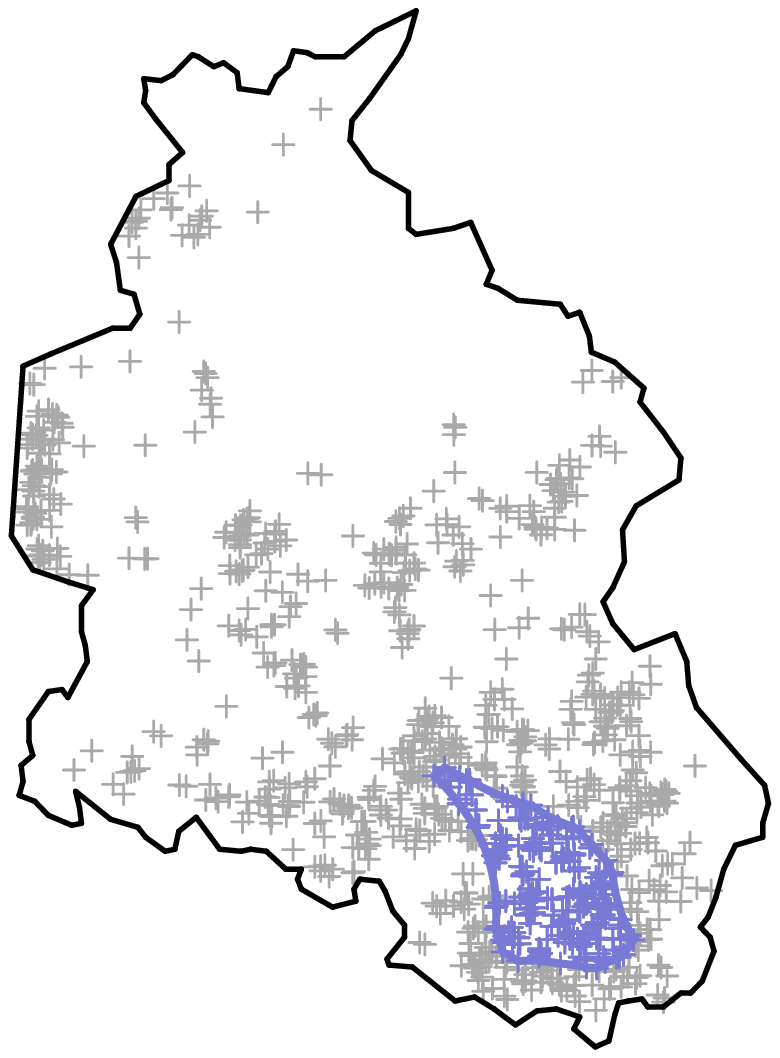}}
\put(-49,-311){\includegraphics[height=12 cm, width=.7\textwidth]{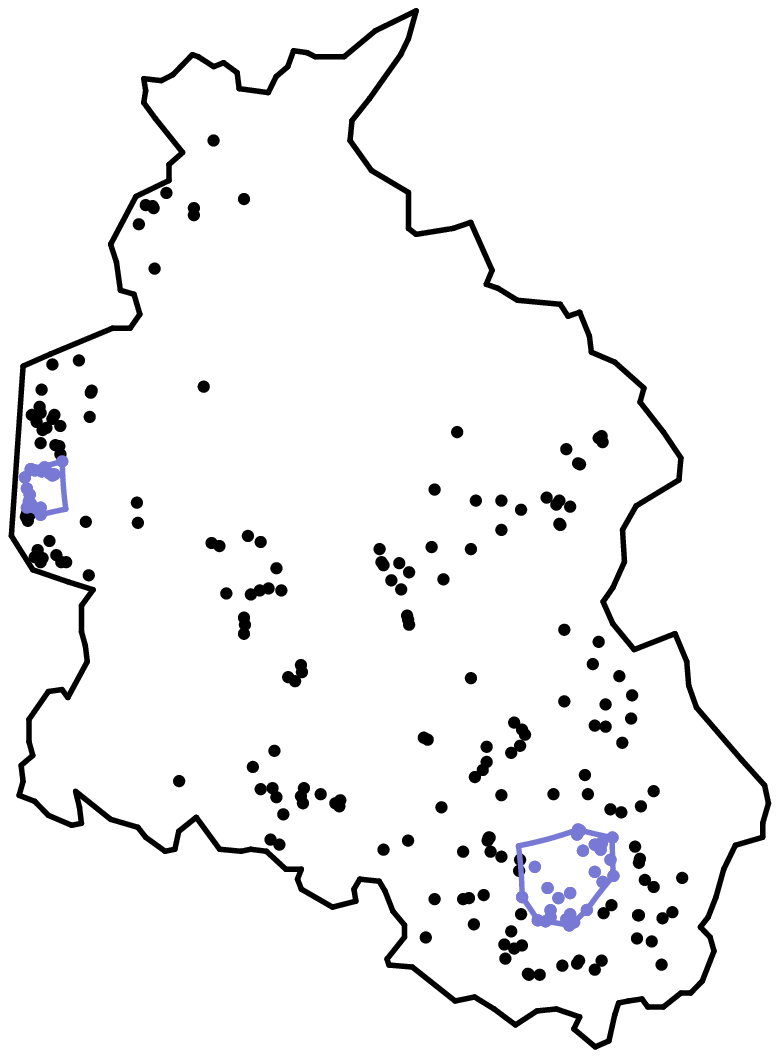}}
\put(157,-311){\includegraphics[height=12cm, width=.7\textwidth]{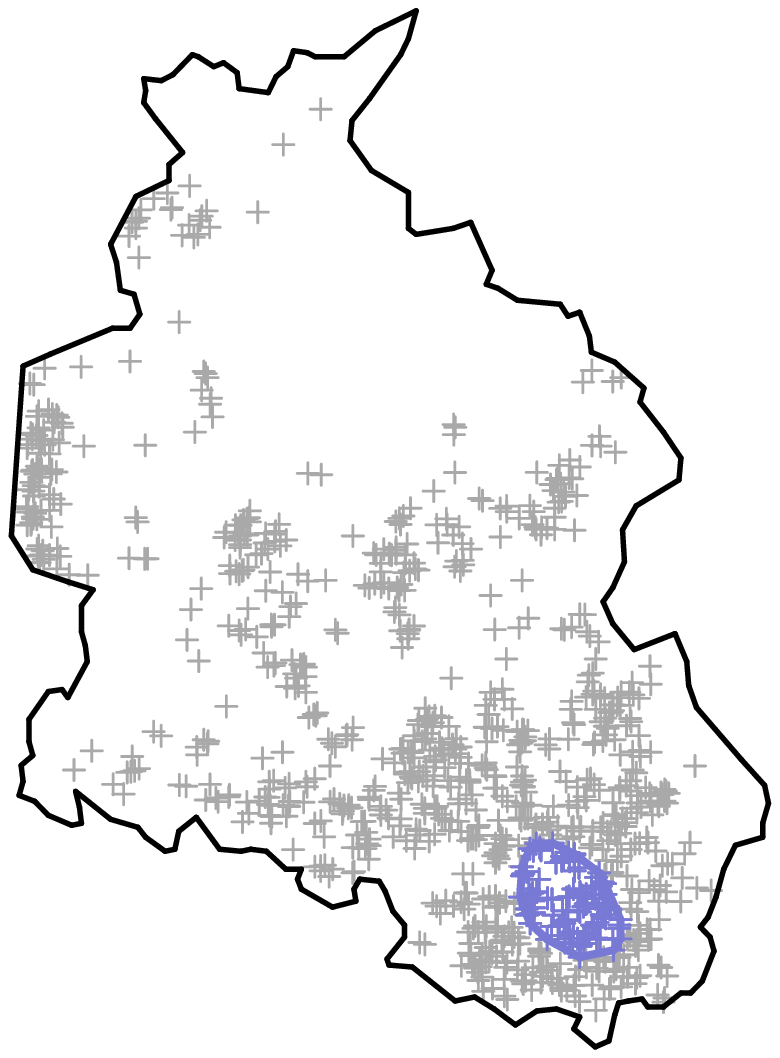}}
}\end{picture}
\vspace{10.cm}\caption{In the first column, estimated level sets for the sample of 322 cases diagnosed of leukaemia on the North West of England with $\tau=0.8$ (top) and $\tau=0.85$ (bottom). In the second column, estimated level sets for the sample of 988 controls of leukaemia on the North West of England with $\tau=0.8$ (top) and $\tau=0.85$ (bottom).}\label{hhhhhiuiop}
\end{figure}

\newpage

\begin{figure}[h!]\hspace{1.1cm}\vspace{-3cm}
\begin{picture}(-85,260)
\scalebox{0.8}[0.8]{
\put(-49,-56){\includegraphics[height=12 cm, width=.7\textwidth]{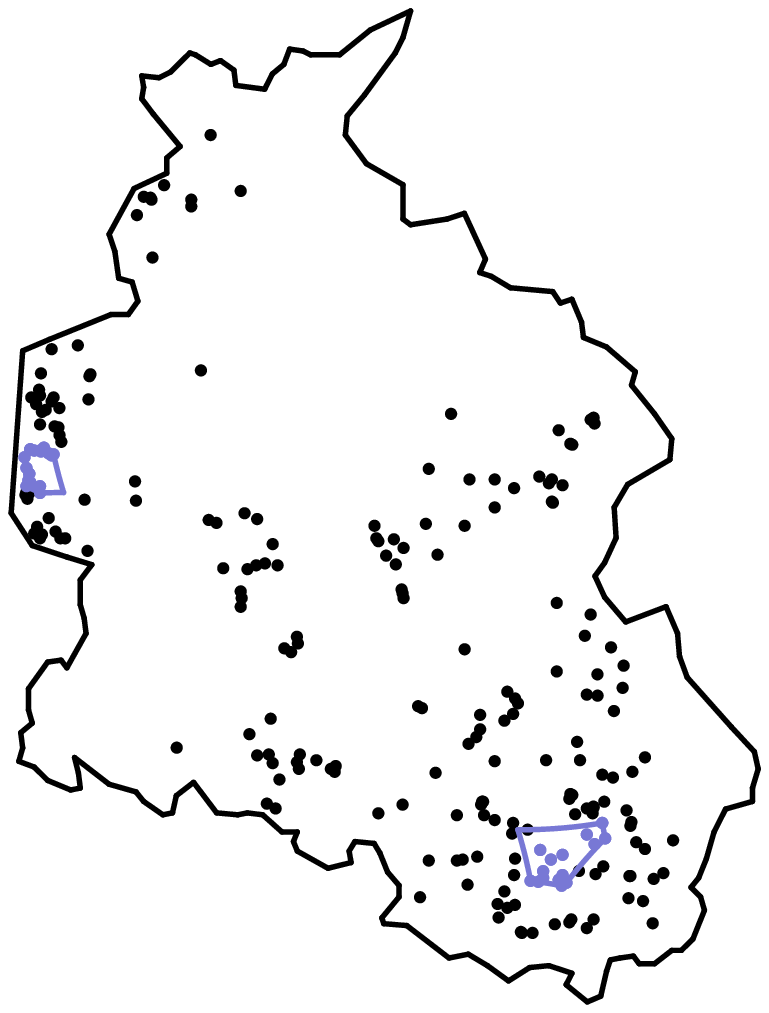}}
\put(157,-56){\includegraphics[height=12cm, width=.7\textwidth]{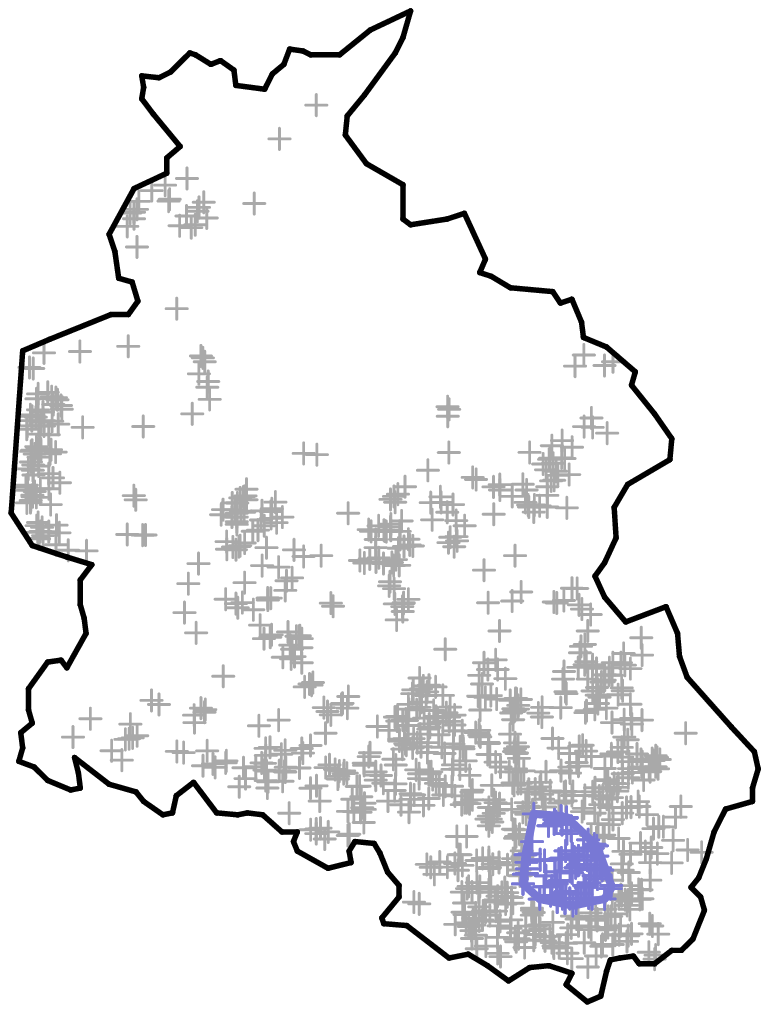}}
\put(-49,-311){\includegraphics[height=12 cm, width=.7\textwidth]{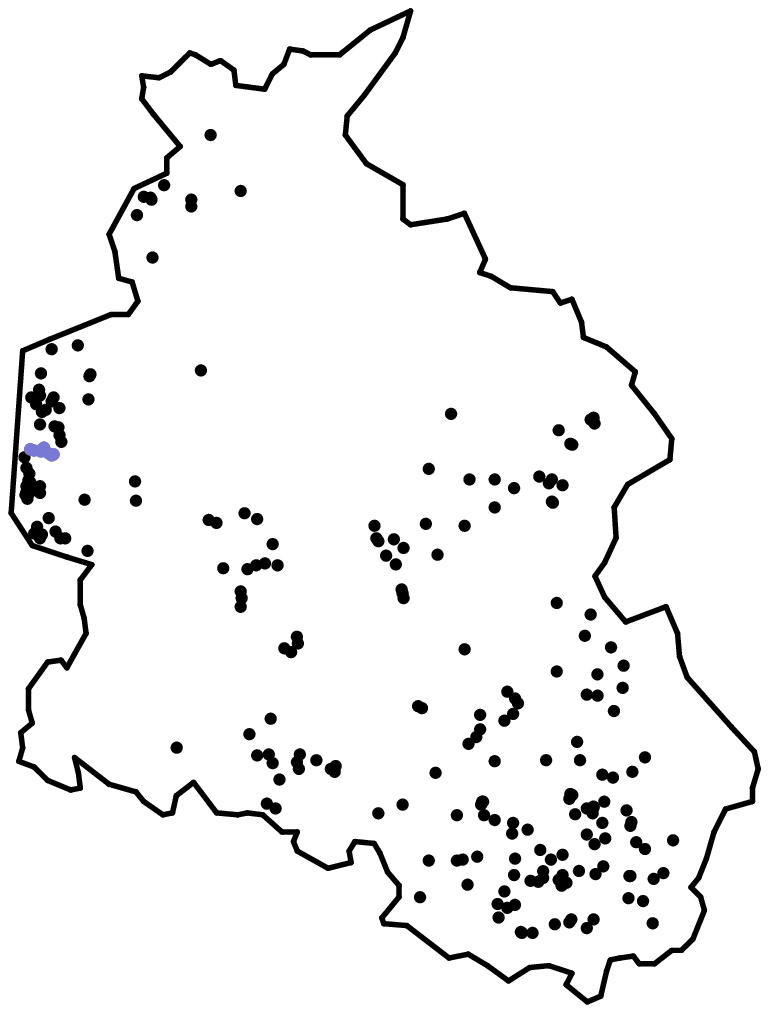}}
\put(157,-311){\includegraphics[height=12cm, width=.7\textwidth]{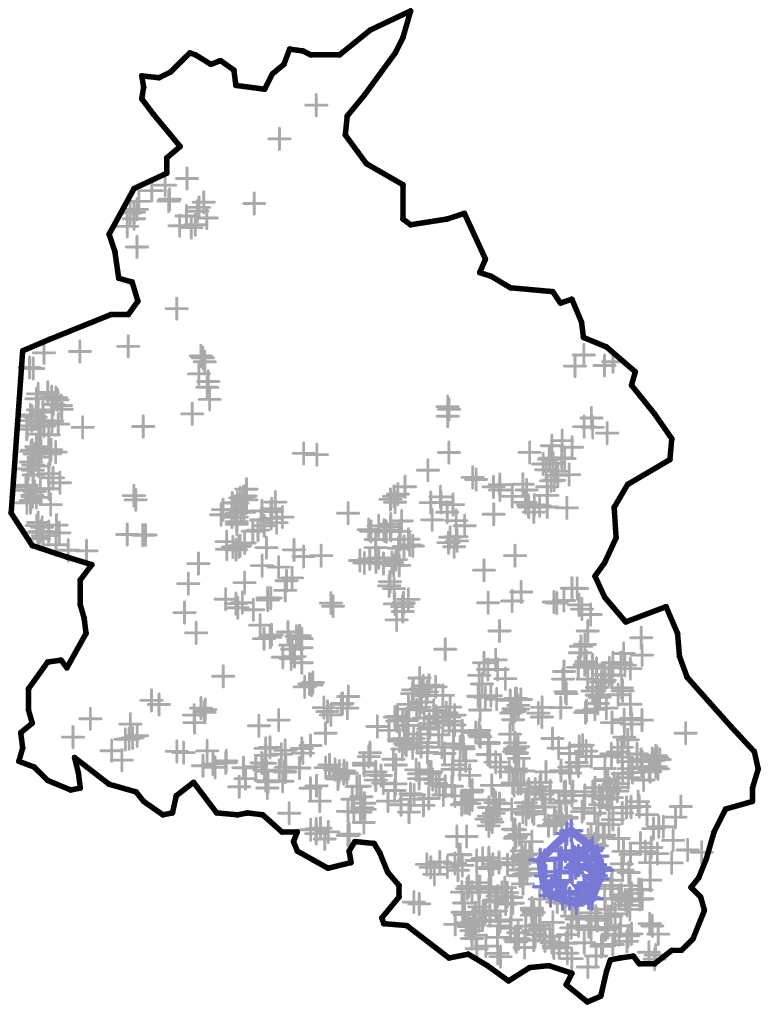}}
\pscircle[fillcolor=white,linecolor=black,linewidth=0.2mm ](1.,-4.05){.15}
 \psline[fillcolor=white,linecolor=black,linewidth=0.2mm ](1.15,-3.99)(4.7,-1.8)
\pscircle[fillcolor=white,linecolor=black,linewidth=0.2mm ](5.9,-1.1){1.4}
\put(120,-90){\includegraphics[height=3.7cm, width=0.23\textwidth]{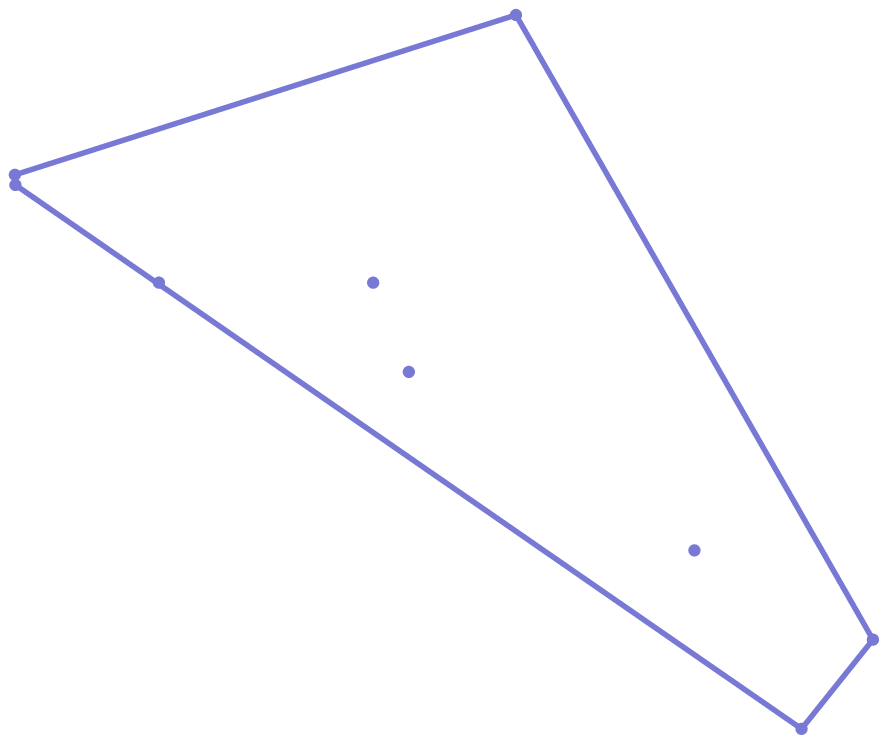}}
}\end{picture}
\vspace{10.cm}\caption{In the first column, estimated level sets for the sample of 322 cases diagnosed of leukaemia on the North West of England with $\tau=0.9$ (top) and $\tau=0.95$ (bottom). In the second column, estimated level sets for the sample of 988 controls of leukaemia on the North West of England with $\tau=0.9$ (top) and $\tau=0.95$ (bottom).}\label{hhhhhiuiop7}
\end{figure}
\newpage $\vspace{-.25cm} $\\the sample of controls are convex for the three largest values of $\tau$ considered, $0.85$, $0.9$ and $0.95$. In addition, $\hat{\tau}^+$ and $\hat{\tau}^-$ are usually different. Only for the controls with $\tau=0.8$, it is verified that $\hat{\tau}^+=\hat{\tau}^-=0.8$. The performance of the estimations for the parameter $k$ is not too clear. In particular, for the cases, $\hat{k}$ takes the values $1$, $3$ and $5$. However, it is always equal to $1$ for the sample of controls.

The resulting level sets are showed for the two samples on North West of England in Figures \ref{hhhhhiuiop1}, \ref{hhhhhiuiop} and \ref{hhhhhiuiop7} for different values of the probability content $\tau$. It is possible to observe an excess of case intensity over that of population. Greater Manchester is a metropolitan county in North West England that encompasses one of the largest metropolitan areas in the United Kingdom. However, Lancashire is a non-metropolitan county that emerged during the Industrial Revolution as a major commercial and industrial region. Therefore, there is evidence of clustering and the leukaemia cases could be related to environmental and industrial factors. Similar studies have been already considered in literature. For instance, see Cuzick and Edwards (1990) for the childhood leukaemia in Humberside, Diggle et al. (1990) for the lung and larynx cancers in Chorley-South Ribble, Kelsall and Diggle (1998) for the lung and stomach cancer in Walsall, Kelsall and Wakefield (2000) for the colorectal cancer in Birminghan or Henderson et al. (2002) for acute myeloid leukemia in North West of England.
\newpage
\section{Proofs}\label{prprpr}In this section the proofs of the stated propositions and theorems are presented.\vspace{.15cm}\\
\emph{Proof of Proposition \ref{mostramaisenGlambda}.}\vspace{.15cm}\\
First, we will prove that,
 $$\mathbb{P}(\mathcal{X}_n^+(t)\subset G(t),\mbox{ eventually} )=1.$$For this, it is enough to prove
\begin{equation}\label{jo}
    \mathbb{P}\left(\sup_{z\in G(t)^c} f_n(z)<t+M\left(\frac{\log{n}}{n}\right)^{p/(d+2p)} ,\mbox{ eventually}\right)=1.
\end{equation}Then, let $z\in G(t)^c$ and $C$ be the compact set defined in Proposition \ref{enPruebaTeorema3Walther}. Two cases are considered: $z\in C$ or $z\in C^c$.
\begin{enumerate}
\item Let $z\in C^c$. Since $z\in G(t)^c$ then $z\notin G(l)$ because $G(l)\setminus \interior(G(u))\subset C$. Therefore, according to Proposition \ref{enPruebaTeorema3Walther3}, with probability one and for $n$ large enough,
$$f_n(z)\leq \sup_{y\in G(l)^c\cap C^c}f_n(y)<l-\frac{w}{2}<l,$$where $w$ denotes a positive constant. Therefore,
$$\mathbb{P}\left(\sup_{z\in G(t)^c\cap C^c}f_n(z)<l, \mbox{ eventually}\right)=1,$$and since $l<t+D_n$ for all $t\in (l,u)$,
$$\mathbb{P}\left(\sup_{z\in G(t)^c\cap C^c}f_n(z)<t+D_n, \mbox{ eventually}\right)=1.$$
\item Let $z\in C$. According to Proposition \ref{enPruebaTeorema3Walther}, we can guarantee that,
$$\sup_{C}|f_n-f|=O\left(  \left(  \frac{\log{n}}{n}\right)^{p/(d+2p)}\right),\mbox{ almost surely}.$$So, there exists $N>0$ such that
\begin{equation}\label{hhhhhhhhhhhhhhhhhhhhhhhh}
\sup_{C}|f_n-f|\leq N  \left(  \frac{\log{n}}{n}\right)^{p/(d+2p)},\mbox{ almost surely}.
\end{equation}Since $z\notin G(t)$ then $f(z)<t$. Taking into account (\ref{hhhhhhhhhhhhhhhhhhhhhhhh}), for $n$ large enough, it is verified that
$$f_n(z)\leq|f_n(z)-f(z)|+f(z)<|f_n(z)-f(z)|+t\leq N\left(  \frac{\log{n}}{n}\right)^{p/(d+2p)}+t.$$If $M\geq N$,
$$f_n(z) < t+M\left(  \frac{\log{n}}{n}\right)^{p/(d+2p)}=t +D_n,\mbox{ almost surely}.$$
\end{enumerate}This concludes the proof of (\ref{jo}).\\Similarly, we will prove that,
 $$\mathbb{P}(\mathcal{X}_n^-(t)\subset G(t)^c ,\mbox{ eventually})=1.$$For this, it is enough to prove
\begin{equation}\label{jo2}
    \mathbb{P}\left(\inf_{z\in G(t)} f_n(z)\geq t-M\left(\frac{\log{n}}{n}\right)^{p/(d+2p)} ,\mbox{ eventually}\right)=1.
\end{equation}Let $z\in G(t)$. Again, two cases are considered: $z\in C$ or $z\in C^c$.
\begin{enumerate}

\item Let $z\in C^c$. Then, $z\notin (G(l)\setminus \interior(G(u)))$. But $z\in G(t)\subset G(l)$. Therefore, $z\in \interior(G(u))$ and, as consequence, $f(z)>u$. According to Proposition \ref{enPruebaTeorema3Walther3}, with probability one,
$$f_n(z)\geq \inf_{y\in G(u)\cap C^c}f_n(y)>u+\frac{w}{2}>u,$$where $w$ denotes a positive constant. Therefore,
$$\mathbb{P}\left(\inf_{z\in G(t)\cap C^c}f_n(z)>u ,\mbox{ eventually}\right)=1,$$and since $t-D_n<u$ for all $t\in(l,u)$,
$$\mathbb{P}\left(\inf_{z\in G(t)\cap C^c}f_n(z)>t-D_n ,\mbox{ eventually}\right)=1.$$
\item Let $z\in C$. Since $z\in G(t)$ then $f(z)\geq t$. Taking into account (\ref{hhhhhhhhhhhhhhhhhhhhhhhh}),
$$f_n(z)\geq f(z)-|f_n(z)-f(z)| \geq t-|f_n(z)-f(z)|$$
$$\geq t-N\left(  \frac{\log{n}}{n}\right)^{p/(d+2p)},\mbox{ almost surely}.$$If $M\geq N$,
$$f_n(z) \geq t-M\left(  \frac{\log{n}}{n}\right)^{p/(d+2p)}=t+D_n,\mbox{ almost surely}.$$
\end{enumerate}This concludes the proof of (\ref{jo2}). The lemma is a straightforward consequence of (\ref{jo}) and (\ref{jo2}).\hfill $\Box$\vspace{.15cm}\\
\emph{Proof of Proposition \ref{alberto5}.}\vspace{.15cm}\\Let $\epsilon>0$. It is clear that it is enough to show the result for a value of $\epsilon$ small enough. The followings steps complete the proof:
\begin{enumerate}
\item Let $x\in G(t)$. Under (A), a ball of radius $m/k$ rolls freely in $G(t)$ and $\overline{G(t)^c}$. According to Lemma 1 in Arias-Castro and Rodr\'iguez-Casal (2014), if $\epsilon\leq m/k$,
$$\exists B_{\frac{\epsilon}{2}}(y)\subset B_{\epsilon}(x)  \mbox{ such that } B_{\frac{\epsilon}{2}}(y)\subset G(t).$$We define $B_t^x=B_{\epsilon/4}(y)$. Obviously, $B_t^x\subset G(t) $. In addition, it verifies that
$$B_t^x\subset G(t)\ominus \frac{\epsilon}{4}B_1[0]$$since, for all $z\in B_t^x$, $z+(\epsilon/4)B_1[0]\subset B_{\epsilon/2}(y)\subset G(t)$. On the other hand, considering Proposition \ref{enPruebaTeorema3Walther4} for $\epsilon$ small enough and $T=\epsilon m/8$,
$$G(t)\ominus \frac{\epsilon}{4}B_1[0] \subset G(t+T).$$Therefore, see Figure \ref{k},
$$B_t^x\subset G(t)\ominus \frac{\epsilon}{4}B_1[0]\subset G(t+T).$$
\begin{figure}[h!]\centering
\begin{pspicture}(1,0)(14,5.9)
\psccurve[showpoints=false,fillstyle=solid,fillcolor=white,linecolor=black,linewidth=0.2mm,linearc=3](5,1)(5,4)(8,3.5)(10,4.5)(10,1)
\rput(0.87,.3){\scalebox{0.9}[0.85]{\psccurve[showpoints=false,fillstyle=solid,fillcolor=white,linecolor=gray,linewidth=0.2mm,linearc=3](5,1)(5,4)(8,3.5)(10,4.5)(10,1)}}
\pscircle[ linecolor=black,linewidth=0.2mm,linearc=3,linestyle=dashed,dash=3pt 2pt](9.7,4.3){.9}
\psdots*[dotsize=2.5pt](9.7,4.3)
\rput(9.7,4.83){$\tiny{B_\epsilon(x)}$}
\rput(4.4,1){$G(t)$}
\rput(6.1,3.3){\textcolor[rgb]{0.52,0.52,0.52}{$G(t+T)$}}
 \pscircle[linearc=0.25,linecolor=gray,linewidth=0.2mm, linestyle=dashed,dash=3pt 2pt,fillstyle=crosshatch*,fillcolor=gray,hatchcolor=white,hatchwidth=1.pt,hatchsep=.4pt,hatchangle=0](9.7,3.8){.25}
\psdots[dotsize=2.pt,linecolor=gray](9.7,3.8)
\end{pspicture}
\caption{Elements in proof of Proposition \ref{alberto5}. $G(t)$ in black, $G(t+T)$ in gray, $B_\epsilon(x)$ in black and $B_t^x$ in gray.}\label{k}
\label{Figuraaa}
\end{figure}
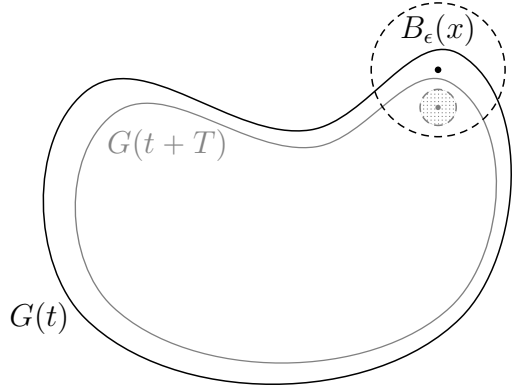

\item Let $\mathcal{F}=\{B_t^x: \mbox{ }x\in G(t)\}$. Under (A), the level set $G(t)$ is bounded since $G(u+\zeta)$ is bounded and $\overline{G(l-\zeta)}\setminus \interior (G(u+\zeta))\subset U$ where $U$ is a bounded set too. As consequence, $\overline{G(l-\zeta)}$ is bounded and, therefore, $G(t)\subset G(l)\subset\overline{G(l-\zeta)}$ too. Then, there exists a finite cover for $G(t)$ of balls of radius, for instance, $\epsilon/10$. Therefore, there exists $z_1,\cdots,z_s\in G(t)$ such that
$$G(t)\subset \bigcup_{i=1}^s B_{\frac{\epsilon}{10}}(z_i).$$Then, for all $B_t^x=B_{\epsilon/4}(y)\in \mathcal{F}$ where $y\in G(t) $,
$$\exists z_j\in \{z_1,\cdots,z_s\}\mbox{ such that }\|z_j-y\|<\frac{\epsilon}{10}.$$Next, we will prove that the ball $B_{\epsilon/10}(z_j)\subset B_t^x$. Let $z\in B_{\epsilon/10}(z_j)$,
$$\|z-y\|\leq\| z-z_j\|+\|z_j-y\| < \frac{\epsilon}{10}+ \frac{\epsilon}{10}=\frac{\epsilon}{5}<\frac{\epsilon}{4}.$$
As consequence, if a ball in $\mathcal{F}$ does not meet $\mathcal{X}_n$ then there exists a ball $ B_{\frac{\epsilon}{10}}(z_i)$ with $ z_i\in \{z_1,\cdots,z_s\}$ such that $B_{\frac{\epsilon}{10}}(z_i)\cap\mathcal{X}_n=\emptyset$. So,
\begin{equation}\label{hu245}\mathbb{P}(\exists  x\in G(t):\mathcal{X}_n\cap B_t^x=\emptyset)\leq \sum_{i=1}^s\mathbb{P}\left(\mathcal{X}_n\cap B_{\frac{\epsilon}{10}}(z_i)=\emptyset\right).\end{equation}
In addition, if $\epsilon$ is small enough then
\begin{equation}\label{hu}
f(z)>l-\zeta\mbox{ for all }z\in  B_{\frac{\epsilon}{10}}(z_i),\mbox{ }i=1,\cdots,s.
\end{equation}Let $z\in B_{\frac{\epsilon}{10}}(z_i)$ for some $i\in\{1,\cdots,s\}$. Since $z_i\in G(t)$ then $f(z_i)\geq t >l$. In addition, $f$ is continuous in $U$. Then, two cases must be considered:
\begin{enumerate}
\item If $z_i\in U$ then given $\zeta>0$, see assumption (A) for more details,
$$\exists \delta_i>0 \mbox{ such that }\forall w\in B_{\delta_i}(z_i)\mbox{ it is verified that }\|f(w)-f(z_i)\|<\zeta.$$Then,
$$\exists \delta_i>0 \mbox{ such that }\forall w\in B_{\delta_i}(z_i)\mbox{ it is verified that }f(w)> l-\zeta.$$
\item If $z_i\notin U$ then $z_i\in \interior(G(u+\zeta))$ since $z_i\in G(l)\cap U^c$. Therefore,
$$\exists \delta_i>0 \mbox{ such that } B_{\delta_i}(z_i)\subset \interior(G(u+\zeta)).$$Then, for all $w\in B_{\delta_i}(z_i)$, $f(w)>u+\zeta>l-\zeta$.
\end{enumerate}In order to guarantee (\ref{hu}), it is enough to take $\epsilon<10\min\{\delta_i:i=1,\cdots,s\}$.
\item Next, using (\ref{hu245}), we will prove that
$$\mathbb{P}(\exists  x\in G(t):\mathcal{X}_n\cap B_t^x=\emptyset,\mbox{ infinitely often})=0.$$Using the same reasoning as in the Step 2, it is enough to analyze if for each fixed $i\in\{1,\cdots,s\}$ $$\mathbb{P}\left(\mathcal{X}_n\cap B_{\frac{\epsilon}{10}}(z_i)=\emptyset,\mbox{ infinitely often}\right)=0.$$According to Borel-Cantelli's Lemmas, it is enough to show that $$\sum_{n=1}^\infty \mathbb{P}\left(\mathcal{X}_n\cap B_{\frac{\epsilon}{10}}(z_i)=\emptyset\right)<\infty.$$Since the observations are independent and identically distributed, we can write
$$\mathbb{P}\left(\mathcal{X}_n\cap B_{\frac{\epsilon}{10}}(z_i)=\emptyset\right)=\mathbb{P}\left(\forall i \in\{1,...,n\},\mbox{ } X_i\notin B_{\frac{\epsilon}{10}}(z_i)\right)$$
$$=\prod_{i=1}^n\mathbb{P}\left( X_i\notin B_{\frac{\epsilon}{10}}(z_i)\right)=\left[\mathbb{P}\left(X_1\notin B_{\frac{\epsilon}{10}}(z_i)\right)\right]^n=\left[1-\mathbb{P}\left(X_1\in B_{\frac{\epsilon}{10}}(z_i)\right)\right]^n$$
$$\leq  e^{-n\mathbb{P}\left(X_1\in B_{\frac{\epsilon}{10}}(z_i)\right)}.$$According to (\ref{hu}),
$$\mathbb{P}\left(X_1\in B_{\frac{\epsilon}{10}}(z_i) \right) =\int_{B_{\frac{\epsilon}{10}}(z_i)} f(x)\mbox{ }d\mu\geq \int_{B_{\frac{\epsilon}{10}}(z_i)} (l-\zeta)\mbox{ }d\mu=\rho>0$$and
$$\mathbb{P}\left(\mathcal{X}_n\cap B_{\frac{\epsilon}{10}}(z_i)=\emptyset\right)\leq  e^{-n\mathbb{P}\left(X_1\in B_{\frac{\epsilon}{10}}(z_i)\right)}= e^{-n \rho}.$$
Then,
$$\sum_{n=1}^\infty \mathbb{P}\left(\mathcal{X}_n\cap B_{\frac{\epsilon}{10}}(z_i)=\emptyset\right)\leq \sum_{n=1}^\infty e^{-n \rho}<\infty.$$
\item According to Step 3, with probability one, there exists $n_0$ such that for all $x\in G(t)$,
$$ \mathcal{X}_n\cap B_{t}^x\neq \emptyset,\mbox{ }\forall n\geq n_0.$$Then, there exists $n_0$ such that for all $x\in G(t)$,
$$ \exists X_{i_x}\in \mathcal{X}_n\cap B_{t}^x \subset \mathcal{X}_n\cap B_\epsilon(x),\mbox{ }\forall n\geq n_0.$$Therefore, it only remains to prove that $X_{i_x}\in \mathcal{X}_n^+(t)$. According to Proposition \ref{enPruebaTeorema3Walther}, it is possible to guarantee that
$$\sup_{C}|f_n-f| =O\left(  \left(  \frac{\log{n}}{n}\right)^{p/(d+2p)}\right),\mbox{ almost surely}$$where $C\subset U$ is under conditions of Proposition \ref{enPruebaTeorema3Walther}. Therefore, there exists $N>0$ such that, with probability one,
$$
\sup_{C}|f_n-f|\leq N  \left(  \frac{\log{n}}{n}\right)^{p/(d+2p)} .
$$Two cases are considered: $X_{i_x}\in C$ and $X_{i_x}\notin C$.
\begin{enumerate}
  \item If $X_{i_x}\in C$ and $D_n= M\left(  \frac{\log{n}}{n}\right)^{p/(d+2p)}$ with $M\geq N$ then $\lim_{n\rightarrow \infty}D_n=0$. So, fixed $T/2>0$ (see Step 1 in this proof),
    $$\exists n_1\in\mathbb{N}\mbox{ such that } D_n<T/2, \forall n\geq n_1.$$Then,
    $$|f_n(X_{i_x})-f(X_{i_x})|\leq \sup_{C}|f_n-f|\leq D_n<T/2,\mbox{ }\forall n\geq\{n_0,n_1\}.$$Therefore, since $X_{i_x}\in B_t^x\subset G(t+T)$,
    $$f_n(X_{i_x})\geq f(X_{i_x})-D_n\geq t+T-D_n>t+T-\frac{T}{2}=t+\frac{T}{2}\geq t+D_n.$$

  \item If $X_{i_x}\notin C$ then, since $X_{i_x} \in B_t^x\subset G(t+T)$, it is verified that $f(X_{i_x})\geq t+T>t\geq l$. So, $X_{i_x} \in \interior(G(l))$. Then,
   $X_{i_x}\in G(u)\cap C^c$. According to Proposition \ref{enPruebaTeorema3Walther3} for a certain $w>0$, with probability one,
  $$\exists n_2 \mbox{ such that }f_n(z)\geq u +\frac{w}{2},\mbox{ }\forall z\in G(u)\cap C^c\mbox{ and }\forall n\geq n_2.$$For $D_n$ fixed previously, $\lim_{n\rightarrow \infty}D_n=0$. So, given $w/2>0$,$$\exists n_3\in\mathbb{N}\mbox{ such that } D_n<w/2, \forall n\geq n_3.$$Therefore, since $t\leq u$,
  \[f_n(X_{i_x})\geq u+\frac{w}{2}\geq t + D_n,  \mbox{ }\forall n\geq\max\{n_0,n_2,n_3\}.\]\hfill $\Box$
\end{enumerate}
\end{enumerate}\vspace{.15cm}
\emph{Proof of Corollary \ref{xeneralizacionsoporte}.}\vspace{.15cm}\\
The proof is a straightforward consequence of Proposition \ref{mostramaisenGlambda} and Proposition \ref{alberto5}.\hfill $\Box$\vspace{.15cm}\\
\emph{Proof of Theorem \ref{rmaiorr02}.}\vspace{.15cm}\\Some auxiliary results are necessary. Lemma \ref{mostmenos} is a useful and auxiliary tool for guaranteeing the consistency for the estimator established in Definition \ref{jejejeje}. It ensures the existence of points in $\mathcal{X}_n^-(t)$ inside any open ball contained in $G(t)^c$. A straightforward consequence is that, with probability one and for $n$ large enough, $\mathcal{X}_n^-(t)$ is not empty.

\begin{lemma}\label{mostmenos}Let $G(t)$ be a compact, nonempty and nonconvex. Under assumptions (A), (D) and (K), let $\mathcal{X}_n$ be a random sample generated from a distribution with density function $f$ and $\mathcal{X}_n^-(t)$ established in Definition \ref{jejejeje}. Let $B_\epsilon(x)$ such that $B_\epsilon(x)\subset \interior (G(l-\zeta))$ and $B_\epsilon(x)\cap G(t)=\emptyset$. Then,
$$\mathbb{P}\left(\mathcal{X}_n^-(t)\cap B_\epsilon(x)\neq \emptyset,\mbox{ eventually} \right)=1.$$

\end{lemma}

\begin{proof}
 Since $x\in G(t)^c \cap \interior(G(l-\zeta))$, it is verified that $l-\zeta<f(x)<t$. The following steps complete the proof:
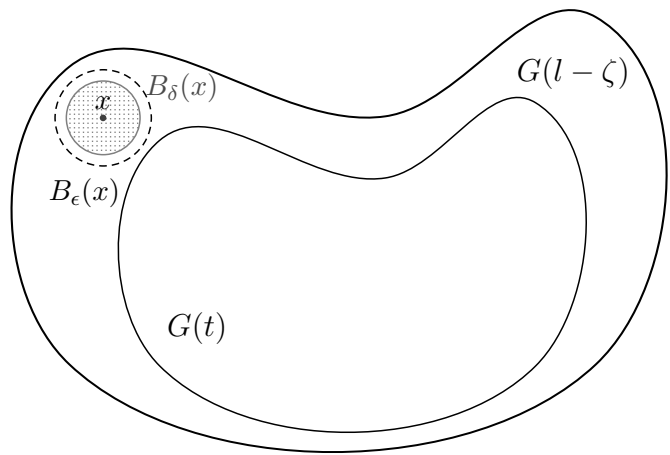
\begin{figure}[h!]\centering
\begin{pspicture}(1,-0.5)(14,7)
\rput(-3.2,-.3){\scalebox{1.4}[1.32]{\psccurve[showpoints=false,fillstyle=solid,fillcolor=white,linecolor=black,linewidth=0.2mm,linearc=3](5,1)(5,4)(8,3.5)(10,4.5)(10,1)}}
\psccurve[showpoints=false,fillstyle=solid,fillcolor=white,linecolor=black,linewidth=0.2mm,linearc=3](5,1)(5,4)(8,3.5)(10,4.5)(10,1)
\rput(5.5,1.5){$G(t)$}
\rput(10.5,4.9){$G(l-\zeta)$}
\pscircle[ linecolor=black,linewidth=0.2mm,linearc=3,linestyle=dashed,dash=3pt 2pt](4.25,4.3){.65}
\pscircle[linearc=0.25,linecolor=gray,linewidth=0.2mm,linestyle=solid,dash=3pt 2pt,fillstyle=crosshatch*,fillcolor=gray,hatchcolor=white,hatchwidth=1.2pt,hatchsep=.5pt,hatchangle=0](4.25,4.3){.5}
\psdots*[dotsize=2.5pt,linecolor=darkgray](4.25,4.3)
\rput(4.25,4.5){\small{$x$}}
\rput(4.,3.3){\small{$B_\epsilon(x)$}}
\rput(5.3,4.7){\small{\textcolor[rgb]{0.29,0.29,0.29}{$B_\delta(x)$}}}
\end{pspicture}
\caption{Elements of Lemma \ref{mostmenos}. $B_\delta(x)\subset B_\epsilon(x)\subset \interior(G(l-\zeta))\cap G(t)^c$.}
\label{Figuraaa7}
\end{figure}
\begin{enumerate}
  \item  Under (A), $f$ is continuous in $x$. Therefore, given $K=\frac{t-f(x)}{2}>0$,
 $$\exists \delta_1>0\mbox{ such that } \forall y\in B_{\delta_1}(x) \mbox{ it is verified that }\|f(x)-f(y)\|<K.$$Since $f(x)=t-2K$,
     $$ \forall y\in B_{\delta_1}(x) \mbox{ it is verified that }f(y)< t-K.$$In addition, $B_\epsilon(x)\subset \interior(G(t-\zeta))$. Therefore,
      $$  \forall y\in B_{\epsilon}(x) \mbox{ it is verified that }f(y)> t-\zeta>0.$$If $\delta=\min\{\delta_1,\epsilon\}$ then it is verified that $l-\zeta<f(y)<t-K$ for all $y\in B_\delta(x)\subset B_\epsilon(x)$. See Figure \ref{Figuraaa7} for more details.
  \item Next, we will prove that, with probability one and for $n$ large enough, there exists $X_{i_x}\in \mathcal{X}_n\cap B_{ \delta }(x) $. That is, we will prove that
  $\mathbb{P}(\mathcal{X}_n\cap B_{ \delta }(x)\neq\emptyset, \mbox{eventually})=1$. According to the Borel-Cantelli's Lemmas, it is enough to prove that $\sum_{n=1}^\infty \mathbb{P}(\mathcal{X}_n\cap B_{ \delta }(x)=\emptyset)<\infty$. Since the observations are independent and identically distributed, we can write
$$\mathbb{P}(\mathcal{X}_n\cap B_{ \delta }(x)=\emptyset)=\mathbb{P}(\forall i \in\{1,...,n\},\mbox{ } X_i\notin B_{ \delta }(x))$$
$$=\prod_{i=1}^n\mathbb{P}(X_i\notin B_{ \delta }(x))=\left[\mathbb{P}(X_1\notin B_{ \delta }(x))\right]^n=\left[1-\mathbb{P}(X_1\in B_{ \delta }(x))\right]^n$$
$$\leq  e^{-n\mathbb{P}(X_1\in B_{ \delta }(x))}.$$According to the previous step, $\forall y\in B_{ \delta }(x) $ it is verified that $f(y)>t-\zeta>0$. Therefore,
$$\mathbb{P}(X_1\in B_{ \delta }(x))\mbox{ }=\int_{B_{ \delta }(x)} f(x)\mbox{ }d\mu\geq \int_{B_{ \delta }(x)} l-\zeta\mbox{ }d\mu$$
$$=(l-\zeta)\mu(B_{ \delta }(x)).$$So,
$$\mathbb{P}(\mathcal{X}_n\cap B_{ \delta }(x)=\emptyset)\leq  e^{-n\mathbb{P}(X_1\in B_{ \delta }(x))}\leq e^{-n(l-\zeta)\mu(B_{ \delta }(x))}>0.$$
Then,
$$\sum_{n=1}^\infty \mathbb{P}(\mathcal{X}_n\cap B_{ \delta }(x)=\emptyset)\leq \sum_{n=1}^\infty e^{-n (l-\zeta) \mu(B_{ \delta }(x))}<\infty.$$
In addition, for all $X_{i_x}\in \mathcal{X}_n\cap B_{ \delta }(x) $ it is satisfied that $f(X_{i_x})<t-K$ con $K>0$ and $f(X_{i_x})>l-\zeta$ (see Step 1 of this proof). It remains to show that $X_{i_x}\in \mathcal{X}_n^-(t)$.

\item According to the previous step, with probability one, there exists $n_0$ such that
$$ \mathcal{X}_n\cap B_{\delta}(x)\neq \emptyset,\mbox{ }\forall n\geq n_0.$$According to Proposition \ref{enPruebaTeorema3Walther},
$$\sup_{C}|f_n-f|=O\left(  \left(  \frac{\log{n}}{n}\right)^{p/(d+2p)}\right),$$ where $C\subset U$ is under conditions of Proposition \ref{enPruebaTeorema3Walther}. Therefore, with probability one and for $n$ large enough,
$$
\exists N>0 \mbox{ such that }\sup_{C}|f_n-f|\leq N  \left(  \frac{\log{n}}{n}\right)^{p/(d+2p)} .
$$Two situations are considered: $X_{i_x}\in C$ and $X_{i_x}\notin C$.

\begin{enumerate}
  \item If $X_{i_x}\in C$ and $D_n=M  \left(  \frac{\log{n}}{n}\right)^{p/(d+2p)}$ with $M\geq N$ then $\lim_{n\rightarrow \infty}D_n=0$. So, fixed $K/2>0$ (see Step 1 in this proof),
    $$\exists n_1\in\mathbb{N}\mbox{ such that } D_n<K/2, \forall n\geq n_1.$$Then, with probability one,
    $$|f_n(X_{i_x})-f(X_{i_x})|\leq \sup_{C}|f_n-f|\leq D_n<K/2,\mbox{ }\forall n\geq\max\{n_0,n_1\}.$$Therefore, for all $ n\geq\max\{n_0,n_1\}$,
   $$f_n(X_{i_x})\leq f(X_{i_x})+D_n<t-K+D_n<t-K+\frac{K}{2}=t-\frac{K}{2}<t-D_n.$$

  \item If $X_{i_x}\notin C$ then, since $f(x)<t-K<t\leq u$, it is verified that $x \in G(u)^c$. Without losing generality, we can assume that $B_{ \delta }(x)\subset G(u)^c$ since $G(u)^c$ is open and $x$ is a interior point. In another case, it is enough to reduce the radius of the ball.
    So, $X_{i_x}\in G(l)^c\cap C^c$. According to Proposition \ref{enPruebaTeorema3Walther3} for some $w>0$, with probability one,
  $$\exists n_1 \mbox{ such that }f_n(z)\leq l -\frac{w}{2},\mbox{ }\forall z\in G(l)^c\cap C^c\mbox{ and }\forall n\geq n_1.$$For $D_n$ previously fixed, $\lim_{n\rightarrow \infty}D_n=0$. Therefore, fixed $w/2>0$,$$\exists n_2\in\mathbb{N}\mbox{ such that } D_n<w/2, \forall n\geq n_2.$$Therefore, since $l\leq t$ and $X_{i_x}\in G(l)^c\cap C^c$,
\[f_n(X_{i_x})\leq l -\frac{w}{2}\leq t -D_n,  \mbox{ }\forall n\geq\max\{n_0,n_1,n_2\}.\qedhere\]

\end{enumerate}
\end{enumerate}

\end{proof}

Lemma \ref{rooooj} proves that, with probability one and for $n$ large enough, the estimator $\hat{r}_0(t)$ is greater than or equal to $r_0(t)$.

\begin{lemma}\label{rooooj}Let $G(t)$ be a compact, nonempty and nonconvex level set. Under assumptions (A), (D) and (K), let $\mathcal{X}_n$ be a random sample generated from a distribution with density function $f$, $r_0(t)$ and $\hat{r}_0(t)$ established in Definitions \ref{r_0_t} and \ref{jejejeje}, respectively. Then,
$$\mathbb{P}(\hat{r}_0(t)\geq r_0(t),\mbox{ eventually})=1.$$\end{lemma}

\begin{proof}According to Proposition \ref{mostramaisenGlambda},
$$\exists n_1\in\mathbb{N}\mbox{ such that }\mathcal{X}_n^+(t)\subset G(t) \mbox{ and }\mathcal{X}_n^-(t)\subset G(t)^c,\mbox{ }\forall n\geq n_1.$$Since $G(t)$ is $r_0(t)-$convex, it is verified that
$$C_{r_0(t)}(\mathcal{X}_n^+(t))\subset C_{r_0(t)}(G(t))=G(t).$$Therefore, since $\mathcal{X}_n^-(t)\subset G(t)^c$,
\[\hat{r}_0(t)= \sup\{\gamma>0: C_\gamma(\mathcal{X}_n^+(t))\cap \mathcal{X}_n^-(t)=\emptyset\} \geq r_0(t),\mbox{ }\forall n\geq n_1.\qedhere\]
\qedhere\end{proof}

Lemma \ref{noais} guarantees a reasonable topological behaviour of sets under rolling freely condition.

\begin{lemma}\label{noais}Let $A\subset \mathbb{R}^{d}$ be a nonempty and closed set. If a ball of radius $\lambda$ rolls freely in $A$ then
$$\interior(\overline{A^c})=A^c\mbox{ and }\partial A=\partial \overline{A^c}.\vspace{1.4mm}$$
\end{lemma}

\begin{proof}First, we will prove that $\interior(\overline{A^c})=A^c$.
Since that $A^c$ is open and $A^c\subset \overline{A^c}$ then $A^c\subset \interior(\overline{A^c})$. Next, it will be proved that $ \interior(\overline{A^c})\subset A^c$. Let us suppose the contrary, that is, there exists $x\in \interior(\overline{A^c})$ such that $x\in A$. Then, $x\in A\cap\overline{A^c}=\partial A$. Rolling freely in $A$ guarantees that there exists $p\in  A$ such that $x\in B_\lambda[p]\subset A$ with $\|x-p\|=\lambda$. Since that $x\in \interior(\overline{A^c})$, there exists $\epsilon >0$ such that $B_\epsilon[x]\subset \overline{A^c}$. Let us assume that $\epsilon<\lambda$ and let us consider the point
$$y_\tau=x+\tau\frac{p-x}{\|p-x\|},\mbox{ }\tau\in(0,\epsilon).\vspace{1.4mm}$$Then, $y_\tau\in B_\lambda(p)\subset \interior(A)$. So, a contradiction is obtained since $y_\tau\in B_\epsilon[x]\subset \overline{A^c}$.\\Proving $\partial A=\partial \overline{A^c}$ is easy because the boundary of a set can be written as the closure minus the interior. In addition, $A$ is closed and $\interior(\overline{A^c})=A^c$. So,
\[\partial \overline{A^c}=\overline{\overline{A^c}}\setminus \interior(\overline{A^c})=\overline{A^c}\setminus A^c=\overline{A^c}\setminus \interior( A^c)=\partial A^c=\partial A.\qedhere\]\end{proof}

The balls of radius $r$ and radius $\lambda$ that roll freely in $\overline{G(t)^c}$ and $G(t)$, respectively, under ($R_{\lambda}^r$) have been characterized, see Lemma \ref{olvido2}. It is easy to prove too that there exists $x_t\in \interior(C_\gamma(G(t)))\cap \partial G(t)$ for all $\gamma>0$ such that $G(t)\varsubsetneq C_\gamma(G(t))$, see Lemma \ref{puntointerior2}.

\begin{lemma}\label{olvido2}Let $G(t)\subset\mathbb{R}^{d}$ be a closed level set verifying ($R_\lambda^r$). Then, for each $x_t\in \partial G(t)$ there exists a unique unit vector $\eta(x_t)$ such that
$$B_{\lambda}(x_t-\lambda \eta(x_t))\subset G(t)\mbox{ and }B_{r}(x_t+r \eta(x_t))\subset \overline{G(t)^c}.\vspace{2mm}$$
\end{lemma}

\begin{proof}Let $x_t\in \partial G(t)$. Under ($R_\lambda^r$), a ball of radius $\lambda$ rolls freely in $G(t)$. Then,

$$\exists x\in G(t) \mbox{ such that }B_{\lambda}(x)\subset G(t).\vspace{2.mm}$$In addition, it is possible to write (see Lemma 7.1 in Rodr\'iguez-Casal and Saavedra-Nieves (2014))

$$x=x_t-\lambda\eta(x_t)\mbox{ with }\eta(x_t)=\frac{x_t-x}{\|x_t-x\|}.\vspace{2.mm}$$According to Lemma \ref{noais}, $\partial A=\partial \overline{A^c}$ and, so, $a\in \partial \overline{A^c}$. Under ($R_\lambda^r$), it is verified that a ball of radius $r$ rolls freely in $\overline{A^c}$. Then,

$$\exists y\in \overline{A^c}\mbox{ such that }B_{r}(y)\subset \overline{A^c}\vspace{2.mm}$$verifying that $\|y-a\|=d(y,A)=r$. So, $a$ is metric projection of a point $y\notin A$. According to the Lemma 7.2 in Rodr\'iguez-Casal and Saavedra-Nieves (2014),

\[y =a+r\eta(a)\mbox{ and then }B_{r}(a+r \eta(a))\subset \overline{A^c}.\qedhere\]\end{proof}

\begin{lemma}\label{puntointerior2}Let $G(t)\subset \mathbb{R}^{d}$ be a compact and nonempty level set verifying ($R_{\lambda}^r$) and $\gamma>0$ such that $G(t)\varsubsetneq C_\gamma(G(t))$. Then, there exists $x_t\in \interior(C_\gamma(G(t)))\cap \partial G(t)$.
\end{lemma}
\begin{proof}The proof can be obtained from Lemma 7.4 in Rodr\'iguez-Casal and Saavedra-Nieves (2014).
\end{proof}

It remains to prove that $\hat{r}_0(t)$ can not be arbitrarily larger than $r_0(t)$. This is established in Lemma \ref{rmaiorr0}. In order to prove consistency, see Lemma \ref{rmaiorr02}.
\begin{lemma}\label{rmaiorr0}Let $G(t)$ be a compact, nonempty and nonconvex level set. Under assumptions (A), (D) and (K), let $\mathcal{X}_n$ be a random sample generated from a distribution with density function $f$, $r_0(t)$ and $\hat{r}_0(t)$ established in Definitions \ref{r_0_t} and \ref{jejejeje}, respectively. Then, for any $\epsilon>0$
$$\mathbb{P}(\hat{r}_0(t)\leq r_0(t)+\epsilon,\mbox{ eventually})=1.$$
\end{lemma}

\begin{proof}Given $\epsilon>0$, let be $r=r_0(t)+\epsilon>r_0(t)$. Let be $r^{'}$ such that $r>r^{'}>r_0(t)$. The proof is split in several steps:

\begin{enumerate}
  \item First, we will prove that there exists $B_\gamma(x)$ verifying that
  \begin{equation}\label{hyuiop}
    B_\gamma(x)\subset C_{r^{'}}(G(t))\cap G(l-\zeta)\mbox{ and }B_\gamma(x)\cap G(t)=\emptyset.
  \end{equation}According to Proposition \ref{puntointerior2}, there exists $x_t\in \interior ( C_{r^{'}}(G(t)) )\cap \partial G(t)$:
\begin{enumerate}
       \item Then, there exists $\gamma_1>0$ such that $B_{\gamma_1}(x_t)\subset  C_{r^{'}}(G(t)) $.
       \item Since $x_t\in \partial G(t)$, $f(x_t)=t>l>l-\zeta$. Therefore, $x_t\in \interior (G(l-\zeta))$. As consequence, there exists $\gamma_2>0$ such that $B_{\gamma_2}(x_t)\subset  \interior ( G(l-\zeta))$.
           \item In addition, a ball of radius $m/k$ rolls freely in $\overline{G(t)^c}$. Then, there exists $y\in G(t)^c$ such that $x_t\in B_{m/k}[y]$ with $B_{m/k}(y)\cap G(t)=\emptyset$.
                           \end{enumerate}
 We fixed $0<\gamma\leq\min\{\gamma_1,\mbox{ }\gamma_2,\mbox{ }m/k\}/2$ and $x=x_t+\gamma\eta(x_t)$, see Figure \ref{oosssokkkdkd} and Lemma \ref{olvido2} for remember details about the vector $\eta(x_t)$. For this $\gamma$, $B_\gamma(x)$ satisfies (\ref{hyuiop}). In addition, notice that we can assume that, without loss of generality, $r\leq r^{'}+\gamma/2$. Otherwise, if $r-r^{'}> \gamma/2$, we could select $r^*=r^{'}+\gamma/2<r$ verifying $r^*>r^{'}>r_0(t)$. For this $r^*$, (\ref{hyuiop}) is still satisfied.
\begin{figure*}[h!]\vspace{-1cm}
\hspace{1cm}\begin{pspicture}(-1.7,-1.9)(10,7.5)
\rput(-1.2,-1.23){\scalebox{1.3}[1.5]{\psccurve[showpoints=false,fillstyle=solid,fillcolor=white,linecolor=black,linewidth=0.3mm,linearc=3](1.5,0)(1,2.5)(1.5,5)(4,4.85)(7,4)(6.2,-0.05)(4,0.1)}}
\psccurve[showpoints=false,fillstyle=solid,fillcolor=white,linecolor=gray,linewidth=.75mm,linearc=3](1.5,0)(1,2.5)(1.5,5)(4,4.85)(7,4)(6.2,-0.05)(5,1.5)
\psccurve[showpoints=false,fillstyle=solid,fillcolor=white,linecolor=white,linewidth=.85mm,linearc=3](1.5,0)(6.2,-0.05)(5,1.5)
\psarc[showpoints=false,fillstyle=solid,fillcolor=white,linecolor=gray,linewidth=0.45mm,linearc=3](6.108,.7351) {.8}{264}{280}
\psarc[showpoints=false,fillstyle=solid,fillcolor=red,linecolor=gray,linewidth=0.45mm,linearc=3](1.87,0.805) {.9}{240}{281}
\psarc[showpoints=false,fillstyle=solid,fillcolor=white,linecolor=gray,linewidth=0.45mm,linearc=3](4.1,-9.867) {10}{78.9}{102}
\psccurve[showpoints=false,fillstyle=solid,fillcolor=white,linecolor=black,linewidth=0.3mm,linearc=3](1.5,0)(1,2.5)(1.5,5)(4,4.85)(7,4)(6.2,-0.05)(5,1.5)

\rput(1.65,1){\small{$B_{\gamma_2}(x_t)$}}
\rput(9.5,3){\small{$G(l-\zeta)$}}
\pscircle[linearc=0.25,linecolor=black,linewidth=0.2mm,linestyle=dashed,dash=3pt 3pt](4.8,1.53){2.6}
\pscircle[linearc=0.25,linecolor=black,linewidth=0.2mm,linestyle=dashed,dash=3pt 3pt](4.8,1.53){1.}
 \pscircle[linearc=0.25,linecolor=black,linewidth=0.2mm,linestyle=solid,dash=3pt 2pt](4.79,0.935){.6}
 \pscircle[linearc=0.25,linecolor=gray,linewidth=0.2mm,linestyle=solid,dash=3pt 2pt,fillstyle=crosshatch*,fillcolor=gray,hatchcolor=white,hatchwidth=1.2pt,hatchsep=.5pt,hatchangle=0](4.8,1.03){.47}
 \rput(4.8,1.685){\small{$x_t$}}
 \rput(1.7,4){$G(t)$}
\rput(6.5,5){$C_{r^{'}}(G(t))$}
\rput(4.8,2.83){\small{$B_{\gamma_1}(x_t)$}}
\rput(3.74,0.46){\small{$B_{m/k}[y]$}}
\psdots*[dotsize=2pt](4.8,1.53)
\end{pspicture}\caption{Elements of proof in Theorem \ref{rmaiorr0}. $B_\gamma[x]$ in gray color.}\label{oosssokkkdkd}
\end{figure*}
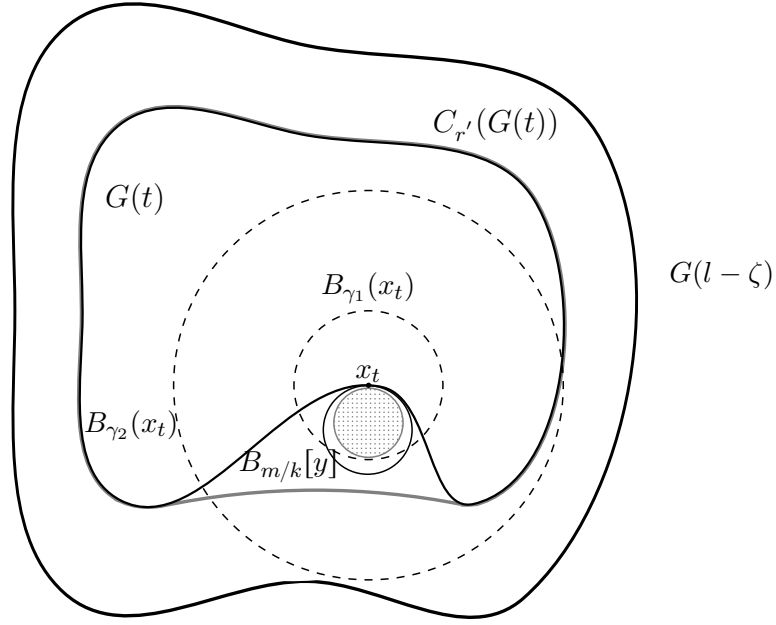
  \item According to Proposition \ref{alberto5}, with probability one and for $n$ large enough, $$G(t)\subset \mathcal{X}_n^+(t)\oplus  B_{r-r^{'}}[0].$$Then, with probability one and for $n$ large enough, it is verified that
  $$G(t)\oplus B_{r^{'}}[0]\subset \left(\mathcal{X}_n^+(t)\oplus  B_{r-r^{'}}[0]\right)\oplus B_{r^{'}}[0].$$Therefore, with probability one and for $n$ large enough,
 $$C_{r^{'}}(G(t))=(G(t)\oplus B_{r^{'}}[0])\ominus B_{r^{'}}[0]\subset \left[\left(\mathcal{X}_n^+(t)\oplus  B_{r-r^{'}}[0]\right)\oplus B_{r^{'}}[0]\right]\ominus B_{r^{'}}[0]$$or equivalently,
 $$C_{r^{'}}(G(t))=(G(t)\oplus B_{r^{'}}[0])\ominus B_{r^{'}}[0]\subset \left[\left(\mathcal{X}_n^+(t)\oplus B_r[0]\right)\right]\ominus B_{r^{'}}[0].$$Then,
 $$C_{r^{'}}(G(t)) \ominus B_{r-r^{'}}[0]\subset \left[\left(\mathcal{X}_n^+(t)\oplus B_r[0]\right)\ominus B_{r^{'}}[0]\right]\ominus B_{r-r^{'}}[0]=C_r(\mathcal{X}_n^+(t)).$$Since $B_\gamma(x)\subset C_{r^{'}}(G(t))$, if $r-r^{'}\leq \gamma/2$ then
 $$B_{\frac{\gamma}{2}}(x)\subset C_{r^{'}}(G(t))\ominus B_{r-r^{'}}[0]\subset C_{r}(\mathcal{X}_n^+(t)).$$
 \item According to Lemma \ref{mostmenos}, with probability one and for $n$ large enough,\\$ \mathcal{X}_n^-(t) \cap B_{\gamma/2}(x)\neq \emptyset$ and, hence, $ \mathcal{X}_n^-(t) \cap C_r(\mathcal{X}_n^+(t))\neq \emptyset$. Therefore, we can conclude that $\hat{r}_0(t)\leq r$.\qedhere
\end{enumerate}
 \end{proof}
The proof is a straightforward consequence of Lemmas \ref{rooooj} and \ref{rmaiorr0}.\hfill $\Box$\vspace{.15cm}\\
\emph{Proof of Proposition \ref{llllll}.}\vspace{.15cm}\\According to Proposition \ref{mostramaisenGlambda}, with probability one,
$$\exists n_1\in\mathbb{N} \mbox{ such that }\mathcal{X}_n^+(t)\subset G(t),\mbox{ }\forall n\geq n_1.$$Since $ r_n(t)$ converges to $\nu r_0(t)$, almost surely, we have that, with probability one,
$$\exists n_2\in\mathbb{N} \mbox{ such that }r_n(t)\leq r_0(t),\mbox{ }\forall n\geq n_2.$$If $n\geq \max\{n_1,n_2\}$,
\[C_{r_n(t)}(\mathcal{X}_n^+(t))\subset C_{r_0(t)}(\mathcal{X}_n^+(t)) \subset G(t). \]\hfill $\Box$\vspace{.15cm}\\
\emph{Proof of Theorem \ref{principal}.}\vspace{.15cm}\\It is necessary to introduce some auxiliary sets in order to obtain the convergence rates of the resulting estimator for the level set, see Definitions \ref{auxiset}, \ref{auximuestra} and Figure \ref{oookffgg}. Really, these new sets are subsets of the original level set $G(t)$ and the sample $\mathcal{X}_n$, respectively. Notice that both are defined from the theoretical density function $f$. The kernel estimator $f_n$ is not considered. On the other hand and although they depend on some parameters like $n$, this fact is not reflected in their names for simplicity in the exposition.

\begin{definition}\label{auxiset}Let $G(t)$ be a compact, nonempty and nonconvex level set. Under assumptions (A) and (D), the set $G^+(t)\subset \mathbb{R}^d$ is defined as the level set with threshold equal to $t+2D_n$. That is, $G^+(t)=G(t+2D_n)$.
\begin{figure}[h!]\centering
\begin{pspicture}(-.0,1)(10,8)
\rput(.1,1.6){\scalebox{1.2}[1.1]{\psccurve[showpoints=false,fillstyle=solid,fillcolor=white,linecolor=black,linewidth=0.25mm,linearc=3](1.5,0)(1.9,2.5)(1.5,5)(4,4.35)(7,4)(6.2,.3)(5,1.5)}}
\rput(1.35,2.39){\scalebox{0.95}[.8]{\psccurve[showpoints=false,fillstyle=solid,fillcolor=white,linecolor=gray,linewidth=0.25mm,linearc=3,fillstyle=crosshatch*,fillcolor=gray,hatchcolor=white,hatchwidth=1.8pt,hatchsep=1pt,hatchangle=0](1.5,0)(1.9,2.5)(1.5,5)(4,4.35)(7,4)(6.2,.3)(5,1.5)}}
\rput(2.3,6.8){$G(t)$}
\rput(7.5,5.00){$G^+(t)$}
\end{pspicture}\caption{Sets $G^+(t)$ and $G(t)$ in Definition \ref{auxiset}.}\label{oookffgg}
\end{figure}
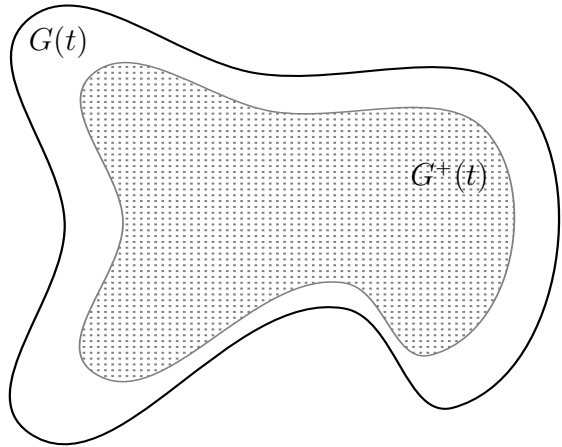
\end{definition}

\begin{definition}\label{auximuestra}Let $G(t)$ be a compact, nonempty and nonconvex level set. Under assumptions (A) and (D), let $\mathcal{X}_n$ be a random sample generated from a distribution with density function $f$ and let $G^+(t)\subset \mathbb{R}^d$ be the level set established in Definition \ref{auxiset}. The set $\mathcal{X}_n^{G^+}$ is defined by $\mathcal{X}_n\cap G^+(t)$. Therefore, it can be written as $\mathcal{X}_n^{G^+}=\{X_i\in\mathcal{X}_n:f(X_i)\geq t+2D_n\}$.
\end{definition}

A new class of sets is presented in Definition \ref{waltherG}. This family was already considered in Walther (1997).

\begin{definition}\label{waltherG}Let $A\subset \mathbb{R}^d$ be a set and $\gamma>0$. Then, $\mathcal{G}_{A}(\gamma)$ denotes all sets $B$ that verify ($R_{\gamma}^\gamma$) satisfying $B\subset A$.
\end{definition}

The smoothing parameter established in Definition \ref{r_0_t} is studied in Lemma \ref{debemosescribirlo5} for the sets $G^+(t)$.

\begin{lemma}\label{debemosescribirlo5}Under assumptions (A) and (D), let $r_0(t)$ established in Definition \ref{r_0_t}. It is verified that
$$\exists n_0\in\mathbb{N}\mbox{ such that }  r_0(t+2D_n) \geq m/k,\mbox{ }\forall n\geq n_0.$$

\end{lemma}

\begin{proof}
Since $\lim_{n\rightarrow\infty}D_n=0$,
$$\exists n_0\in\mathbb{N} \mbox{ such that }2D_n< u-t,\mbox{ }\forall n\geq n_0.$$
Therefore,
$$l< t+2D_n< u,\mbox{ }\forall n\geq n_0.$$\\
Under (A), $G^+(t)$ verifies that a ball of radius $m/k$ rolls freely in $G^+(t)$ and ${G^+(t)}^c$ for $n\geq n_0$. Therefore, \[0<m/k\leq r_0(t+2D_n), \mbox{ }\forall n\geq n_0.\qedhere\]\end{proof}

Next, it will be proved that $G^+(t)\in\mathcal{G}_{G(l)}(r_\nu)$ for $n$ large enough and $r_\nu>0$, see Lemma \ref{debemosescribirlo55} for details about the positive constant $r_\nu$.

\begin{lemma}\label{debemosescribirlo55}Let $G(t)$ be a compact, nonempty and nonconvex level set. Let $G^+(t)$ be the set established in Definition \ref{auxiset}. Under assumptions (A), (D) and (K), let $\hat{r}_0(t)$ established in Definition \ref{jejejeje}, $\nu\in(0,1)$ be a fixed number and $r_n(t)=\nu\hat{r}_0(t)$. Then, there exists $0<r_\nu<m/k$ such that
$$\mathbb{P}(r_n(t)>r_\nu,\mbox{ eventually})=1.$$Further,
$$\exists n_0\in\mathbb{N}\mbox{ such that } G^+(t)\in\mathcal{G}_{G(t)}(r_\nu),\mbox{ }\forall n\geq n_0$$and, therefore,
$$G^+(t)\in\mathcal{G}_{G(l)}(r_\nu),\mbox{ }\forall n\geq n_0$$for $\mathcal{G}_{G(t)}(r_\nu)$, $\mathcal{G}_{G(l)}(r_\nu)$ and $G^+(t)$ established in Definitions \ref{waltherG} and \ref{auxiset}, respectively.
\end{lemma}

\begin{proof}Let be $r_\nu>0$ verifying that $r_\nu<\nu (m/k)<m/k$. It is easy to prove that $\mathbb{P}(r_n(t)>r_\nu,\mbox{ eventually})=1$ taking into account that $ r_n(t)$ converges to $\nu r_0(t)$, almost surely. According to Lemma \ref{debemosescribirlo5}, it is verified that
$$\exists n_0\in\mathbb{N}\mbox{ such that } r_0(t+2D_n)\geq m/k,\mbox{ }\forall n\geq n_0.$$Since $0<r_\nu<m/k$,
$$r_0(t+2D_n)\geq m/k>r_\nu,\mbox{ }\forall n\geq n_0.$$
Then, since $G^+(t)=G(t+2D_n)\subset G(t)$ for all $n$,
\[G^+(t) \in\mathcal{G}_{G(t)}(r_\nu),\mbox{ }\forall n\geq n_0.\qedhere\]

 $ $

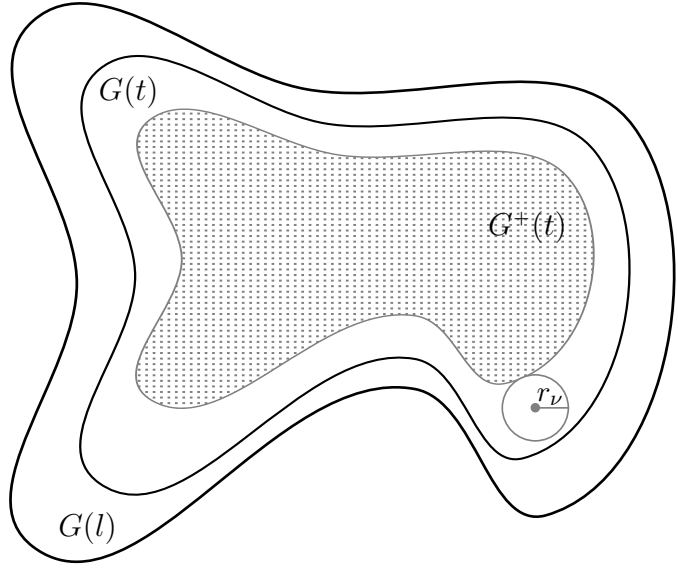
\begin{figure}[h!]\vspace{-1cm}\centering
\begin{pspicture}(1,-0.5)(10,8.5)
\rput(-1.15,0.75){\scalebox{1.45}[1.4]{\psccurve[showpoints=false,fillstyle=solid,fillcolor=white,linecolor=black,linewidth=0.25mm,linearc=3](1.5,0)(1.9,2.5)(1.5,5)(4,4.35)(7,4)(6.2,.3)(5,1.5)}}
\rput(.1,1.6){\scalebox{1.2}[1.1]{\psccurve[showpoints=false,fillstyle=solid,fillcolor=white,linecolor=black,linewidth=0.25mm,linearc=3](1.5,0)(1.9,2.5)(1.5,5)(4,4.35)(7,4)(6.2,.3)(5,1.5)}}
\rput(1.1,2.7){\scalebox{1}[.75]{\psccurve[showpoints=false,fillstyle=solid,fillcolor=white,linecolor=gray,linewidth=0.25mm,linearc=3,fillstyle=crosshatch*,fillcolor=gray,hatchcolor=white,hatchwidth=1.8pt,hatchsep=1pt,hatchangle=0](1.5,0)(1.9,2.5)(1.5,5)(4,4.35)(7,4)(6.2,.3)(5,1.5)}}
\rput(2.3,6.8){$G(t)$}
\rput(7.6,5.00){$G^+(t)$}
\rput(1.75,1.){$G(l)$}
\pscircle[showpoints=false,fillstyle=solid,fillcolor=white,linecolor=gray,linewidth=0.22mm,linearc=3 ](7.7,2.6){.45}
\psdots[linecolor=gray](7.7,2.6)
\psline[linecolor=gray,linewidth=0.22mm](7.7,2.6)(8.15,2.6)
\rput(7.9,2.75){\small{$r_\nu$}}
\end{pspicture}\vspace{-.3cm}\caption{Elements in Lemma \ref{debemosescribirlo55}. $G(t)$, $G(l)$ and $G^+(t)$. A ball of radius $r_\nu$ (gray color) rolls freely in $\overline{{G^+(t)}^c}$.}\label{oookffghhhg}
\end{figure}
\end{proof}

In Lemma \ref{dddd}, it will proved that, given the threshold $t$, the set $\mathcal{X}_{n}^{G^+} $ is eventually contained in $\mathcal{X}_{n}^+(t)$.

\begin{lemma}\label{dddd}Let $G(t)$ be a compact, nonempty and nonconvex level set. Under assumptions (A), (D) and (K), let $\mathcal{X}_n$ be a random sample generated from a distribution with density function $f$, let $\mathcal{X}_n^+(t)$ be established in Definition \ref{jejejeje} and let $\mathcal{X}_{n}^{G^+}$ be the subsample defined in Definition \ref{auximuestra}. Then,
$$\mathbb{P}(\mathcal{X}_{n}^{G^+}\subset \mathcal{X}_{n}^+(t),   \mbox{ eventually})=1.$$
\end{lemma}

\begin{proof}Let $X_i\in \mathcal{X}_{n}^{G^+}$. Therefore, $f(X_i)\geq t+2D_n$. According to Proposition \ref{enPruebaTeorema3Walther},
$$\sup_{C}|f_n-f|=O\left(  \left(  \frac{\log{n}}{n}\right)^{p/(d+2p)}\right), \mbox{ almost surely.}$$where $C\subset U$ is under conditions of Proposition \ref{enPruebaTeorema3Walther}. So, with probability one and for $n$ large enough,
\begin{equation}\label{gt}
 \exists N>0 \mbox{ such that }\sup_{C}|f_n-f|\leq N  \left(  \frac{\log{n}}{n}\right)^{p/(d+2p)} .
\end{equation}Two cases are considered: $X_i$ belongs to $C$ or $X_i$ does not belong to $C$.
\begin{enumerate}
  \item  Let $X_i\in C $. According to (\ref{gt}), if $M\geq N$,
  $$|f_n(X_i)-f(X_i)|\leq D_n .$$Therefore,
  $$f_n(X_i)\geq f(X_i)-D_n\geq t +2D_n -D_n=t+D_n.$$

  \item If $X_i\notin C $ then $X_i\in G(u)\cap C^c$ since $X_i\in G(l)$ and $G(l)\setminus \interior(G(u))\subset C$. According to Proposition \ref{enPruebaTeorema3Walther3}, with probability one and for $n$ large enough, $f_n(X_i)\geq u+ v/2$ for some $v>0$. In addition, since $D_n$ converges to zero,
      $$\exists n_0\in\mathbb{N}\mbox{ such that }2D_n< \frac{v}{2},\mbox{ }\forall n\geq n_0.$$Then, with probability one and for $n$ large enough,
\[f_n(X_i)\geq u+ \frac{v}{2}\geq t+\frac{v}{2}\geq t+D_n. \qedhere\]
\end{enumerate}\end{proof}

According to Lemma \ref{dddd}, it is verified that $C_{r_n(t)}(\mathcal{X}_{n}^{G^+})\subset C_{r_n(t)}(\mathcal{X}_{n}^+(t)$). That is, $\mathcal{X}_{n}^+(t)$ is at least as good as $\mathcal{X}_{n}^{G^+}$ in order to estimate $G(t)$. Remember that  $\mathcal{X}_{n}^{G^+}$ is constructed from $f$. It does not depend on the kernel estimator $f_n$. In addition, $\mathcal{X}_{n}^{G^+}$ would be the natural sample for estimating $G^+(t)$. Then, let $r_\nu$ be a positive constant under the conditions in Lemma \ref{debemosescribirlo55} and let $r>0$ such that $0<r\leq r_\nu$. Let $\epsilon_n=\left(\frac{C\log{n}}{n}\right)^{\frac{2}{d+1}}$ where $C>0$ denotes a big enough constant to be established later. Since $\lim_{n\rightarrow\infty}\epsilon_n=0$,
$$\exists n_0\in\mathbb{N}\mbox{ such that }0<\epsilon_n<\max\left\{\frac{r}{3},1\right\},\mbox{ }\forall n\geq n_0.$$On the other hand,
$$\exists n_1\in\mathbb{N}\mbox{ such that }r-2\epsilon_n\geq r/2,\mbox{ }\forall n\geq n_1.$$
According to Proposition \ref{enPruebaTeorema3Walther5}, if $f\geq b>0$,
$$\mathbb{P}(A\oplus B_{r-3\epsilon_n}[0]\nsubset \left[\left(A\cap\mathcal{X}_n\right)\oplus B_r[0]\right]\mbox{ for some }A\in \mathcal{G}_{G(l)}(r_\nu))$$
$$\leq D(\epsilon_n, G(l)\oplus B_r[0])D\left(\frac{\epsilon_n}{10r},S^{d-1}\right)\exp\left\{-nab(r-2\epsilon_n)^\frac{d-1}{2}(\epsilon_n/2)^{\frac{d+1}{2}}\right\},$$where $D(\epsilon,B)=\max\{card\mbox{ } V:V\subset B,\mbox{ }|x-y|>\epsilon\mbox{ for different }x,y\in V\}$, $S^{d-1}$ denotes the unit sphere in $\mathbb{R}^d$ and $a$ is a dimensional constant. Therefore, if $n\geq \max\{n_0,n_1\}$ then
$r-2\epsilon_n\geq r/2$ and
$$\mathbb{P}(A\oplus B_{r-3\epsilon_n}[0]\nsubset \left[\left(A\cap\mathcal{X}_n\right)\oplus B_r[0]\right]\mbox{ for some }A\in \mathcal{G}_{G(l)}(r_\nu))$$
$$\leq Q \epsilon_n^{-d} \epsilon_n^{-(d-1)}\exp\left\{-nab\left(\frac{r}{2}\right)^{\frac{d-1}{2}}\left(\frac{C\log{n}}{2^{(d+1)/2}n}\right)\right\}=Q\epsilon_n^{(-2d+1)}\exp\{-W\log{n}\}$$with $Q$ is a constant depending on $r$ and the dimension $d$ and $W=\frac{ab}{2^{(d+1)/2}}\left(\frac{r}{2}\right)^{\frac{d-1}{2}}C.$ If $C$ tends to infinite then $W$ tends to it too. Then, given $Q>0$
$$\exists n_2\in\mathbb{N}\mbox{ such that }\exp\{-W\log{n}\}\leq Q,\mbox{ }\forall n\geq n_2.$$Therefore,
$$\mathbb{P}(A\oplus(r-3\epsilon_n)B_1[0]\nsubset \left[\left(A\cap\mathcal{X}_n\right)\oplus B_r[0]\right]\mbox{ for some }A\in \mathcal{G}_{G(l)}(r_\nu))$$
$$\leq Q^2\epsilon_n^{-(2d-1)}\leq Q^2\left(\frac{n}{\log{n}}\right)^{\frac{(2d-1)(d+1)}{2}}n^{-M},\mbox{ }\forall n\geq\max\{n_0,n_1,n_2\}.$$If $W>\frac{(2d-1)(d+1)}{2}$ it is verified that
$$\sum_{i=1}^\infty \left(\frac{n}{\log{n}}\right)^{\frac{(2d-1)(d+1)}{2}}n^{-W}<\infty.$$So,
$$\mathbb{P}(A\oplus B_{r-3\epsilon_n}[0]\nsubset \left[\left(A\cap\mathcal{X}_n\right)\oplus B_r[0]\right]\mbox{ for some }A\in \mathcal{G}_{G(l)}(r_\nu),\mbox{ infinitely often})=0.$$
Then, with probability one,
$$\exists n_3\in\mathbb{N}\mbox{ such that }A\oplus B_{r-3\epsilon_n}[0]\subset  \left(A\cap\mathcal{X}_n\right)\oplus B_r[0] ,\mbox{ }\forall A\in\mathcal{G}_{G(l)}(r_\nu)\mbox{ and }\forall n\geq n_3.$$According to Lemma \ref{debemosescribirlo55}, for $n$ large enough, $G^+(t)\in \mathcal{G}_{G(l)}(r_\nu)$. So, for $n$ large enough, with probability one,
$$(G^+(t)\oplus B_{r-3\epsilon_n}[0])\ominus B_r[0]\subset (\mathcal{X}_n^{G^+}\oplus B_r[0])\ominus B_r[0]       =C_r(\mathcal{X}_n^{G^+}).$$
Since $G^+(t)$ is $(r-3\epsilon_n)-$convex because $r-3\epsilon_n\leq r_\nu$, it is satisfied that
$$(G^+(t)\oplus B_{r-3\epsilon_n}[0])\ominus B_r[0]=$$
$$=(G^+(t)\oplus  B_{r-3\epsilon_n}[0])\ominus( B_{r-3\epsilon_n}[0]\oplus B_{3\epsilon_n}[0])$$
$$=(G^+(t)\oplus  B_{r-3\epsilon_n}[0]\ominus B_{r-3\epsilon_n}[0])\ominus  B_{3\epsilon_n}[0]=G^+(t)\ominus  B_{3\epsilon_n}[0].$$
Therefore, since $r_\nu>r>0$,
$$\exists n_4\in\mathbb{N}\mbox{ such that }G^+(t)\ominus  B_{3\epsilon_n}[0]\subset C_{r}(\mathcal{X}_n^{G^+})\subset   C_{r_\nu}(\mathcal{X}_n^{G^+}),\mbox{ } \forall n\geq n_4.$$According the Lemma \ref{debemosescribirlo55},
$$\exists n_5\in\mathbb{N}\mbox{ such that }r_n(t)\geq r_\nu,\mbox{ } \forall n\geq n_5.$$
Then,
$$G^+(t)\ominus  B_{3\epsilon_n}[0]\subset C_{r_n(t)}(\mathcal{X}_n^{G^+}),\mbox{ } \forall n\geq \max\{n_4, n_5\}.$$
Therefore, since $\mathcal{X}_n^{G^+}\subset\mathcal{X}_n^+(t)\subset G(t) $ and $r_n(t)\leq r_0(t)$, it is verified
\begin{equation}\label{900}
G^+(t)\ominus   B_{3\epsilon_n}[0]\subset C_{r_n(t)}(\mathcal{X}_n^+(t))\subset C_{r_0(t)}(G(t))=G(t),\mbox{ }\forall n\geq \max\{n_4,n_5\}.
\end{equation}
Using (\ref{900}), with probability one and for $n$ large enough,
$$d_H(C_{r_n(t)}(\mathcal{X}_n^+(t)),G(t))\leq d_H(G^+(t)\ominus B_{3\epsilon_n}[0],G(t)).$$By the triangle inequality,
\small{\begin{equation}\label{juioplr}
    d_H(C_{r_n(t)}(\mathcal{X}_n^+(t)),G(t)) \leq d_H(G^+(t),G(t))+d_H(G^+(t),G^+(t)\ominus   B_{3\epsilon_n}[0]).
\end{equation}
}
Since $\lim_{n\rightarrow \infty}D_n=0$,
$$\exists n_6\in\mathbb{N}\mbox{ such that }2D_n<\min\left\{(m/2)c,\frac{\zeta}{2}\right\},\mbox{ }\forall n\geq n_6.$$
According to Proposition \ref{enPruebaTeorema3Walther4},$$G(t)\subset  G^+(t)\oplus B_{\frac{4}{m}D_n}[0],\mbox{ }\forall n\geq n_6.$$
Since $G^+(t)\subset G(t)$, $d_H(G^+(t),G(t))=O(D_n)$. On the other hand,
$$G^+(t)\ominus  B_{3\epsilon_n}[0]\subset G^+(t)\subset G^+(t)\oplus  B_{3\epsilon_n}[0].$$Therefore,
$d_H(G^+(t),G^+(t)\ominus  B_{3\epsilon_n}[0])=O(\epsilon_n)$. As consequence, using (\ref{juioplr})
\[ d_H(C_{r_n(t)}(\mathcal{X}_n^+(t)),G(t))=O(\max\{ D_n,\epsilon_n \}), \mbox{ almost surely.}  \]  \hfill $\Box$

\section{Appendix}\label{apendix}
Many proofs in Section \ref{mainresults} take into account mathematical aspects considered in Walther (1997). Next, we will summarize these theoretical results. In particular, Proposition \ref{enPruebaTeorema3Walther} can be obtained directly from proof of Theorem 3 in Walther (1997). It guarantees the existence of a compact set  $C$ where the convergence rate for the density kernel estimator is established.

\begin{proposition}\label{enPruebaTeorema3Walther}Under assumptions (A) and (K), there exists $\upsilon>0$ and a compact set $C$ verifying that $$G(l)\setminus \interior (G(u))\oplus  B_{\upsilon}[0]\subset U$$and$$G(l)\setminus \interior (G(u))\oplus B_{\frac{\upsilon}{2}}[0]\subset C.$$In addition,
\begin{equation*}
\sup_{C}|f_n-f| =O\left(  \left(  \frac{\log{n}}{n}\right)^{p/(d+2p)}\right),\mbox{ almost surely}.
\end{equation*}

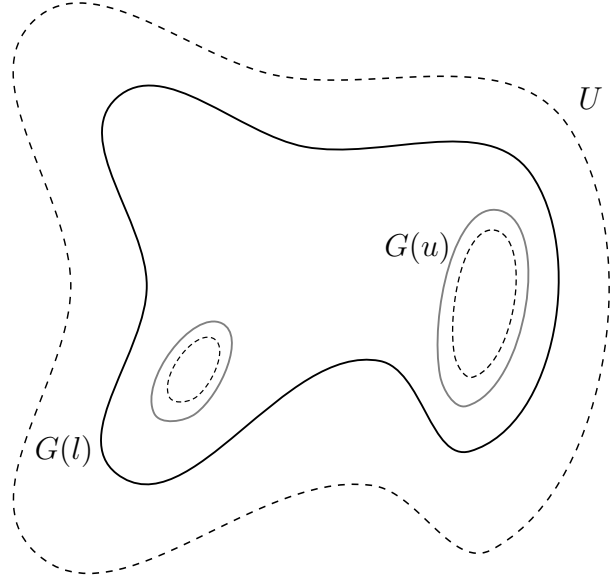
\begin{figure}[h!]\centering
\begin{pspicture}(-2,-1.95)(10,7)

\rput(-1.58,-1.){\scalebox{1.3}[1.4]{\psccurve[showpoints=false,fillstyle=solid,fillcolor=white,linecolor=black,linewidth=0.15mm,linearc=3,linestyle=dashed,dash=2pt 2pt](1.5,0)(1.9,2.5)(1.5,5)(4,4.5)(7,4)(6.2,-0.01)(5,0.6)}}
\psccurve[showpoints=false,fillstyle=solid,fillcolor=white,linecolor=black,linewidth=0.25mm,linearc=3](1.5,0)(1.9,2.5)(1.5,5)(4,4.35)(7,4)(6.2,.3)(5,1.5)

\psccurve[showpoints=false,fillstyle=solid,fillcolor=white,linecolor=gray,linewidth=0.25mm,linearc=3](5.8,1.3)(6.6,3.5)(6.15,0.9)

\rput(1.93,0.61){\scalebox{0.7}[0.75]{\psccurve[showpoints=false,fillstyle=solid,fillcolor=white,linecolor=black,linewidth=0.25mm,linearc=3,linestyle=dashed,dash=3pt 2pt](5.8,1.3)(6.6,3.5)(6.15,0.9) }}
\psccurve[showpoints=false,fillstyle=solid,fillcolor=white,linecolor=gray,linewidth=0.25mm,linearc=3](2,0.8)(2.9,2)(2.55,0.8)

\rput(0.9,0.5){\scalebox{0.65}[0.65]{\psccurve[showpoints=false,fillstyle=solid,fillcolor=white,linecolor=black,linewidth=0.25mm,linearc=3,linestyle=dashed,dash=3pt 2pt](2,0.8)(2.9,2)(2.55,0.8)  }}

\rput(5.5,3){$G(u)$}
\rput(7.8,5.00){$U$}
\rput(0.8,0.3){$G(l)$}
\end{pspicture}\vspace{-.3cm}\caption{Elements in Proposition \ref{enPruebaTeorema3Walther}. $G(u)$ in gray, $G(l)$ in black and the open set $U$ in dashed line.}\label{oookffghhffffffffffffffffffffffhg1o}
\end{figure}
\end{proposition}

\begin{proof}According to the proof of Theorem 3 in Walther (1997), one can find $\upsilon>0$ such that $ G(l)\setminus \interior (G(u))\oplus B_{ \upsilon}[0]\subset U$, see Figure \ref{oookffghhffffffffffffffffffffffhg1o}. In addition, the kernel $K$ satisfies the assumptions in Theorem 3.1 in Stute (1984). Following Walther (1997), one can prove that there exists a compact set $C$ such that $G(l)\setminus \interior (G(u))\oplus  B_{\upsilon/2}[0]\subset C$ and such that if $h_n$ is a sequence of the order $(\log{n}/n)^{1/(d+2p)}$ then it is verified that
\begin{equation}\label{13} \sup_{C}|f_n-K\ast f|=O\left(n^{-p/(d+2p)}\right),\mbox{ almost surely},\end{equation}
\begin{equation}\label{14} \sup_{C}|K\ast f-f|=O\left(h_n^p\right),\end{equation}where we write $K\ast f$ for $\int h_n^{-d}K((\cdot-x)/h_n)f(x)\mbox{ } dx$.
Equations (\ref{13}) and (\ref{14}) correspond to equations (13) and (14) in Walther (1997). By the triangle inequality, we can guarantee that, almost surely,
$$\sup_{C}|f_n-f|=O\left( \max\left\{ \left(\frac{\log{n}}{n}\right)^{p/(d+2p)},n^{-p/(d+2p)}\right \} \right)$$
\[=O\left(  \left(  \frac{\log{n}}{n}\right)^{p/(d+2p)}\right) .\qedhere\]\end{proof}

Proposition \ref{enPruebaTeorema3Walther3} corresponds to equation (15) in Walther (1997). The behavior of the kernel density estimator is studied in the complement of the compact set $C$.

\begin{proposition}\label{enPruebaTeorema3Walther3}Let $C$ the compact set in Proposition \ref{enPruebaTeorema3Walther}. Under assumptions (A) and (K), there exists $w>0$ verifying that
\begin{equation*}
 \mathbb{P}\left( \inf_{G(u)\cap C^c} f_n(x)>u+\frac{w}{2},\mbox{ eventually}\right)=1
\end{equation*}and\begin{equation*}
\mathbb{P}\left( \sup_{G(l)^c \cap C^c } f_n(x)<l-\frac{w}{2},\mbox{ eventually} \right)=1.
\end{equation*}

\end{proposition}

Proposition \ref{enPruebaTeorema3Walther4} corresponds to Lemma 2 (b) in Walther (1997). It establishes some interesting relationships between level sets with close enough thresholds.

\begin{proposition}\label{enPruebaTeorema3Walther4}Under assumption (A), there exists a constant $c>0$ such that if $t_1$ and $t_2$ are such that $l\leq t_1<t_2\leq u$ and $t_2-t_1\leq (m/2)c$ then
\begin{equation*}
 G(t_1)\ominus B_{\frac{2}{m}(t_2-t_1)}[0]\subset G(t_2)\subset G(t_1)\ominus B_{\frac{1}{2m}(t_2-t_1)}[0]\end{equation*}and
\begin{equation*}
 G(t_2)\oplus B_{\frac{1}{2m}(t_2-t_1)}[0]\subset G(t_1)\subset G(t_2)\oplus B_{\frac{2}{m}(t_2-t_1)}[0].\end{equation*}
\end{proposition}

Finally, Proposition \ref{enPruebaTeorema3Walther5} is presented. It corresponds to Lemma 3 in Walther (1997).

\begin{proposition}\label{enPruebaTeorema3Walther5}Let $K\subset\mathbb{R}^d$ be a compact set, $r > 0$ and let $\mathcal{X}_n$ be a i.i.d. sample generated from a distribution with density function $f$. Let $ \mathcal{G}_K(r)$ be the family of sets defined in Definition \ref{waltherG}.
\begin{enumerate}
  \item If $f \geq b > 0$ on $A\in \mathcal{G}_K(r)$ and $0 < \epsilon< \min\{\overline{r}/2 , r\}$ then
$$
\mathbb{P}\left(A\oplus B_{\overline{r}-2\epsilon}[0]\nsubset (A\cap\mathcal{X}_n)\oplus B_{\overline{r}}[0]\right)$$
$$\leq
D\left(\epsilon,A\oplus B_{\overline{r}}[0]\right) \exp{ \left( -nab \min\{\overline{r}-\epsilon,r\}^{(d-1)/2}\epsilon^{(d+1)/2} \right) }.
$$where$$D\left(\epsilon,A\oplus B_{\overline{r}}[0]\right)= \max\{card\mbox{ } V:V\subset A\oplus B_{\overline{r}}[0],\mbox{ }|x-y|>\epsilon\mbox{ for different }x,y\in N\}$$and $a$ is a dimensional constant.\vspace{.3mm}
  \item Further, if $f \geq b > 0$ on $K$, $0 < \epsilon< \min\{\overline{r}/3 , 1\}$ and $r\geq \overline{r}-2\epsilon$ then
$$
\mathbb{P}\left(A\oplus B_{\overline{r}-3\epsilon}[0]\nsubset (A\cap\mathcal{X}_n)\oplus B_{\overline{r}}[0] \mbox{ for some }A\in \mathcal{G}_K(r)\right)$$
$$\leq
D\left(\epsilon,K\oplus B_{\overline{r}}[0]\right)D\left(\frac{\epsilon}{10\overline{r}},S^{d-1}\right)\exp{\left(-nab(\overline{r}-2\epsilon)^{(d-1)/2}(\epsilon/2)^{(d+1)/2}\right)}
$$where $S^{d-1}$ denotes the unit sphere.
\end{enumerate}

\end{proposition}
$ $\\
\textbf{Funding.} This work has been supported by Project MTM2008-03010 of the Spanish Ministry of Science and Innovation and the IAP network StUDyS (Developing crucial Statistical methods for
Understanding major complex Dynamic Systems in natural, biomedical and social sciences) of the
Belgian Science Policy.\\

 $ $\\

\end{document}